\documentclass[12pt]{amsart}
\usepackage{enumerate,amssymb,amsfonts,latexsym,psfrag}
\usepackage{amscd,amsmath}
\usepackage[dvips]{graphicx,epsfig}

\usepackage{url}
\usepackage{graphics}
\usepackage[all]{xy}
\usepackage{fullpage}
\usepackage{subfig}
\usepackage{multirow,array}
\usepackage{color}
\usepackage{cancel}
\usepackage{color}

\usepackage{titlesec}

\titleformat{\section}
  {\Large\bfseries\center}{\thesection.}{0.5em}{}
  \titlespacing*{\section} {0pt}{1em}{2.3ex plus .2ex}
\titlelabel{\thetitle.\quad}
\titleformat{\subsection}
  [runin]{\normalfont\bfseries}{\thesubsection.}{0.5em}{}[.]
\titlespacing*{\subsection}{0pt}{1em}{2.3ex plus .2ex}

\usepackage{titletoc}
\titlecontents{section}[1.5em]
  {\normalfont\bfseries}
  {\contentslabel{1.5em}}
  {\hspace*{-2.3em}}
 {\titlerule*[1pc]{}\contentspage}

\newtheorem{thm}{Theorem}[section]
\newtheorem{cor}[thm]{Corollary}
\newtheorem{lem}[thm]{Lemma}
\newtheorem{prop}[thm]{Proposition}
\theoremstyle{remark}
\newtheorem{rem}{Remark}[section]

\theoremstyle{definition}
\newtheorem{defn}{Definition}[section]

\newtheorem{example}[thm]{Example}

\numberwithin{equation}{section}
\numberwithin{figure}{section}

\font\nt=cmr7

\def\note#1
{\marginpar
{\nt $\leftarrow$
\par
\hfuzz=20pt \hbadness=9000 \hyphenpenalty=-100 \exhyphenpenalty=-100
\pretolerance=-1 \tolerance=9999 \doublehyphendemerits=-100000
\finalhyphendemerits=-100000 \baselineskip=6pt
#1}\hfuzz=1pt}

\newcommand{\di}{\partial}

\newcommand{\ra}{\rightarrow}

\def\lra{\longrightarrow}

\def\ssk{\smallskip}
\def\msk{\medskip}
\def\bsk{\bigskip}

\def\nin{\noindent}

\newcommand{\diam}{\operatorname{diam}}
\newcommand{\dist}{\operatorname{dist}}

\newcommand{\id}{\operatorname{id}}

\newcommand{\Id}{\operatorname{Id}}

\newcommand{\const}{\mathrm{const}}

\newcommand{\eps}{{\varepsilon}}

\newcommand{\de}{{\delta}}
\newcommand{\la}{{\lambda}}
\newcommand{\La}{{\Lambda}}
\newcommand{\si}{{\sigma}}

\newcommand{\bde}{{\boldsymbol{\delta}}}

\newcommand{\oeps}{{\overline\eps}}

\newcommand{\CC}{{\mathcal C}}

\newcommand{\II}{{\mathcal I}}

\newcommand{\NN}{{\mathcal N}}
\newcommand{\OO}{{\mathcal O}}

\newcommand{\TT}{{\mathcal T}}

\newcommand{\N}{{\mathbb N}}

\newcommand{\R}{{\mathbb R}}

\newcommand{\vv}{{\mathbf v}}
\newcommand{\cc}{{\mathbf c}}
\newcommand{\ww}{{\mathbf w}}

\def\Bz{{\mathbf{z}}}

\def\Bd{{\mathbf{d}}}
\def\Bu{{\mathbf{u}}}

\def\Bq{{\mathbf{q}}}
\def\Bp{{\mathbf{p}}}

\def\Br{{\mathbf{r}}}

\def\BPhi{{\boldsymbol{\BPhi}}}
\def\B0{{\mathbf{0}}}

\def\BR{{\mathbf{R}}}

\newcommand{\Jac}{\operatorname{Jac}}

\newcommand{\Dom}{\operatorname{Dom}}

\catcode`\@=12

\def\Empty{}
\newcommand\oplabel[1]{
  \def\OpArg{#1} \ifx \OpArg\Empty {} \else
  	\label{#1}
  \fi}
		
\newcommand{\comm}[1]{}
\newcommand{\comment}[1]{}

\allowdisplaybreaks[4]

\begin{document}

\bigskip\bigskip

\title[H\'enon renormalization]{H\'enon renormalization in arbitrary dimension : Invariant space under renormalization operator}
\author{Young Woo Nam}

\address {Young Woo Nam \\ \quad e-mail : namyoungwoo@hongik.ac.kr \quad ellipse7\,@\ daum.net}


\date{June 23, 2015}

\begin{abstract}
Infinitely renormalizable H\'enon-like map in arbitrary finite dimension 
is considered. The set, $ \NN $ of infinitely renormalizable H\'enon-like maps satisfying the certain condition is invariant under renormalization operator. The Cantor attractor of infinitely renormalizable H\'enon-like map, $ F $ in $ \NN $ has {\em unbounded geometry} almost everywhere in the parameter space of the universal number which corresponds to the average Jacobian of two dimensional map. This is the extension of the same result in $ \NN $ for three dimensional infinitely renormalizable H\'enon-like maps.
\end{abstract}

\maketitle

\thispagestyle{empty}

\setcounter{tocdepth}{1}
\tableofcontents

\renewcommand{\labelenumi} {\rm {(}\arabic{enumi}{)}}
\renewcommand{\CancelColor}{ \color{blue} }

\section{Introduction}

The universality of one dimensional dynamical system was discovered by Feigenbaum and independently by Coullet and Tresser in the mid 1970's and universality of higher dimensional maps is conjectured by Coullet and Tresser in \cite{CT}. 
Hyperbolicity at the fixed point of renormalization operator is finally proved by Lyubich in \cite{Lyu} using quadratic-like maps in one dimensional holomorphic dynamical systems. 
The similar universal properties are expected in higher dimensional maps which are strongly dissipative and close to the one dimensional maps. In particular, renormalizable maps with {\em periodic doubling type} are interesting in two or higher dimension. Universality of two dimensional strongly dissipative infinitely renormalizable H\'enon-like maps is justified in \cite{CLM}. 
Cantor attractor of two dimensional H\'enon-like maps is the counterpart of that of one dimensional maps but it has different small scale geometric properties. In particular, it has non rigidity and unbounded geometry. These geometric properties are generalized in special classes of the highly dissipative three dimensional H\'enon-like family in \cite{Nam2, Nam3}.

\msk

\subsection{H\'enon maps and bifurcation of the homoclinic tangency}
H\'enon map is a polynomial diffeomorphism from $ \R^2 $ to itself as follows
$$ H_{a,\,b}(x,y) = (1- ax^2 + y,\ bx) . $$
A famous conjecture is the existence of {\em strange attractor} at the parameter $ a=1.4 $ and $ b=0.3 $. The first significant achievement about the H\'enon map with parameter space $ (a,b) $ was done by Benedics and Carleson in \cite{BC}. There exist strange attractors for the positive density of the parameter space, $ (a,b) $ such that $ a_0 < a < 2 $ and $ b < b_0 $ where $ a_0 $ is close to $ 2 $ and $ b_0 $ is small. This parameter values which are considered in \cite{BC} is a generalization of one dimensional Misiuriewicz maps in \cite{Jak}. Jakobson proved that the maps $ 1- ax^2 $ which have absolutely continuous invariant measure with respect to Lebegue measure has positive density on the parameter space. So is in the H\'enon family in \cite{BC}. Wang and Young used the certain geometric conditions to generalize H\'enon family in \cite{WY1} with statistical properties. 
Furthermore, it is generalized to arbitrary finite dimension in \cite{WY2} with the rank one attractor, that is, attractor with the neutral or repulsive direction is one dimensional. 
\ssk \\ 
\comm{******************
\nin The hyperbolic systems have been studied from 1960s. Moreover, those systems were expected to be generic in the whole dynamical systems. This generic hyperbolicity is the main conjecture of the rational maps on the Riemann sphere. In the quadratic polynomial case, it is known that MLC (Mandelbrot set is locally connected) conjecture is equivalent to the generic hyperbolicity. However, in the dynamical systems of two or higher dimensional maps Newhouse proved that the maps in the chaotic region contains an open set. The proof used a perturbation of homoclinic tangency with certain condition of the invariant Cantor set. After Newhouse proof, Palis suggested a new conjecture about the generic dynamical system at p.134 in \cite{PT}. 
\begin{conj}[Palis conjecture]
Every $ C^r $ diffeomorphism in Diff $(M)$ for $ r \geq 1 $ can be approximated by a hyperbolic diffeomorphism or else by one exhibiting a homoclinic bifurcation involving a homoclinic tangency or a cycle of hyperbolic periodic saddles with different indices.
\end{conj}
***************}
\nin Let us consider a homoclinic tangency of two dimensional maps. Then the dimension of both unstable and stable manifold at the homoclinic point is one. After bifurcation of homoclinic tangency, stable and unstable manifold are (transversally) intersected around the homoclinic point. Let us choose a bounded region on which this bifurcation occurs and consider the first return map, $ H $. Then after appropriate smooth coordinate change, the image of horizontal lines in the bounded region is the vertical lines. Then we obtain that the simplest non-linear map which satisfies the above properties, which is called H\'enon map 
$$ H_{a,\: b}(x,y) = (x^2 - a + by,\ x) . $$
However, in general the first coordinate map of the first return map is not a polynomial but is a perturbation of a unimodal one dimensional map, say $ f(x) $. Then the first return map is of the following form
$$ F(x,y) = (f(x) - \eps(x,y),\ x) $$
where $ f(x) $ is a unimodal map. We call this map {\em H\'enon-like map}. 
Moreover, a perturbation of homoclinic tangency could occur in higher dimension. 
In order to make that the first return map has H\'enon-like form in the first two coordinates, let us assume that dimension of unstable manifold at the homoclinic point is one. Then the first return map in higher dimension has first two coordinates similar to the two dimensional H\'enon-like map after smooth coordinate change which flattens the stable manifold at the homoclinic point in the bounded region 
$$ F(x,y,z_1,z_2,\ldots,z_m) = (f(x) - \eps(x,y,z_1,z_2,\ldots,z_m),\ x,\ \bullet \,) $$
where $ f(x) $ is a unimodal map. If the  maximal backward invariant set has only one dimensional neutral or unstable direction, then it is called {\em rank one attractor}. This viewpoint is reflected in the paper of Wang and Young, \cite{WY2} for the maps in higher dimension on the chaotic region in the parameter space. 

\msk
\subsection{Statement of result}
Period doubling renormalization of analytic H\'enon-like maps was constructed on \cite{CLM} by de Carvalho, Lyubich and Martens. The set of renormalizable H\'enon-like maps is a counterpart of the set of H\'enon maps in chaotic region. 
Moreover, we expect the map $ F $ has no homoclinic tangency between the stable and unstable manifolds of the regular fixed point. 
\comm{********************
H\'enon map is a polynomial diffeomorphism in two dimension
$$ F(x,y) = (1-ax^2 + by,\ x) $$ 
for $ a $ and $ b $ are real numbers. If $ b $ is a small enough positive number, then $ F $ is a kind of a small perturbation of one dimensional quadratic polynomial $ x \mapsto 1 -ax^2 $. 
If the dynamical system has the perturbation of homoclinic tangency, then the first return map can be H\'enon-like map after appropriate smooth coordinate change as follows
\begin{equation*}
F(x,y) = (f(x) - \eps(x,y),\ x).
\end{equation*}
H\'enon renormalization was introduced in \cite{CLM} with universal limit of infinitely renormalizable H\'enon-like maps. 
********************************}
H\'enon renormalization is extended to three dimensional H\'enon-like maps in \cite{Nam1}. Geometric property of the Cantor attractor, for example, {\em unbounded geometry} of infinitely renormalizable dimensional H\'enon-like maps was extened in certain invariant space of three dimensional maps under renormalization operator in \cite{Nam2, Nam3}. A particular set of three dimensional H\'enon-like maps, say $ \NN $ in which every H\'enon-like map, say $ F $, satisfies the following condition \ssk
\begin{equation} \label{eq-invariant space NN in three dimension}
\begin{aligned}
\di_y \de \circ (F(x,y,z)) + \di_z \de \circ (F(x,y,z)) \cdot \di_x \de(x,y,z) \equiv 0 .
\end{aligned} \ssk
\end{equation}
where $ F(x,y,z) = (f(x) - \eps(x,y,z),\ x,\ \de(x,y,z)) $. Denote the set of infinitely renormalizable H\'enon-like maps by $ \II_B(\oeps) $ where $ B $ is the domain of H\'enon-like maps. Then $ \NN \cap \II_B(\oeps) $ was found as an invariant space under renormalization operator in \cite{Nam2}.
 \ssk \\
H\'enon renormalization of period doubling type in arbitrary finite dimension is defined in \cite{Nam4} and a space invariant under renormalization operator. H\'enon-like map in higher dimension is defined as follows \msk
\begin{equation*}
\begin{aligned}
F(x,y,\Bz) = (f(x) - \eps(x,y,\Bz),\ x,\ \bde(x,y,\Bz))
\end{aligned} \msk
\end{equation*}
where $ \Bz = (z_1, z_2,\ldots,z_m) $ and $ \bde = (\de^1, \de^2, \de^3, \ldots , \de^m) $ is a map from the $ m+2 $ dimensional hypercube $ B $ to $ \R^m $ for $ m \geq 1 $. H\'enon-like map is generalized in higher dimension as a model of the first return map of perturbation of the {homoclinic tangency} of {\em sectional dissipative map} at fixed points. In this paper, we find the invariant space which is the extension of the space $ \NN \cap \II_B(\oeps) $\ in arbitrary finite dimension and prove the unnbounded geometry of Cantor attractor of each element in $ \NN \cap \II_B(\oeps) $. Many parts of this paper are self contained although most ideas and notations are the same as those in \cite{Nam2}. 
\msk
\begin{defn}[Definition \ref{def-definition of NN}]
Let $ \NN $ be the set of $ m+2 $ dimensional renormalizable H\'enon-like maps such that each map $ F \in \NN $ satisfies the following equation
$$ \di_y \de^j \circ F(w) + \sum_{l=1}^m \di_{z_l} \de^j \circ F(w) \cdot \di_x \de^l(w) = 0 $$
where $ w = (x,y,\Bz) \in \Dom(F) $ and $ F(w) = \big(\,f(x) - \eps(w),\ x,\ \de^1(w),\,\de^2(w),\,\ldots,\, \de^m(w) \,\big) $ for all\, $ 1 \leq j \leq m $.
\end{defn}

\begin{thm}[Theorem \ref{thm - invariant space NN}]
The space $ \NN \cap \II_B(\oeps) $ is invariant under renormalization, that is, if $ F \in \NN \cap \II_B(\oeps) $, then $ RF \in \NN \cap \II_B(\oeps) $.
\end{thm}
\msk
\nin H\'enon-like maps in the above space has generic unbounded geometry with the universal number $ b_1 $ corresponding the average Jacobian of two dimensional H\'enon-like map.

\begin{thm}[Theorem \ref{unbounded geometry with b-1}]
Let $ F_{b_1} $ be an element of parametrized space in $ \NN \cap \II_B(\oeps) $ with $ b_1 = b_F/b_{\Bz} $ where $ b_F $ is the average Jacobian of $ F $ and $ b_{\Bz} $ is the number defined in Section \ref{subsec-universal number b-1}. Then there exists a small interval $ [0, b_{\bullet}] $ for which there exists a $ G_{\de} $ subset $ S \subset [0, b_{\bullet}] $ with full Lebesgue measure such that the critical Cantor set, $ \OO_{F_{b_1}} $  has unbounded geometry for all $ b_1 \in S $.
\end{thm}
\msk
\nin There exists non rigidity between Cantor attractors of H\'enon-like maps.

\begin{thm}[Theorem \ref{Non rigidity with b-1}]
Let H\'enon-like maps $ F $ and $ \widetilde{F} $ be in $ \NN \cap \II_B(\bar \eps) $. 
Let $ b_1 $ be the ratio $ b_F/b_{\Bz} $ where $ b_F $ is the average Jacobian and $ b_{\Bz} $ for $ F $ is the number defined in Proposition \ref{prop-universal number b-Z}. The number $ \widetilde{b}_1 $ is defined by the similar way for the map $ \widetilde{F} $. Let $ \phi \colon \OO_{\widetilde{F}} \ra \OO_F $ be a homeomorphism \ssk which conjugate $ F_{\,\OO_F} $ and $ \widetilde{F}_{\,\OO_{\widetilde{F}}} $ and $ \phi(\tau_{\widetilde{F}}) = \tau_{F} $. If \,$ b_1 > \widetilde{b}_1 $, \ssk then the H\"older exponent of $ \phi $ is not greater than $ \displaystyle{\frac{1}{2} \left(1 + \frac{\log b_1}{ \log \widetilde{b}_1} \right)} $.
\end{thm}

\bsk

\section{Preliminaries}
Let $ m+2 $ dimensional H\'enon-like map be the map from the $ m+2 $ dimensional hypercube, $ B $ to $ \R^{m+2} $ which is of the following form
\begin{equation} \label{eq-Henon-like map in m+2 dimension}
F(x,y,\Bz) = (f(x) - \eps(x,y,\Bz),\ x,\ \bde(x,y,\Bz))
\end{equation}
where $ \Bz = (z_1, z_2, \ldots , z_m) $, $ f $ is a unimodal map on the close interval $ \pi_x(B) $, $ \eps $ and $ \bde $ is a map from $ B $ to $ \R $ and $ \R^m $ respectively. In addition to the above definition we assume that $ F $ is an orientation preserving analytic map and both $ \| \;\! \eps \| $ and $ \| \;\! \bde \| $ bounded above by $ C\oeps $ for sufficiently small positive $ \oeps $  and for some $ C>0 $. Then the H\'enon-like map in arbitrary finite dimension is the simplest model as the first return map of the homoclinic bifurcation with one dimensional unstable direction and sectional dissipative map with strong contractable directions. 
\\ \ssk
Let $ f $ be the unimodal map on the closed interval $ I $ which contains the critical point and the critical value such that $ f(I) \subset I $. The map $ f $ is called {\em renormalizable} with periodic doubling type if there exists a closed subinterval $ J \neq I $ which contains the critical point, $ c_f $ of $ f $ and $ J $ is (forward) invariant under $ f^2 $ and $ f^2(c_f) \in \di J $. Thus if $ f $ is renormalizable, then we can choose the smallest interval $ J_f $ satisfying the above properties. The conjugation of the appropriate affine conjugation rescales $ J_f $ to $ I $ and it defines the renormalization of $ f $, $ Rf \colon I \ra I $. For the renormalization of smooth map on the closed interval, for example, see \cite{BB}. If $ f $ is infinitely renormalizable, then there exists the renormalization fixed point, $ f $ under the renormalization operator. Moreover, the {\em scaling factor} of $ f_* $ is
\begin{equation*}
\si = \frac{| J_{f_*} |}{| \:\! I |}
\end{equation*}
and $ \lambda = 1/\si = 1/2.6\ldots $\,. If the unimodal map $ f $ has two fixed point on the interval $ I \equiv [-c,c] $ for some $ c>0 $ and both $ \| \:\!\eps \| $ and $ \| \:\! \bde \| $ is small enough, then the H\'enon-like map $ F $ defined at \eqref{eq-Henon-like map in m+2 dimension} has only two fixed points on the hypercube $ [-c,c]^{m+2} $. Let $ \beta_0 $ be one fixed point which has positive eigenvalues and $ \beta_1 $ be the other fixed point which has both positive and negative eigenvalues. We call that H\'enon-like map $ F $ is {\em renormalizable} if $ W^u(\beta_0) \cap W^s(\beta_1) $ is the orbit of a single point. Thus if $ F $ is renormalizable, then there exist invariant regions under $ F^2 $. However, the image of the hyperplane $ \{ \,x= \const . \,\} $ under $ F^2 $ is the surface
$$ f(x) - \eps(x,y,\Bz) = \const . $$
which is not a hyperplane. Then $ F^2 $ is not H\'enon-like map. Thus the analytic definition of the renormalization of $ F $ requires the non linear scaling map. The {\em horizontal-like} diffeomorphism $ H $ of $ F $ is defined as follows
\begin{equation*}
H(x,y,\Bz) = (f(x) - \eps(x,y,\Bz),\ y,\ \Bz - \bde(y, f^{-1}(y), \B0))
\end{equation*}
and it preserves the hyperplane $ \{ \, y= \const . \,\} $. Then the {\em renormalization} of $ F $, $ RF $ is defined
\begin{equation*}
RF = \La \circ H \circ F^2 \circ H^{-1} \circ \La^{-1}
\end{equation*}
where the dilation $ \La(x,y,z) = (sx, sy, sz) $ for the appropriate constant $ s < -1 $. If $ F $ is $ n $ times renormalizable, then each $ R^kF $ for $ 2 \leq k \leq n $ is defined as the renormalization of $ R^{k-1}F $ successively. If $ n $ is unbounded, then we call $ F $ is {\em infinitely renormalizable}. 
\ssk \\
Let $ B $ the hypercube as the domain of H\'enon-like map. If it is emphasized with relation of particular map, for instance, $ F $ or $ R^kF $, then we express this region as $ B(F) $ or $ B(R^kF) $ and so on. Any other set may have similar expression with a particular map. In this paper, we assume that any given H\'enon-like map is analytic unless any other statement is specified. Moreover, the norm of $ \eps $ and $ \bde $ is always $ O(\oeps) $ for some sufficiently small $ \oeps > 0 $. Let the set of infinitely renormalizable H\'enon-like map be $ \II_B(\oeps) $. Then the renormalization operator has the fixed point in $ \II_B(\oeps) $. The fixed point $ F_* $ is a degenerate map as follows
\begin{equation*}
F_*(x,y,\Bz) = (f_*(x),\ x,\ \B0)
\end{equation*}
where $ f_* $ is the fixed point of the renormalization operator of the (analytic) unimodal map. Furthermore, the renormalization $ R^nF $ converges to $ F_* $ as $ n \ra \infty $ exponentially fast and $ F_* $, the map $ F_* $ is the hyperbolic fixed point under renormalization operator and the stable manifold at $ F_* $ is codimension one. All of above properties is valid for infinitely renormalizable  H\'enon-like maps in any two or greater dimension. See \cite{CLM, Nam1, Nam3}. 
\ssk \\
For $ F \in \II_B(\oeps) $, $ F_k $ denotes $ R^kF $ for each $ k \in \N $. Let the non linear scaling map which conjugates $ F_k^2|_{\La_k^{-1}}(B) 0 $ to $ RF_k $ is
\begin{equation*}
\psi^{k+1}_v \equiv H_k^{-1} \circ \La_k^{-1} \colon \Dom(RF_k) \ra \La_k^{-1}(B)
\end{equation*}
where $ H_k $ is the horizontal-like diffeomorphism and $ \La_k $ is the dilation with the appropriate constant $ s < -1 $. Let $ \psi^{k+1}_c $ be $ F_k \circ \psi^{k+1}_v $ for each $ k \in \N $. For $ k < n $, we express the consecutive compositions of $ \psi^{k+1}_v $ or $ \psi^{k+1}_c $ as follows \msk
\begin{equation} \label{eq-psi-v,c composition}
\begin{aligned}
\Psi^n_{k,\,{\bf v}} &= \psi^{k+1}_v \circ \psi^{k+2}_v \circ \cdots \circ \psi^n_v \\[0.3em]
\Psi^n_{k,\,{\bf c}} &= \psi^{k+1}_c \circ \psi^{k+2}_c \circ \cdots \circ \psi^n_c
\end{aligned} \msk
\end{equation}
Let $ w $ be a letter which is $ v $ or $ c $. Let the word of length $ n $ in the Cartesian product, $ W^n \equiv \{\,v,c \,\}^n $ be $ {\bf w}_n $ or simply be $ {\bf w} $. Thus the map $ \Psi^n_{k,\,{\bf w}} $ is defined with the word $ {\bf w} $ of which length is $ n-k $. The region $ B^n_{\bf w} $ denotes $ \Psi^n_{k,\,{\bf w}}(B(R^nF)) $. If $ F $ is in the set $ \II_B(\oeps) $ then it has the minimal invariant Cantor attractor
\begin{equation*}
\OO_F = \bigcap_{n=1}^{\infty} \bigcup_{{\bf w} \in W^n} B^n_{{\bf w}}
\end{equation*}
and $ F $ acts as a dyadic adding machine on $ \OO_F $. The counterpart of the critical value of the unimodal renormalizable map is called the {\em tip}
\begin{equation*}
\{ \tau_F \} \equiv \bigcap_{n \geq 0}B^n_{v^n}
\end{equation*}
Let $ \mu $ be the unique invariant probability measure on $ \OO_F $. The average Jacobian of $ F $, $ b_F $ is defined as
\begin{equation*}
b_F = \exp \int_{\OO_F} \log \Jac F \,d\mu .
\end{equation*}
Then there exists the asymptotic universal expression of Jacobian determinant of $ R^nF $ with the average Jacobian $ b_F $ (Theorem 5.10 in \cite{Nam4}). 
\begin{equation*}
\Jac R^nF = b_F^{2^n}a(x)\,(1 + O(\rho^n))
\end{equation*}
where $ a(x) $ is the universal positive function and $ \rho $ is a number in $ (0,1) $. Let $ \tau_n $ be the tip of $ R^nF $ for each $ n \in \N $. The definition of the tip and $ \Psi^n_{k\,{\bf v}} $ implies that $ \Psi^n_{k,\,{\bf v}}(\tau_n) = \tau_k $ for $ k < n $. After composing appropriate translations, tips on each level moves to the origin as the fixed point of the following map
\begin{equation*}
\Psi^n_k(w) = \Psi^n_{k,\,{\bf v}}(w + \tau_n) - \tau_k
\end{equation*}
for $ k < n $. If $ n=k+1 $, then $ \Psi^{k+1}_k $ is separated to the linear and non linear parts as follows \msk
\begin{equation*}
\begin{aligned}
& \quad \Psi^{k+1}_k(w) \\[0.3em]
&=
 \left(    \begin{array} {cccccc}
\alpha_k & \si_k t_k   & \si_k u_k^1 & \cdots & \si_k u_k^m  &   \\[0.2em]
         & \si_k       &             &        &              &   \\[0.2em] \cline{3-5}
         & \multicolumn{1}{c|}{\si_k d_k^1} & \multicolumn{3}{c}{\multirow{3}{*}{ $ \si_k \cdot \Id_{m \times m}  $ }} & \multicolumn{1}{|c}{}\\
         & \multicolumn{1}{c|}{\vdots}      &  &  &  & \multicolumn{1}{|c}{} \\ 
         & \multicolumn{1}{c|}{\si_k d_k^m} &  &  &  & \multicolumn{1}{|c}{} \\       \cline{3-5}
    \end{array}       \right) 
\begin{pmatrix}
x + s_k(w) \\
y \\[0.2em]
z_1 + r_k^1(y) \\
\vdots \\
z_m + r_k^m(y)
\end{pmatrix} = 
\begin{pmatrix}
\alpha_k & \si_k \;\! t_k & \si_k \;\! \Bu_k     \\
         & \si_k          &      \\
         & \si_k \;\! \Bd_k    & \si_k
\end{pmatrix}
\begin{pmatrix}
x + s_k(w) \\
y \\
\Bz + \Br_k(y) 
\end{pmatrix}
\end{aligned} \msk
\end{equation*}
where $ w =(x,y,\Bz) $ and the boldfaced letters are the vectors each of which has $ m $ coordinates. The map $ \Psi^n_k $ is also separated to the linear and non linear parts after reshuffling. \msk
\begin{equation*}
\begin{aligned}
\Psi^n_k(w) = 
\begin{pmatrix}
1 &  t_{n,\,k} &  \Bu_{n,\,k}     \\[0.2em]
         & 1          &      \\[0.2em]
         &  \Bd_{n,\,k}   & 1
\end{pmatrix}
\begin{pmatrix}
\alpha_{n,\,k} & & \\[0.2em]
& \si_{n,\,k} & \\[0.2em]
& & \si_{n,\,k} \cdot \Id_{m \times m}
\end{pmatrix}
\begin{pmatrix}
x + S^n_k(w) \\
y \\[0.2em]
\Bz + \BR_{n,\,k}(y) 
\end{pmatrix}
\end{aligned} \msk
\end{equation*}
where $ \alpha_{n,\,k} = \si^{2(n-k)}(1 + O(\rho^n)) $ and $ \si_{n,\,k} = \si^{n-k}(1 + O(\rho^n)) $. In this paper, we often confuse the map $ \Psi^n_{k,\,{\bf v}} $ with $ \Psi^n_k $ in order to obtain the simpler expression of each coordinate map of $ \Psi^n_{k,\,{\bf v}} $. For example, the expression of the first coordinate map of $ \Psi^n_{k,\,{\bf v}} $
\begin{equation*}
\alpha_{n,\,k} (x + S^n_k(w)) + \si_{n,\,k} t_{n,\,k}y + \si_{n,\,k} \Bu_{n,\,k} \cdot (\Bz + \BR_{n,\,k}(y)) 
\end{equation*}
means that
\begin{equation*}
\begin{aligned}
\alpha_{n,\,k} \big[\,(x + \tau^x_n) + S^n_k(w + \tau_n)\,\big] + \si_{n,\,k} t_{n,\,k} (y + \tau^y_n) + \si_{n,\,k} \Bu_{n,\,k} \cdot \big[\,(\Bz + \tau^{\Bz}_n) + \BR_{n,\,k}(y + \tau^y_n)\,\big] - \tau_k 
\end{aligned} \msk
\end{equation*}
where $ \tau_n = (\tau^x_n, \tau^y_n, \tau^{z_1}_n, \ldots, \tau^{z_m}_n)  $ and $  \tau^{\Bz}_n = (\tau^{z_1}_n, \ldots, \tau^{z_m}_n) $ for $ k < n $. Recall the following definitions related to the infinitely renormalizable H\'enon-like map $ F $ for later use \msk
\begin{equation*}
\begin{aligned}
\La_n^{-1}(w) = \si_n \cdot w, \qquad \psi^{n+1}_v(w) = H_n^{-1}(\si_n w), \qquad \psi^{n+1}_c(w) = F_n \circ H_n^{-1}(\si_n w) \\[0.3em]
\psi^{n+1}_v(B(R^{n+1}F)) = B^{n+1}_v, \qquad \psi^{n+1}_c(B(R^{n+1}F)) = B^{n+1}_c
\end{aligned} \msk
\end{equation*}
for each $ n \in \N $.


\bsk 
\section{An invariant space under renormalization}
Let $ F $ be a H\'enon-like map defined on the $ m + 2 $ dimensional hypercube, $ B $ as follows
$$ F(w) = (f(x) - \eps(w),\ x,\ \bde(w)) $$
where $ w = (x,y,z_1, \ldots ,z_m) $ and $ \bde \colon B \ra \pi_{\Bz}(B) $. The map $ \bde(w) = (\de^1(w), \de^2(w), \ldots, \de^m(w)) $ where $ \de^j \colon B \ra \R $ for $ 1 \leq j \leq m $. Suppose that $ F $ is renormalizable. Let $ \bde_1 $ be $ \pi_{\Bz} \circ RF $. The recursive formulas of partial derivatives of $ D\bde_1 $ inspires to find an invariant space under renormalization operator.

\subsection{A space of renormalizable maps from recursive formulas of $ \de $}
Let $ F $ be the renormalizable H\'enon-like map. Lemma \ref{lem-recursive formula of delta} implies that
\begin{align*}
\di_x \de^j_1(w) &= \boxed {\Big[\, \di_y \de^j \circ \psi^1_c(w) + \sum_{l=1}^m \di_{z_l} \de^j \circ \psi^1_c(w) \cdot \di_x \de^l \circ \psi^1_v(w) \,\Big]} \cdot \di_x \phi^{-1}( \si_0 w) \\[-0.1em]
& \quad \ + \di_x \de^j \circ \psi^1_c(w) - \frac{d}{dx}\,\de^j (\si_0 x,\,f^{-1}(\si_0 x), {\bf 0}) \\[0.8em]
\di_y \de^j_1(w) &= \boxed{\Big[\, \di_y \de^j \circ \psi^1_c(w) + \sum_{l=1}^m \di_{z_l} \de^j \circ \psi^1_c(w) \cdot \di_x \de^l \circ \psi^1_v(w) \,\Big]} \cdot \di_y \phi^{-1}( \si_0 w) \\[-0.1em]
& \quad \ + \sum_{i=1}^m \di_{z_i} \de^j \circ \psi^1_c(w) \cdot \Big[\, \di_y \de^i \circ \psi^1_v(w) + \sum_{l=1}^m \di_{z_l} \de^i \circ \psi^1_v(w) \cdot \frac{d}{dy}\,\de^l(\si_0y,f^{-1}(\si_0 y), {\bf 0}) \,\Big] \\[0.7em]
\di_{z_i} \de^j_1(w) &= \boxed{\Big[\, \di_y \de^j \circ \psi^1_c(w) + \sum_{l=1}^m \di_{z_l} \de^j \circ \psi^1_c(w) \cdot \di_x \de^l \circ \psi^1_v(w) \,\Big]} \cdot \di_{z_i} \phi^{-1}( \si_0 w) \\[-0.1em]
& \quad \ + \sum_{l=1}^m \di_{z_l} \de^j \circ \psi^1_c(w) \cdot \di_{z_i} \de^l \circ \psi^1_v(w)
\end{align*} 
for $ 1 \leq j \leq m $ and $ 1 \leq i \leq m $. Then the common factor of the first terms of each partial derivatives is the expression in the box. Let us consider the set of renormalizable maps the expression of the box. 
\msk
\begin{defn} \label{def-definition of NN}
Let $ \NN $ be the set of $ m+2 $ dimensional renormalizable H\'enon-like maps satisfying the following equation
$$ \di_y \de^j \circ F(w) + \sum_{l=1}^m \di_{z_l} \de^j \circ F(w) \cdot \di_x \de^l(w) = 0 $$
where $ w \in \psi^1_c(B) \cup \psi^1_v(B) $ \,for all\, $ 1 \leq j \leq m $.
\end{defn} 

\begin{rem}
\nin This definition is the generalization for three dimensional H\'enon-like maps in \cite{Nam2}. 
\end{rem}

\begin{example}
H\'enon-like maps in the above class $ \NN $ is non empty and non trivial. For instance, let us consider the map $ \bde $ of $ F \in \NN $ as follows
$$ \de^j(x,y,\Bz) = \eta^j \circ \Big(\;\! C_{j,\;0}\,y - \sum_{i=1}^m C_{j,\;i}\, z_i \Big) + x $$
where $ C_{j,\;0} = \sum_{i=1}^m C_{j,\;i} $ \,for every\, $ 1 \leq j \leq m $. Moreover, $ \displaystyle{\max_{ 1 \leq \, i,\;\! j \,\leq \,m} \{ \| \,\eta^j \|,\,C_{j,\;i} \} = O(\oeps)\;} $ for small enough $ \oeps > 0 $. We can choose $ \eta^j $ independent of $ \eta^k $ for $ j \neq k $.
\end{example}
\msk
\nin The derivative of $ \bde $ is of $ m \times (m+2) $ matrix form which contains each partial derivatives of $ \de^j $ for $ 1 \leq j \leq m $. Thus
\begin{equation*}
\begin{aligned}
D\bde(w) = \left(
\begin{array}{c|c|ccc}
\di_x \de^1(w) & \di_y \de^1(w) & \di_{z_1} \de^1(w) & \cdots & \di_{z_m}\de^1(w) \\
\vdots & \vdots & \vdots & \ddots & \vdots \\
\di_x \de^m(w) & \di_y \de^m(w) & \di_{z_1} \de^m(w) & \cdots & \di_{z_m}\de^m(w) 
\end{array} \right)
\end{aligned} \msk
\end{equation*} 
Let us introduce the notation to simply expressions \msk
\begin{equation}
\begin{aligned}
X(w) &= \ (\di_x \de^1(w) \cdots \di_x \de^m(w) )^{Tr} \\[0.4em]
Y(w) &= \ (\di_y \de^1(w) \cdots \di_y \de^m(w) )^{Tr} \\[0.5em]
Z(w) &=
\begin{pmatrix}
\di_{z_1} \de^1(w) & \cdots & \di_{z_m}\de^1(w) \\
\vdots & \ddots & \vdots \\
\di_{z_1} \de^m(w) & \cdots & \di_{z_m}\de^m(w) 
\end{pmatrix} 
\end{aligned} \msk
\end{equation}
where $ Tr $ is the transpose of the matrix. Then $ D\bde(w) = \left( X(w) \ \ Y(w) \ \ Z(w) \right) $. For infinitely renormalizable map $ F \in \II_B(\oeps) $, let us express $ D\bde_k(w) = \left( X_k(w) \ \ Y_k(w) \ \ Z_k(w) \right) $. The equation in Definition \ref{def-definition of NN} can be express as the sum of vectors for $ 1 \leq j \leq m $
\msk
\begin{align*}
\begin{pmatrix}
\di_y \de^1 \circ F(w) \\
\vdots \\
\di_y \de^m \circ F(w) 
\end{pmatrix} \ + \
\begin{pmatrix}
\di_{z_1} \de^1 \circ F(w) & \cdots & \di_{z_m}\de^1 \circ F(w) \\
\vdots & 
& \vdots \\
\di_{z_1} \de^m \circ F(w) & \cdots & \di_{z_m}\de^m \circ F(w) 
\end{pmatrix} \cdot
\begin{pmatrix}
\di_x \de^1(w) \\
\vdots \\
\di_x \de^m(w) 
\end{pmatrix} = {\bf 0} .
\end{align*} 
The above equation is the same as 
\begin{equation} \label{eq-another equation for NN}
Y \circ F(w) + Z \circ F(w) \cdot X(w) = {\bf 0}
\end{equation}
\nin Then $ \NN $ can be defined as the set of renormalizable map $ F $ satisfying the above equation \eqref{eq-another equation for NN} where $ w \in \psi^1_c(B) \cup \psi^1_v(B) $. 
Let
\begin{equation*}
q^j(y) = \frac{d}{dy}\, \de^j(y, f^{-1}(y),0)
\end{equation*}
for $ 1 \leq j \leq m $. Let $ \Bq(y) = (q^1(y), \ldots, q^m(y)) $. For instance, the value of the derivative of $ \de^j(y, f^{-1}(y),0) $ at $ \si_0 y $ is expressed as $ q^j \circ (\si_0 y) $. 
For the map $ F \in \NN $,\, $ X_1 $, $ Y_1 $ and $ Z_1 $ are as follows
\msk
\begin{equation} \label{eq-recursive formula of Dde_1}
\begin{aligned}
X_1(w) &= X \circ \psi^1_c(w) - \Bq \circ (\si_0 x) \\[0.5em]
Y_1(w) &= Z \circ \psi^1_c(w) \cdot Y \circ \psi^1_v(w) + Z \circ \psi^1_c(w) \cdot Z \circ \psi^1_v(w) \cdot \Bq \circ (\si_0 y) \\[0.2em]
&= Z \circ \psi^1_c(w) \cdot \big[\, Y \circ \psi^1_v(w) + Z \circ \psi^1_v(w) \cdot \Bq \circ (\si_0 y) \,\big] \\[0.5em]
Z_1(w) &= Z \circ \psi^1_c(w) \cdot Z \circ \psi^1_v(w) .
\end{aligned}
\end{equation}

\msk
\subsection{Invariance of the space $ \NN $ under renormalization}
Let us show the invariance of $ \NN \cap \II_B(\oeps) $ under renormalization operator.

\begin{lem} \label{lem-recursive formula of de-k}
Suppose that $ F \in \NN $ is $ n $ times renormalizable. Let $ F_k $ be $ R^kF $ and let $ \bde_k $ be $ \pi_{\Bz} \circ F^k $ for $ k =1,2, \ldots ,n $. Then
\msk
\begin{equation*} 
\begin{aligned}
X_k(w) &= X_{k-1} \circ \psi^k_c(w) - \Bq_{k-1} \circ (\si_{k-1} x) \\[0.4em]
Y_k(w) &= Z_{k-1} \circ \psi^k_c(w) \cdot Y_{k-1} \circ \psi^k_v(w) + Z_{k-1} \circ \psi^k_c(w) \cdot Z_{k-1} \circ \psi^k_v(w) \cdot \Bq_{k-1} \circ (\si_{k-1} y) \\[0.2em]
&= Z_{k-1} \circ \psi^k_c(w) \cdot \big[\, Y_{k-1} \circ \psi^k_v(w) + Z_{k-1} \circ \psi^k_v(w) \cdot \Bq_{k-1} \circ (\si_{k-1} y) \,\big] \\[0.4em]
Z_k(w) &= Z_{k-1} \circ \psi^k_c(w) \cdot Z_{k-1} \circ \psi^k_v(w)
\end{aligned} \msk
\end{equation*}
for $ k = 1,2, \ldots, n $.
\end{lem}
\begin{proof}
See the equation \eqref{eq-recursive formula of Dde_1} and use the induction.
\end{proof}
\msk
\begin{lem} \label{lem-relation of F-k and F-k-1 with psi}
Let $ F $ be an infinitely renormalizable H\'enon-like map and let $ F_k $ be $ R^kF $ for each $ k \in \N $. Then
$$ \psi^k_v \circ F_k = F_{k-1} \circ \psi^k_c $$
for each $ k \in \N $. Moreover, 
$$ \pi_y \circ \psi^k_v \circ F_k = \pi_x \circ \psi^k_c $$
for each $ k \in \N $.
\end{lem}
\begin{proof}
Let us recall that $ \psi^k_v = H_{k-1} \circ \La_{k-1} $, $ \psi^k_c = F_{k-1} \circ \psi^k_v $ and $ F_k = (\psi^k_v)^{-1} \circ F_{k-1}^2 \circ \psi^k_v $. Then
\begin{equation} \label{eq-relation between psi-c,v}
\begin{aligned}
 \psi^k_v \circ F_k &= \psi^k_v \circ (\psi^k_v)^{-1} \circ F_{k-1}^2 \circ \psi^k_v \\[0.3em]
&= F_{k-1}^2 \circ \psi^k_v = F_{k-1} \circ \big[\,F_{k-1} \circ \psi^k_v \,\big] \\[0.3em] 
&= F_{k-1} \circ \psi^k_c
\end{aligned} \msk
\end{equation}
for $ k \in \N $. Take two points $ w = (x,y, \Bz) $ and $ w'= (x',y', \Bz') $ satisfying the equation $ w = F_k(w') $. Since $ F_k $ is a H\'enon-like map, we obtain
$$ (x,y,\Bz) = F(w') = (f(x') - \eps(w'),\; x',\; \bde(w')) . $$
Then $ \pi_y (F(w')) = \pi_x (w') $. Hence, the equation \eqref{eq-relation between psi-c,v} implies that $ \pi_y \circ \psi^k_v \circ F_k = \pi_x \circ \psi^k_c $.
\end{proof}
\msk
\begin{thm} \label{thm - invariant space NN}
The space $ \NN \cap \II_B(\oeps) $ is invariant under renormalization, that is, if $ F \in \NN \cap \II_B(\oeps) $, then $ RF \in \NN \cap \II_B(\oeps) $.
\end{thm}
\begin{proof}
Let $ \pi_{\Bz} \circ R^kF $ be $ \bde_k $ and $ D\bde(w) = (X_k(w) \ Y_k(w) \ Z_k(w)) $ for $ k \in \N $. Suppose that
$$ Y_{k-1} \circ F_{k-1}(w) + Z_{k-1} \circ F_{k-1}(w) \cdot X_{k-1}(w) = \B0 $$
where $ w \in \psi^k_c(B) \cup \psi^k_v(B) $. By induction it suffice to show that 
$$ Y_{k} \circ F_{k}(w) + Z_{k} \circ F_{k}(w) \cdot X_{k}(w) = \B0 $$
where $ w \in \psi^{k+1}_c(B) \cup \psi^{k+1}_v(B) $. Observe that $ \si_{k-1}\;\! y = \pi_y \circ \psi^k_v(w) $ and \ssk $ \si_{k-1}\;\! x = \pi_x \circ \psi^k_c(w) $. By Lemma \ref{lem-relation of F-k and F-k-1 with psi}, we have that 
$ \pi_x \circ \psi^k_c(w) = \pi_y \circ \psi^k_v \circ F_{k}(w) $ and $ F_k \circ \psi^k_c(w) = \psi^k_v \circ F_{k}(w) $. Then
\msk
\begin{equation}
\begin{aligned}
&\quad \ \ Y_{k} \circ F_{k}(w) + Z_{k} \circ F_{k}(w) \cdot X_{k}(w) \\[0.6em]
&= \ Z_{k-1} \circ \psi^k_c \circ F_{k}(w) 
\cdot \big[\, Y_{k-1} \circ \psi^k_v \circ F_{k}(w) + Z_{k-1} \circ \psi^k_v \circ F_{k}(w) \cdot \Bq_{k-1} \circ \pi_y \circ \psi^k_v \circ F_{k}(w) \,\big] \\[0.3em]
&\qquad \ + Z_{k-1} \circ \psi^k_c \circ F_{k}(w) \cdot Z_{k-1} \circ \psi^k_v \circ F_{k}(w) 
\cdot \big[\, X_{k-1} \circ \psi^k_c(w) - \Bq_{k-1} \circ \pi_x \circ \psi^k_c(w) \,\big] \\[0.6em]
&= \ Z_{k-1} \circ \psi^k_c \circ F_{k}(w) \cdot Y_{k-1} \circ \psi^k_v \circ F_{k}(w) \\[0.3em]
&\qquad \ + Z_{k-1} \circ \psi^k_c \circ F_{k}(w) \cdot Z_{k-1} \circ \psi^k_v \circ F_{k}(w) \cdot X_{k-1} \circ \psi^k_c(w) \\[0.6em]
&= \ Z_{k-1} \circ \psi^k_c \circ F_{k}(w) \cdot \big[\, Y_{k-1} \circ \psi^k_v \circ F_{k}(w) + Z_{k-1} \circ \psi^k_v \circ F_{k}(w) \cdot X_{k-1} \circ \psi^k_c(w) \,\big] \\[0.5em]
&= \ \B0 .
\end{aligned} 
\end{equation}
For any point $ w \in \Dom(R^kF) $, we obtain \ssk that $ \psi^k_c(w) \in \psi^k_c(B) \cup \psi^k_v(B) $. Then $ F_k \in \NN \cap \II_B(\oeps) $. Hence, the space $ \NN \cap \II_B(\oeps) $ is invariant under renormalization.
\end{proof}

\bsk
\section{Universal numbers with $ \di_{\Bz} \bde $ and $ \di_y \eps $}
\subsection{Critical point and the recursive formula of $ \di_x \bde $}

\begin{prop}
Let $ F $ be the H\'enon-like map in $ \NN \cap \II_B(\oeps) $. Let $ \bde_k $ be $ \pi_{\Bz} \circ F_k $ for $ k \in \N $ and let $ X_k(w) $ be the column matrix of $ (\di_x \de^j_k) $ for $ 1 \leq j \leq m $. Then 
$$ X_n(w) = X_k \circ \Psi^n_{k,\,\cc}(w) - \sum_{i=k}^{n-1} \Bq_i \circ (\pi_x \circ \Psi^n_{i,\,\cc}(w)) $$
for $ k < n $. Moreover, passing the limit
$$ X_k(c_{F_k}) = \lim_{n \ra \infty} \sum_{i=k}^n \Bq_i \circ (\pi_x (c_{F_i})) $$
where $ c_{F_i} $ is the critical point of $ F_i $ for each $ k \leq i \leq n $.
\end{prop}
\begin{proof}
By Lemma \ref{lem-recursive formula of de-k}, we have \msk
\begin{equation} \label{eq-X-k and X-n at the critical point}
\begin{aligned}
X_n(w) &= X_{n-1} \circ \psi^n_c(w) - \Bq_{n-1} \circ (\si_{n-1} x) \\[0.2em]
&= X_{n-1} \circ \psi^n_c(w) - \Bq_{n-1} \circ (\pi_x \circ \psi^n_c(w)) \\[0.2em]
& \hspace{1in} \vdots \\
&= X_k \circ \Psi^n_{k,\,\cc}(w) - \sum_{i=k}^{n-1} \Bq_i \circ (\pi_x \circ \Psi^n_{i,\,\cc}(w))
\end{aligned} \msk
\end{equation}
for $ k < n $. Since $ \| X_n \| = O(\oeps^{2^n}) $, passing the limit we have the following equation
\msk
\begin{equation}
\lim_{n \ra \infty} X_k \circ \Psi^n_{k,\,\cc}(w) = \lim_{n \ra \infty} \;\sum_{i=k}^{n-1} \Bq_i \circ (\pi_x \circ \Psi^n_{i,\,\cc}(w)) .
\end{equation}

\nin By Corollary \ref{cor-image of critical point under Psi}, we obtain that
\msk
\begin{equation*}
\begin{aligned}
X_k(c_{F_k}) &= \lim_{n \ra \infty} \; \sum_{i=k}^{n-1} \Bq_i \circ (\pi_x \circ \Psi^n_{i,\,\cc}(w)) \\
& = \lim_{n \ra \infty} \; \sum_{i=k}^{n-1} \Bq_i \circ (\pi_x(c_{F_i})) .
\end{aligned}
\end{equation*}
\end{proof}

\begin{rem}
Putting the critical point $ c_{F_n} $ at the equation \eqref{eq-X-k and X-n at the critical point}, we obtain that
\begin{equation} \label{eq-formula of X-k at the critical point}
X_k(c_{F_k}) = \sum_{i=k}^{n-1} \Bq_i \circ (\pi_x(c_{F_i})) + X_n(c_{F_n})
\end{equation}
\end{rem}

\msk
\subsection{Universal numbers $ b_{\Bz} $ with the asymptotic of $ Z(w) $}

\begin{prop} \label{prop-universal number b-Z}
Let $ F $ be the H\'enon-like diffeomorphism in $ \NN \cap \II_B(\oeps) $. Let $ \pi_{\Bz} \circ F_k $ be $ \bde_k $ for $ k \in \N $ and $ m \times m $ matrix, $ Z_k(w) $ be $ \di_{z_i} \de_k^j(w) $ for $ 1 \leq i,j \leq m $. Suppose that $ \det Z(w) $ is non zero for all $ w \in \psi^1_c(B) \cup \psi^1_v(B) $. Then
$$ | \det Z_n(w) | = b_{\Bz}^{2^n}(1 + O(\rho^n)) $$
where $ b_{\Bz} $ is a universal positive number for each $ n \in \N $ and for some $ 0 < \rho < 1 $.
\end{prop}
\begin{proof}
For the map $ F_n \in \NN \cap \II_B(\oeps) $, Definition \ref{def-definition of NN} implies that
\begin{equation*}
\frac{\di}{\di z_i}\, \de^j_n(w) = \sum_{l=1}^m \di_{z_l} \de^j_{n-1} \circ \psi^n_c(w) \cdot \di_{z_i} \de^l_{n-1} \circ \psi^n_v(w)
\end{equation*}
for $ n \in \N $. Since $ \dfrac{\di}{\di z_i}\, \de^j_n(w) $ is the $ (j,i) $ element of the matrix $ Z_n(w) $, it is the product of the $ j^{th} $ row of $ Z_{n-1} \circ \psi^n_c(w) $ and $ i^{th} $ column of $ Z_{n-1} \circ \psi^n_v(w) $. Then the equation \eqref{eq-recursive formula of Dde_1} is generalized with $ n \in \N $ with inductive calculation \msk
\begin{equation}
\begin{aligned}
Z_n(w) &= Z_{n-1} \circ \psi^n_c(w) \cdot Z_{n-1} \circ \psi^n_v(w) \\[0.5em]
&= Z_{n-2} \circ \psi^{n-1}_c \circ \psi^n_c(w) \cdot Z_{n-2} \circ \psi^{n-1}_v \circ \psi^n_c(w) \\[0.3em]
& \qquad \cdot Z_{n-2} \circ \psi^{n-1}_c \circ \psi^n_v(w) \cdot Z_{n-2} \circ \psi^{n-1}_v \circ \psi^n_v(w) \\[0.2em]
&\hspace{1in} \vdots \\
& = \prod_{\ww \in W^n} Z \circ \Psi^n_{\ww}(w)
\end{aligned} \msk
\end{equation}
where $ {\bf w} $ is a word in $ W^n $ which is the set of Cartesian product of the letters $ \{v,\,c\}^n $. Let us take the logarithmic average of $ | \det Z_n(w) | $ as follows
\begin{equation*}
\begin{aligned}
l_n(w) &= \frac{1}{\;2^n} \,\log \,\big|\;\! \det Z_n(w) \;\!\big| \\[0.3em]
&= \frac{1}{\;2^n} \,\sum_{\ww \in \;\! W^n} \log \,\big|\;\! \det Z \circ \Psi^n_{\ww}(w) \;\!\big| .
\end{aligned}
\end{equation*}

\nin By the continuity of the determinant of matrix and the compactness of the domain, we may assume that $ \det Z_n(w) $ has its positive lower bound or negative upper bound. 
The the limit $ l_n(w) $ exists as $ n \ra n $ on the critical Cantor set $ \OO_F $ where $ \mu $ is the unique ergodic probability measure on $ \OO_F $. The uniqueness of $ \mu $ implies that the limit is a constant function. Let this constant be $ \log b_{\Bz} $ for some $ b_{\Bz} > 0 $, that is,
\begin{equation} \label{eq-convergence to log b-z}
l_n(w) \lra \int_{\OO_F} \log \big| \det Z(w) \big| \;d\mu \equiv \log b_{\Bz} .
\end{equation}

\nin Since $ \diam \,(\Psi^n_{\ww}(B)) \leq C \si^n $ for all $ \ww \in W^n $ and for some $ C>0 $, the $ l_n $ in \eqref{eq-convergence to log b-z} converges exponentially fast. In other words,
\begin{equation*}
\frac{1}{\;2^n} \,\log \,\big|\;\! \det Z_n(w) \;\!\big| = \log b_{\Bz} + O(\rho^n_0)
\end{equation*}
for some $ 0 < \rho_0 < 1 $. Take the constant $ \rho = \rho_0/2 $. Then \msk
\begin{equation*}
\begin{aligned}
\log \,\big|\:\! \det Z_n(w) \;\!\big| & = 2^n \log b_{\Bz} + O(\rho^n) \\[0.2em]
&= 2^n \log b_{\Bz} + \log (1 + O(\rho^n)) \\[0.2em]
&= \log b_{\Bz}^{2^n} (1 + O(\rho^n)) .
\end{aligned} \msk
\end{equation*}
Hence, 
\begin{equation}
\big|\:\! \det Z_n(w) \;\!\big| = b_{\Bz}^{2^n} (1 + O(\rho^n)).
\end{equation}
The proof is complete.
\end{proof}
\ssk

\begin{lem} \label{lem-relation between Y-n and Z-n with Y-k and Z-k}
Let $ F $ be the H\'enon-like diffeomorphism in $ \NN \cap \II_B(\oeps) $. Let $ D\bde_n $ be $ (X_n \ Y_n \ Z_n ) $. Then
$$ Y_n(w) = Z_n(w) \cdot \Big[\, \big(Z_k \circ \Psi^n_{k,\,\vv}(w) \big)^{-1} \cdot \big(Y_k \circ \Psi^n_{k,\,\vv}(w) \big) + \sum_{i=k}^{n-1} \Bq_i \circ (\pi_y \circ \Psi^n_{i,\,\vv}(w)) \,\Big] $$
for each $ n \in \N $. Moreover, 
$$ \big(Z_k \circ \Psi^n_{k,\,\vv}(w) \big)^{-1} \cdot \big(Y_k \circ \Psi^n_{k,\,\vv}(w) \big) + \sum_{i=k}^{n-1} \Bq_i \circ (\pi_y \circ \Psi^n_{i,\,\vv}(w)) $$
converges to $ \B0 $ exponentially fast as $ n \ra \infty $. In particular, each coordinate of above expression is less than $ C \si^{n-k} $ for some $ C>0 $ independent of $ k $.
\end{lem}
\begin{proof}
Lemma \ref{lem-recursive formula of de-k} implies \msk
\begin{equation} \label{eq-recursive formula of Z-n Y-n again}
\begin{aligned}
Y_n(w) &= Z_{n-1} \circ \psi^n_c(w) \cdot Y_{n-1} \circ \psi^n_v(w) + Z_{n-1} \circ \psi^n_c(w) \cdot Z_{n-1} \circ \psi^n_v(w) \cdot \Bq_{n-1} \circ (\si_{n-1} y) \\[0.2em]
&= Z_{n-1} \circ \psi^n_c(w) \cdot Y_{n-1} \circ \psi^n_v(w) + Z_{n} \circ \psi^n_v(w) \cdot \Bq_{n-1} \circ (\pi_y \circ \psi^n_v(w)) \,\big] \\[0.5em]
Z_n(w) &= Z_{n-1} \circ \psi^n_c(w) \cdot Z_{n-1} \circ \psi^n_v(w)
\end{aligned} \msk
\end{equation}
for $ n \in \N $. Thus by the inductive calculation, $ Y_n(w) $ as follows 
\begin{align*}
& \quad \ Y_n(w) \\[0.5em]
&= Z_{n-1} \circ \psi^n_c(w) \cdot Y_{n-1} \circ \psi^n_v(w) + Z_{n} \circ \psi^n_v(w) \cdot \Bq_{n-1} \circ (\pi_y \circ \psi^n_v(w)) \\[0.6em]
&= Z_{n-1} \circ \psi^n_c(w) \cdot \big[\, Z_{n-2} \circ (\psi^{n-1}_c \circ \psi^n_v)(w) \cdot Y_{n-2} \circ (\psi^{n-1}_v \circ \psi^n_v)(w) \\[0.2em]
&\qquad + Z_{n-1} \circ \psi^{n}_v(w) \cdot \Bq_{n-2} \circ (\pi_y \circ \psi^{n-1}_v \circ \psi^n_v(w)) \,\big] 
+ Z_{n} \circ \psi^n_v(w) \cdot \Bq_{n-1} \circ (\pi_y \circ \psi^n_v(w)) \\[0.6em]
&= Z_{n-1} \circ \psi^n_c(w) \cdot Z_{n-2} \circ (\psi^{n-1}_c \circ \psi^n_v)(w) \cdot Y_{n-2} \circ (\psi^{n-1}_v \circ \psi^n_v)(w) \\[0.2em]
&\qquad + Z_n(w) \cdot \big[\, \Bq_{n-2} \circ (\pi_y \circ \psi^{n-1}_v \circ \psi^n_v(w)) + \Bq_{n-1} \circ (\pi_y \circ \psi^n_v(w)) \,\big] \\[0.2em]
&\hspace{2in} \vdots \\[0.2em]
&= Z_{n-1} \circ \psi^n_c(w) \cdot Z_{n-2} \circ (\psi^{n-1}_c \circ \psi^n_v)(w) \cdots Z_k \circ (\psi^{k+1}_c \circ \psi^{k+2}_v \circ \cdots \circ \psi^n_v(w)) \\
& \qquad \cdot Y_k \circ \Psi^n_{k,\, \vv}(w) 
 + Z_n(w) \cdot \sum_{i=k}^{n-1} \Bq_i \circ (\pi_y \circ \Psi^n_{i,\,\vv}(w)) \\[-0.4em]
(*) \ \ &= Z_n(w) \cdot \big( Z_k \circ \Psi^n_{k,\,\vv}(w) \big)^{-1} \cdot Y_k \circ \Psi^n_{k,\,\vv}(w) + Z_n(w) \cdot \sum_{i=k}^{n-1} \Bq_i \circ (\pi_y \circ \Psi^n_{i,\,\vv}(w)) \\[-0.4em]
&= Z_n(w) \cdot \Big[\, \big( Z_k \circ \Psi^n_{k,\,\vv}(w) \big)^{-1} \cdot Y_k \circ \Psi^n_{k,\,\vv}(w) + \sum_{i=k}^{n-1} \Bq_i \circ (\pi_y \circ \Psi^n_{i,\,\vv}(w)) \,\Big] 
\end{align*} \msk
for $ k<n $. The second equation in \eqref{eq-recursive formula of Z-n Y-n again} implies $ (*) $ in the above equation. Recall the equation \eqref{eq-X-k and X-n at the critical point}
\begin{equation*}
\begin{aligned}
X_k \circ \Psi^n_{k,\,\cc}(w) = \sum_{i=k}^{n-1} \Bq_i \circ (\pi_x \circ \Psi^n_{i,\,\cc}(w)) + X_n(w)
\end{aligned} \msk
\end{equation*}
Thus the expression
$$ \sum_{i=k}^{n-1} \Bq_i \circ (\pi_x \circ \Psi^n_{i,\,\cc}(w)) = \sum_{i=k}^{n-1} \Bq_i \circ ( \si_{n,\,i}\;\! x) $$
converges to the value $ X_k(c_{F_k}) $ exponentially fast as $ n \ra \infty $ where $ \si_{n,\,i} = \si_n \cdot \si_{n-1} \cdot \si_i $. Thus the expression
$$ \sum_{i=k}^{n-1} \Bq_i \circ (\pi_y \circ \Psi^n_{i,\,\vv}(w)) = \sum_{i=k}^{n-1} \Bq_i \circ (\si_{n,\,i}\;\! y) $$
also converges to a single value exponentially fast as \, $ n \ra \infty $. Moreover, since $ \Psi^n_{i,\,\vv}(B) $ contains $ \tau_i $ for all $ i < n $, we have the following equation
\begin{equation*}
\lim_{n \ra \infty} \sum_{i=k}^{n-1} \Bq_i \circ (\pi_y \circ \Psi^n_{i,\,\vv}(w)) = \lim_{n \ra \infty} \sum_{i=k}^{n-1} \Bq_i \circ (\pi_y (\tau_i)) .
\end{equation*} \msk

\nin Recall that if $ F \in \II_B(\oeps) $, then $ F_j(c_{F_j}) = \tau_j $ for $ j \in \N $. By the equation \eqref{eq-formula of X-k at the critical point}, we obtain that
\begin{align*}
&\quad \ \lim_{n \ra \infty} \Big[\,\big( Z_k \circ \Psi^n_{k,\,\vv}(w) \big)^{-1} \cdot Y_k \circ \Psi^n_{k,\,\vv}(w) + \sum_{i=k}^{n-1} \Bq_i \circ (\pi_y \circ \Psi^n_{i,\,\vv}(w)) \,\Big] \\
&= \ \big( Z_k(\tau_k) \big)^{-1} \cdot Y_k (\tau_k) + \lim_{n \ra \infty}\sum_{i=k}^{n-1} \Bq_i \circ (\pi_y (\tau_i)) \\
& = \ \big( Z_k(\tau_k) \big)^{-1} \cdot Y_k (\tau_k) + \lim_{n \ra \infty}\sum_{i=k}^{n-1} \Bq_i \circ (\pi_x (c_{F_i})) \\[0.2em]
& = \ \big( Z_k(\tau_k) \big)^{-1} \cdot Y_k (\tau_k) + X_k(c_{F_k}) - \lim_{n \ra \infty}X_n(c_{F_n}) \\[0.6em]
& = \ \B0
\end{align*} \msk
for $ k<n $. Since $ \diam \,(\Psi^n_{k,\,\vv}(B)) \leq C \si^{n-k} $, the convergence of above equation is exponentially fast.
\end{proof}
\msk

\begin{cor} \label{cor-Y-k Z-k and q estimation}
Let $ F \in \NN \cap \II_B(\oeps) $. For $ \pi_{\Bz} \circ R^kF = \bde_k $, let $ D\bde_k $ be $ (X_k(w) \ Y_k(w) \ Z_k(w)) $. Then
\begin{equation*}
\Big\| \;\! Y_k \circ \Psi^n_{k,\,\vv}(w) + Z_k \circ \Psi^n_{k,\,\vv}(w) \cdot \sum_{i=k}^{n-1} \Bq_i \circ (\pi_y \circ \Psi^n_{i,\,\vv}(w)) \:\! \Big\| \leq C\,\oeps^{2^k} \si^{n-k} 
\end{equation*}
for some $ C>0 $.
\end{cor}
\begin{proof}
\begin{align*}
& \quad \ \ Y_k \circ \Psi^n_{k,\,\vv}(w) + Z_k \circ \Psi^n_{k,\,\vv}(w) \cdot \sum_{i=k}^{n-1} \Bq_i \circ (\pi_y \circ \Psi^n_{i,\,\vv}(w)) \\[-0.4em]
& = \ Z_k \circ \Psi^n_{k,\,\vv}(w) \cdot \Big[\, \big( Z_k \circ \Psi^n_{k,\,\vv}(w) \big)^{-1} \cdot Y_k \circ \Psi^n_{k,\,\vv}(w) + \sum_{i=k}^{n-1} \Bq_i \circ (\pi_y \circ \Psi^n_{i,\,\vv}(w)) \,\Big] 
\end{align*}
By Lemma \ref{lem-relation between Y-n and Z-n with Y-k and Z-k}, we have  
\begin{equation*}
\Big\| \;\! \big( Z_k \circ \Psi^n_{k,\,\vv}(w) \big)^{-1} \cdot Y_k \circ \Psi^n_{k,\,\vv}(w) + \sum_{i=k}^{n-1} \Bq_i \circ (\pi_y \circ \Psi^n_{i,\,\vv}(w)) \:\! \Big\| \leq C_0\, \si^{n-k} 
\end{equation*}
and $ \| \;\! Z_k \| \leq C_1\,\oeps^{2^k} $.

\end{proof}

\msk

\subsection{Universal number $ b_1 $ and $ \di_y \eps $} \label{subsec-universal number b-1}
The universal number $ b_{\Bz} $ is defined as the asymptotic number of $ \det Z(w) $. Universality of Jacobian determinant implies that the average Jacobian $ b $ is the universal number of $ \Jac F $. Let us define another number $ b_1 $ as the ratio $ b_1 = b/b_{\Bz} $. Then it is shown that $ b_1 $ is the universal number for $ \di_y \eps $ below.
\\ \ssk
\nin 
Let $ M $ be a block matrix 
\begin{equation*}
\begin{aligned}
M =
\begin{pmatrix}
A & B \\
C & D
\end{pmatrix}
\end{aligned} \msk
\end{equation*}
where $ A $ and $ D $ are square matrices. Assume that the square matrix $ D $ is invertible. The equation
\begin{equation*}
\begin{aligned}
\begin{pmatrix}
A & B \\
C & D
\end{pmatrix}
\begin{pmatrix}
I & \B0 \\
D^{-1}C & I
\end{pmatrix} = 
\begin{pmatrix}
A- BD^{-1}C & B \\
\B0 & D
\end{pmatrix}
\end{aligned} \msk
\end{equation*}
Then $ \det M = \det (A - BD^{-1}C) \cdot \det (D) $. 
Applying this, we have that \msk
\begin{equation}
\begin{aligned}
\Jac F_n &= \det 
\begin{pmatrix}
\di_y \eps_n & E_n \\[0.2em]
Y_n & Z_n
\end{pmatrix} \\[0.3em]
&= \det \,(\,\di_y \eps_n - E_n \cdot (Z_n)^{-1} \cdot Y_n \,) \cdot \det Z_n \\[0.4em]
&= \big[\,\di_y \eps_n - E_n \cdot (Z_n)^{-1} \cdot Y_n \,\big] \cdot \det Z_n
\end{aligned} \msk
\end{equation}
where $ \big(\,\di_y \eps_n(w) \quad \di_{z_1} \eps_n(w) \quad \di_{z_2} \eps_n(w) \ \cdots \ \di_{z_m} \eps_n(w) \big) $ is $ E_n(w) $ for $ n \in \N $. \ssk Thus \,$ \det Z_n(w) $ is not zero for all $ w \in B $. Universality of Jacobian and Proposition \ref{prop-universal number b-Z} implies that \msk
\begin{equation*}
\begin{aligned}
\Jac F_n(w) &= b^{2^n}a(x)(1 + O(\rho^n)) \\[0.2em]
\det Z_n(w) &= b^{2^n}_{\Bz}(1 + O(\rho^n))
\end{aligned} \ssk
\end{equation*}
where $ b $ is the average Jacobian and $ b_{\Bz} $ is a positive number for some $ 0 < \rho <1 $. Let $ b_1 = b/b_{\Bz} $. Then by Lemma \ref{lem-relation between Y-n and Z-n with Y-k and Z-k}, we obtain that \msk
\begin{equation*}
\begin{aligned}
\Jac F_k(w) &= \ b^{2^k}a(x)(1 + O(\rho^k))
\\[0.2em]
&= \ \big[\,\di_y \eps_k(w) - E_k(w) \cdot (Z_k(w))^{-1} \cdot Y_k(w) \,\big] \cdot b^{2^k}_{\Bz}(1 + O(\rho^k)) 
\end{aligned} \msk
\end{equation*}
Thus 
\begin{equation} \label{eq-universal expression with the number b-1}
\di_y \eps_k(w) - E_k(w) \cdot (Z_k(w))^{-1} \cdot Y_k(w) \ = \ b_1^{2^k}a(x)\,(1 + O(\rho^k))
\end{equation}

\msk
\begin{lem} \label{lem-di-y eps and q asymptotic from n to k}
Let $ F \in \NN \cap \II_B(\oeps) $ and $ F_k $ be $ R^kF $ for $ k \in \N $. Then
$$ \Big|\;\di_y\eps_k \circ (\Psi^n_{k,\,\vv}(w)) + E_k \circ (\Psi^n_{k,\,\vv}(w)) \cdot \sum_{i=k}^{n-1} \Bq_i \circ (\pi_y \circ (\Psi^n_{i,\,\vv}(w)) \; \Big| \ \leq \ C_0\,b_1^{2^k} + C_1\,\oeps^{2^k}\si^{n-k} $$
where $ w \in B(R^nF) $ for some positive $ C_0 $ and $ C_1 $. Moreover, if $ n $ satisfies that $ \si^{n-k} \asymp b_1^{2^k} $, then
$$ \di_y\eps_k \circ (\Psi^n_{k,\,\vv}(w)) + E_k \circ (\Psi^n_{k,\,\vv}(w)) \cdot \sum_{i=k}^{n-1} \Bq_i \circ (\pi_y \circ (\Psi^n_{i,\,\vv}(w)) \ \asymp \ b_1^{2^k}
$$
\end{lem}
\begin{proof}
The equation \eqref{eq-universal expression with the number b-1} implies that \msk
\begin{equation} \label{eq-b1-2-k with partial derivatives of de-k}
\begin{aligned}
&\quad \ 
\ b_1^{2^k}a\,(\pi_x \circ (\Psi^n_{k,\,\vv}(w))\,(1 + O(\rho^k)) \\[0.7em]
& = \ \di_y\eps_k \circ (\Psi^n_{k,\,\vv}(w)) - E_k \circ (\Psi^n_{k,\,\vv}(w)) \cdot (Z_k \circ (\Psi^n_{k,\,\vv}(w))^{-1} \cdot Y_k \circ (\Psi^n_{k,\,\vv}(w)) \\
& = \ \di_y\eps_k \circ (\Psi^n_{k,\,\vv}(w)) - E_k \circ (\Psi^n_{k,\,\vv}(w)) \cdot \Big[\, -\sum_{i=k}^{n-1} \Bq_i \circ (\pi_y \circ (\Psi^n_{i,\,\vv}(w)) + (Z_n(w))^{-1} \cdot Y_n(w) \,\Big] \\
& = \ \di_y\eps_k \circ (\Psi^n_{k,\,\vv}(w)) + E_k \circ (\Psi^n_{k,\,\vv}(w)) \cdot \sum_{i=k}^{n-1} \Bq_i \circ (\pi_y \circ (\Psi^n_{i,\,\vv}(w)) \\
& \qquad - E_k \circ (\Psi^n_{k,\,\vv}(w)) \cdot (Z_n(w))^{-1} \cdot Y_n(w) 
\end{aligned} \msk
\end{equation}
Lemma \ref{lem-relation between Y-n and Z-n with Y-k and Z-k} implies that
\msk
\begin{equation*}
\begin{aligned}
\| \:\! E_k \circ (\Psi^n_{k,\,\vv}(w)) \cdot (Z_n(w))^{-1} \cdot Y_n(w) \| & \leq \ \| \:\! E_k \circ (\Psi^n_{k,\,\vv}(w)) \| \cdot \| \:\!(Z_n(w))^{-1} \cdot Y_n(w) \| \\[0.2em]
& \leq \ C_1\, \oeps^{2^k} \si^{n-k} 
\end{aligned} 
\end{equation*}
for some $ C>0 $ independent of $ k $. Hence, 
$$ \Big| \;\di_y\eps_k \circ (\Psi^n_{k,\,\vv}(w)) + E_k \circ (\Psi^n_{k,\,\vv}(w)) \cdot \sum_{i=k}^{n-1} \Bq_i \circ (\pi_y \circ (\Psi^n_{i,\,\vv}(w)) \;\Big| \ \leq \ C_0\,b_1^{2^k} + C_1\, \oeps^{2^k}\si^{n-k}
$$
where $ w \in B(R^nF) $ for some positive $ C_0 $ and $ C_1 $. If $ \si^{n-k} \asymp b_1^{2^k} $ then $ \oeps^{2^k}\si^{n-k} \leq C\,b_1^{2^k} $ for some $ C>0 $.
\end{proof}

\bsk
\section{Recursive formula of $ \Psi^n_k $} \label{sec-recursive formula of Psi-n-k}

\begin{prop} \label{recursive formula of d, u, and t}
Let $ F \in \II_B(\bar \eps) $. $ F_k $ and $ F_n $ denote $ k^{th} $ and $ n^{th} $ renormalized map of $ F $ respectively. The derivative of the non-linear conjugation $ \Psi^n_k $ at the tip, $ \tau_{F_k} $ between $ F_k^{2^{n-k}} $ and $ F_n $ 
is called $ D^n_k $, which is as follows \msk
\begin{equation*}
\begin{aligned}
D^n_k = 
\begin{pmatrix}
\;\alpha_{n,\,k} & \si_{n,\,k}\, t_{n,\,k} & \si_{n,\,k}\, \Bu_{n,\,k} \\[0.2em]
& \si_{n,\,k} & \\[0.2em]
& \si_{n,\,k}\, \Bd_{n,\,k} & \si_{n,\,k} \cdot \Id_{m \times m} \;
\end{pmatrix}
\end{aligned} \msk
\end{equation*}
where $ \Id_{m \times m} $ is the $ m \times m $ identity matrix, $ \si_{n,\,k} $ and $ \alpha_{n,\,k} $ are \ssk linear scaling factors such that $ \si_{n,\,k} = (-\si)^{n-k} (1 + O(\rho^k)) $ and $ \alpha_{n,\,k} = \si^{2(n-k)} (1 + O(\rho^k)) $.
Then 
\begin{equation*}
\begin{aligned}
\Bd_{n,\,k} &= \sum_{i=k}^{n-1} \Bd_{i+1,\,i} \, , \quad \Bu_{n,\,k} = \sum_{i=k}^{n-1} \si^{i-k}\,\Bu_{i+1,\,i}\,(1 + O(\rho^k))\\
  t_{n,\,k} &= \sum_{i=k}^{n-1} \si^{i-k}\, \big[\, t_{i+1,\,i} + \Bu_{i+1,\,i}\cdot \Bd_{n,\,i+1} \big](1 + O(\rho^k)) \\
  t_{n,\,k}- \Bu_{n,\,k}\cdot \Bd_{n,\,k} &= \sum_{i=k}^{n-1} \si^{i-k}\, \big[\, t_{i+1,\,i} - \Bu_{i+1,\,i}\cdot \Bd_{i+1,\,k} \big](1 + O(\rho^k)) 
\end{aligned}
\end{equation*}
where\; $ \si^{i-k} (1 + O(\rho^k)) = {\displaystyle\prod_{j=k}^{i-1}} \dfrac{\alpha_{j+1,\,j}}{\si_{j+1,\,j}} $. Moreover, $ \Bd_{n,\,k} $, $ \Bu_{n,\,k} $ and $ t_{n,\,k} $ are convergent as $ n \ra \infty $ super exponentially fast.
\end{prop}

\begin{proof}
$ D^n_k = D^m_k \cdot D^n_m $ for any $ m $ between $ k $ and $ n $ because the image of the tip under $ \Psi^n_k(\tau_{F_n}) $ is the tip of $ k^{th} $ level. By the direct calculation, \msk
\begin{equation*}
\begin{aligned}
& \quad D^m_k \cdot D^n_m \\
& = 
\begin{pmatrix}
\alpha_{n,\,k} & \boxed{ \alpha_{m,\,k}\,\si_{n,\,m}\, t_{n,\,m} + \si_{n,\,k}\,t_{m,\,k} + \si_{n,\,k}\,\Bu_{m,\,k}\cdot \Bd_{n,\,m} } & \boxed{ \alpha_{m,\,k}\si_{n,\,m}\, \Bu_{n,\,m} + \si_{n,\,k}\,\Bu_{m,\,k} } \ \ \\[0.2em]
& \si_{n,\,k} & \\[0.2em]
& \boxed{ \si_{n,\,k}\, \Bd_{m,\,k} + \si_{n,\,k}\,\Bd_{n,\,m} } & \si_{n,\,k}\cdot \Id_{m \times m}
\end{pmatrix} .
\end{aligned} \msk
\end{equation*}
 Then
\begin{equation*}
\begin{aligned}
\si_{n,\,k}\,t_{n,\,k} &= \alpha_{m,\,k}\,\si_{n,\,m}\, t_{n,\,m} + \si_{n,\,k}\,t_{m,\,k} + \si_{n,\,k}\,\Bu_{m,\,k}\cdot \Bd_{n,\,m} \\
\si_{n,\,k}\,\Bu_{n,\,k} &= \alpha_{m,\,k}\,\si_{n,\,m}\, \Bu_{n,\,m} + \si_{n,\,k}\,\Bu_{m,\,k} \\
\si_{n,\,k}\,\Bd_{n,\,k} &= \si_{n,\,k}\, \Bd_{m,\,k} + \si_{n,\,k}\,\Bd_{n,\,m}
\end{aligned} \msk
\end{equation*}
for any $ m $ between $ k $ and $ n $. Recall that $ \si_{n,\,k} = \si_{n,\,m} \cdot \si_{m,\,k} $ and $ \alpha_{n,\,k} = \alpha_{n,\,m} \cdot \alpha_{m,\,k} $. Let $ m $ be $ k+1 $. Then
\begin{equation} \label{d-n,k recursive form}
\begin{aligned}
\Bd_{n,\,k} &= \Bd_{n,\,k+1} + \Bd_{k+1,\,k} \\
&= \Bd_{n,\,k+2} + \Bd_{k+2,\,k+1} + \Bd_{k+1,\,k} \\
& \hspace{1in} \vdots \\
&= \Bd_{n,\,n-1} + \cdots + \Bd_{k+2,\,k+1} + \Bd_{k+1,\,k} \\
&=  \sum_{i=k}^{n-1} \Bd_{i+1,\,i} .
\end{aligned} \msk
\end{equation}
Moreover, the absolute value each term is super exponentially small. More precisely, each term is bounded by $ \oeps^{2^{\!\:i}} $ for each $ i $, that is, $ \|\;\!\Bd_{i+1,\,i} \| \asymp \|\,\Bq_i (\pi_y(\tau_{i+1}))\| \leq \| D\bde_i \| = O(\bar \eps^{2^{\!\:i}}) $. Then $ \Bd_{n,\,k}^j $ converges to the number, say $ \Bd_{*,\,k}^j $ super exponentially fast for each $ 1 \leq j \leq m $.
\ssk \\
Let us see the recursive formula of $ \Bu_{n,\,k} $
\begin{equation} \label{u-n,k recursive form}
\begin{aligned}
\Bu_{n,\,k} &= \frac{\alpha_{k+1,\,k}}{\si_{k+1,\,k}}\, \Bu_{n,\,k+1} + \Bu_{k+1,\,k} \\
&= \frac{\alpha_{k+1,\,k}}{\si_{k+1,\,k}} \left[\,\frac{\alpha_{k+2,\,k+1}}{\si_{k+2,\,k+1}}\, \Bu_{n,\,k+2} + \Bu_{k+2,\,k+1} \,\right] + \Bu_{k+1,\,k} \\
& \hspace{1in} \vdots \\
&= \sum_{i=k+1}^{n-1}\prod_{j=k}^{i-1} \frac{\alpha_{j+1,\,j}}{\si_{j+1,\,j}}\ \Bu_{i+1,\,i} + \Bu_{k+1,\,k}\\
&= \ \sum_{i=k}^{n-1} \si^{i-k} \Bu_{i+1,\,i}\, (1 + O(\rho^k)) .
\end{aligned} \ssk
\end{equation}
Moreover, $ \Bu_{i+1,\,i}^j \asymp \di_{\Bz_j} \eps_i(\tau_{F_{i+1}}) $ for each $ 1 \leq j \leq m $. Then $ \Bu^j_{n,\,k} $ converges to the number, say $ \Bu^j_{*,\,k} $ for each $ j =1,2,\ldots,m $ super exponentially fast by the similar reason for $ \Bd_{n,\,k}^j $. 
Let us see the recursive formula of $ t_{n,\,k} $ \msk
\begin{equation} \label{t-n,k recursive form}
\begin{aligned}
t_{n,\,k} &= \frac{\alpha_{k+1,\,k}}{\si_{k+1,\,k}}\ t_{n,\,k+1} + t_{k+1,\,k} + \Bu_{k+1,\,k}\cdot \Bd_{n,\,k+1} \\
&= \frac{\alpha_{k+1,\,k}}{\si_{k+1,\,k}} \left[\,\frac{\alpha_{k+2,\,k+1}}{\si_{k+2,\,k+1}}\ t_{n,\,k+2} + t_{k+2,\,k+1} + \Bu_{k+2,\,k+1}\cdot \Bd_{n,\,k+2} \,\right] + t_{k+1,\,k} + \Bu_{k+1,\,k}\cdot \Bd_{n,\,k+1} \\
& \hspace{2in} \vdots \\
&= \sum_{i=k+1}^{n-1}\prod_{j=k}^{i-1} \frac{\alpha_{j+1,\,j}}{\si_{j+1,\,j}}\ t_{i+1,\,i} + t_{k+1,\,k} +  \sum_{i=k+1}^{n-1}\prod_{j=k}^{i-1} \frac{\alpha_{j+1,\,j}}{\si_{j+1,\,j}}\ \Bu_{i+1,\,i}\cdot \Bd_{n,\,i+1} + \Bu_{k+1,\,k}\cdot \Bd_{n,\,k+1} \\
&= \ \sum_{i=k}^{n-1} \si^{i-k} \big[ \,t_{i+1,\,i} + \Bu_{i+1,\,i}\cdot \Bd_{n,\,i+1} \,\big] (1 + O(\rho^k)) .
\end{aligned} \msk
\end{equation}
By the equations \eqref{d-n,k recursive form}, \eqref{u-n,k recursive form} and \eqref{t-n,k recursive form}, we obtain the recursive formula of $ t_{n,\,k}- \Bu_{n,\,k}\cdot \Bd_{n,\,k} $ as follows
\msk
\begin{align*}
& \qquad  t_{n,\,k}- \Bu_{n,\,k}\cdot \Bd_{n,\,k} \\
&= \ \sum_{i=k+1}^{n-1}\prod_{j=k}^{i-1} \frac{\alpha_{j+1,\,j}}{\si_{j+1,\,j}} \big[ \,t_{i+1,\,i} + \Bu_{i+1,\,i}\cdot \Bd_{n,\,i+1} \,\big] + t_{k+1,\,k} + \Bu_{k+1,\,k}\cdot \Bd_{n,\,k+1} \\
& \qquad - \left[\, \sum_{i=k+1}^{n-1}\prod_{j=k}^{i-1} \frac{\alpha_{j+1,\,j}}{\si_{j+1,\,j}}\ \Bu_{i+1,\,i} + \Bu_{k+1,\,k} \right] \cdot \Bd_{n,\,k} \\
&= \ \sum_{i=k+1}^{n-1}\prod_{j=k}^{i-1} \frac{\alpha_{j+1,\,j}}{\si_{j+1,\,j}} \big[\,t_{i+1,\,i} + \Bu_{i+1,\,i}\cdot \Bd_{n,\,i+1} - \Bu_{i+1,\,i} \cdot \Bd_{n,\,k} \,\big] \\
&\qquad + t_{k+1,\,k} + \Bu_{k+1,\,k}\cdot \Bd_{n,\,k+1} - \Bu_{k+1,\,k}\cdot \Bd_{n,\,k} \\[0.3em]
&= \ \sum_{i=k+1}^{n-1}\prod_{j=k}^{i-1} \frac{\alpha_{j+1,\,j}}{\si_{j+1,\,j}} \big[\,t_{i+1,\,i} - \Bu_{i+1,\,i} \cdot \Bd_{i+1,\,k} \,\big] + t_{k+1,\,k} - \Bu_{k+1,\,k}\cdot \Bd_{k+1,\,k} \\
&= \ \ \sum_{i=k}^{n-1} \si^{i-k} \big[ \,t_{i+1,\,i} - \Bu_{i+1,\,i}\cdot \Bd_{i+1,\,k} \,\big] (1 + O(\rho^k)) .
\end{align*} \msk

\nin Recall the expression of the derivative of the coordinate change map at the tip on each level  \ssk
\begin{align*}
(D^{k+1}_k)^{-1} = 
\begin{pmatrix}
(\alpha_k)^{-1} & & \\
& (\si_k)^{-1} & \\
& & (\si_k)^{-1}\cdot \Id_{m \times m}
\end{pmatrix}  \cdot
\begin{pmatrix}
1 & - t_k + \Bu_k \cdot \Bd_k & - \Bu_k \\
& 1 & \\
& -\Bd_k & \Id_{m \times m}
\end{pmatrix} .
\end{align*} \bsk
Since $ H_k(w) = (f_k(x) -\eps_k(w),\, y,\, \Bz - \bde_k(y,f_k^{-1}(y),\B0)) $, \ssk we see that $ \di_y \eps_k(\tau_{F_k}) \asymp - t_k + \Bu_k \cdot \Bd_k $ for every $ k \in \N $. Moreover, the fact that $ t_{i+1,\,i} - \Bu_{i+1,\,i}\cdot \Bd_{i+1,\,i} \asymp \di_y \eps_i(\tau_{F_{i+1}}) $ and $ |\,\Bu_{i+1,\,i}\cdot \Bd_{n,\,i}| $ is super exponentially small for each $ i<n $ implies that $ t_{n,\,k} $ converges to a number, say $ t_{*,\,k} $ super exponentially fast.

\end{proof}

\nin Recall the expression of the map $ \Psi^n_k $ from $ B(R^nF) $ to $ B^{n-k}_{{\bf v}}(R^kF) $ \msk
\begin{align*}
\Psi^n_k(w) = 
\begin{pmatrix}
1 & t_{n,\,k} & \Bu_{n,\,k} \\[0.2em]
& 1 & \\
& \Bd_{n,\,k} & 1
\end{pmatrix}
\begin{pmatrix}
\alpha_{n,\,k} & & \\
& \si_{n,\,k} & \\
& & \si_{n,\,k} \cdot \Id_{m \times m}
\end{pmatrix}
\begin{pmatrix}
x + S^n_k(w) \\
y \\[0.2em]
\Bz + {\bf R}_{n,\,k}(y)
\end{pmatrix} 
\end{align*}
where $ {\bf v} = v^{n-k} \in W^{n-k} $.\msk

\begin{prop} \label{exponential smallness of R-n-k}
Let $ F \in \II_B(\bar \eps) $ and $ \Psi^n_k $ be the map from $ B(R^nF) $ to $ B^(R^kF) $ as the conjugation between $ (R^kF)^{2^{n-k}} $ and $ R^nF $. Let $ {\bf R}_{n,\,k}(y) $ be the non linear part of $ \pi_z \circ \Psi^n_k $ depending on the second variable $ y $. Then both $ {\bf R}_{n,\,k}(y) $ and $ ({\bf R}_{n,\,k})'(y) $ converges to zero exponentially fast as $ n \ra \infty $. In particular,
\begin{equation*}
\begin{aligned}
\| \:\! {\bf R}_{n,\,k} \| \leq C\,\si^{n-k} 
, \qquad \| \:\! ({\bf R}_{n,\,k})' \| \leq C\,\si^{n-k}
\end{aligned}
\end{equation*}
for some $ C>0 $.
\end{prop}
\begin{proof}
Let $ w =(x,y,\Bz) $ be the point in $ B(R^nF) $ and let $ \Psi^n_{n-1}(w) $ be $ w' = (x',y',\Bz') $. Recall $ \Psi^n_k = \Psi^n_{n-1} \circ \Psi^{n-1}_k $. Thus \msk
\begin{equation*}
\begin{aligned}
\Bz' &= \ \pi_z \circ \Psi^n_{n-1}(w) = \si_{n,\,n-1} \,\big[\, \Bd_{n,\,n-1}\, y + \Bz + {\bf R}_{n,\,n-1}(y) \,\big] \\
y' &= \ \pi_y \circ \Psi^n_{n-1}(w) = \si_{n,\,n-1}\,y .
\end{aligned} \msk
\end{equation*}
Then by the similar calculation and the composition of $ \Psi^{n-1}_k $ and $ \Psi^n_{n-1} $, we obtain the recursive formula of $ \pi_z \circ \Psi^n_k $ as follows  \msk
\begin{equation} \label{recursive formula of pi-z Psi}
\begin{aligned}
& \quad \ \  \pi_z \circ \Psi^n_k(w) =  \si_{n,\,k}\,\big[\,\Bd_{n,\,k}\,y + \Bz + {\bf R}_{n,\,k}(y) \,\big] \\[0.3em]
&= \ \pi_z \circ \Psi^{n-1}_k(w') =  \si_{n-1,\,k}\,\big[\,\Bd_{n-1,\,k}\,y' + \Bz' + {\bf R}_{n-1,\,k}(y') \,\big] \\[0.3em]
&= \ \si_{n-1,\,k}\,\big[\,\Bd_{n-1,\,k}\,\si_{n,\,n-1}\,y + \si_{n,\,n-1} \,\big[\, \Bd_{n,\,n-1}\:\! y + \Bz + {\bf R}_{n,\,n-1}(y) \,\big] + {\bf R}_{n-1,\,k}(\si_{n,\,n-1}\,y) \,\big] \\[0.3em]
&= \ \si_{n,\,k}\,( \Bd_{n-1,\,k} + \Bd_{n,\,n-1})y + \si_{n,\,k}\,\Bz + \si_{n,\,k}\,{\bf R}_{n,\,n-1}(y) + \si_{n-1,\,k}\,{\bf R}_{n-1,\,k}(\si_{n,\,n-1}\,y) .
\end{aligned} \msk
\end{equation}
By Proposition \ref{recursive formula of d, u, and t}, $ \Bd_{n,\,k} = \Bd_{n-1,\,k} + \Bd_{n,\,n-1} $. Let us compare the left side of \eqref{recursive formula of pi-z Psi} with the right side of it. Recall the equation $ \si_{n,\,k} = \si_{n,\,n-1} \cdot \si_{n-1,\,k} $. Then
\ssk
\begin{equation} \label{eq-bf R-n-k recursive relation}
\begin{aligned}
{\bf R}_{n,\,k}(y) = {\bf R}_{n,\,n-1}(y) + \frac{1}{\si_{n,\,n-1}}\ {\bf R}_{n-1,\,k}(\si_{n,\,n-1}\,y) .
\end{aligned} \ssk
\end{equation}
\nin The map $ {\bf R}_{n,\,k}(y) $ has the expression of each coordinate maps, that is,
$$ {\bf R}_{n,\,k}(y) = ( R^1_{n,\,k}(y),\; R^2_{n,\,k}(y),\; \ldots ,\; R^m_{n,\,k}(y)) $$
for $ k < n $ and for $ 1 \leq j \leq m $. Each $ R^j_{n,\,k}(y) $ is the sum of second and higher order terms of $ \pi_{z_j} \circ \Psi^n_k $ for $ k < n $. Thus \msk
\begin{equation} \label{eq-component of R-n-k recursive relation}
\begin{aligned}
R^j_{n,\,k}(y) =  a_{n,\,k}^j\,y^2 + A_{n,\,k}^j(y)\cdot y^3
\end{aligned} \msk
\end{equation}

\nin Moreover, $ \|\:\!R^j_{n,\,n-1} \| = O(\bar \eps^{2^{n-1}}) $ because $ R^j_{n,\,n-1}(y) $ is the second and higher order terms of the map $ \pi_{z_j} \circ \bde_{n-1}(\si_{n,\,n-1}\,y,\;f_{n-1}^{-1}(\si_{n,\,n-1}\,y),\;\B0) $. The equation \eqref{eq-bf R-n-k recursive relation} is applies to each coordinate maps. Then \msk
\begin{equation*}
\begin{aligned}
R^j_{n,\,k}(y) = \frac{1}{\si_{n,\,n-1}}\ R^j_{n-1,\,k}(\si_{n,\,n-1}\,y) + c^j_{n,\,k}\,y^2 + O(\bar \eps^{2^{n-1}}y^3)
\end{aligned} \ssk 
\end{equation*}
where $ c^j_{n,\,k} = O(\bar \eps^{2^{n-1}}) $. The recursive formula for $ a_{n,\,k}^j $ and $ A_{n,\,k}^j $ as follows \msk
\begin{equation*}
\begin{aligned}
R^j_{n,\,k}(y) = \frac{1}{\si_{n,\,n-1}}\ \Big( a_{n-1,\,k}^j\cdot (\si_{n,\,n-1}\,y)^2 + A_{n-1,\,k}^j(\si_{n,\,n-1}\,y) \cdot (\si_{n,\,n-1}\,y)^3 \Big) + O(\bar \eps^{2^{n-1}}y^3) .
\end{aligned} \msk
\end{equation*}
Then \ssk $ a_{n,\,k}^j = \si_{n,\,n-1}\,a_{n-1,\,k}^j + c_{n,\,k}^j $ and $ \| \:\! A_{n,\,k}^j \| \leq \| \,\si_{n,\,n-1}\|^2 \| \:\! A_{n-1,\,k}^j \| + O(\bar \eps^{2^{n-1}}) $. Hence, for each fixed $ k<n $,\, $ a_{n,\,k}^j \ra 0 $ and $ A_{n,\,k}^j \ra 0 $ exponentially fast as $ n \ra \infty $. \ssk $ R^j_{n,\,k}(y) $ converges to zero as $ n \ra \infty $ exponentially fast. 
\ssk \\
Let us estimate $ \| \:\! (A^j_{n,\,k})' \| $ in order to measure how fast $ (R^j_{n,\,k})'(y) $ is convergent. By the similar method of the recursive formula of $ R^j_{n,\,k}(y) $, we have the expression and recursive formula of $ (R^j_{n,\,k})'(y) $ as follows \msk
\begin{equation*}
\begin{aligned}
(R^j_{n,\,k})'(y) &= \ 2\, a_{n,\,k}^j\,y + 3 \:\! A_{n,\,k}^j(y)\cdot y^2 + (A^j_{n,\,k})'(y)\cdot y^3 \\[0.5em]
  (R^j_{n,\,k})'(y) &= \ (R^j_{n,\,n-1})'(y) + R^j_{n-1,\,k}(\si_{n,\,n-1}\,y) \\[0.2em]
&= \ R^j_{n-1,\,k}(\si_{n,\,n-1}\,y) + 2\, c_{n,\,k}^j\,y + O(\bar \eps^{2^{n-1}}y^2) .
\end{aligned} \msk
\end{equation*}
Then
\begin{equation*}
\begin{aligned}
(R^j_{n,\,k})'(y) &= \ 2\, a^j_{n-1,\,k}\,\si_{n,\,n-1}\,y + 3 \:\! A_{n-1,\,k}^j(\si_{n,\,n-1}\,y)\cdot (\si_{n,\,n-1}\,y)^2 \\[0.3em]
 & \qquad + (A^j_{n,\,k})'(\si_{n,\,n-1}\,y)\cdot (\si_{n,\,n-1}\,y)^3 + 2\, c_{n,\,k}^j\,y + O(\bar \eps^{2^{n-1}}y^2) .
\end{aligned} \msk
\end{equation*}
Let us compare quadratic and higher order terms of $ (R^j_{n,\,k})'(y) $ \ssk
\begin{equation*}
\begin{aligned}
3 \, A_{n,\,k}^j(y)\cdot y^2 + (A^j_{n,\,k})'(y)\cdot y^3 &= \ 3 \:\! A_{n-1,\,k}^j(\si_{n,\,n-1}\,y)\cdot (\si_{n,\,n-1}\,y)^2 \\[0.2em]
 & \qquad + (A^j_{n,\,k})'(\si_{n,\,n-1}\,y)\cdot (\si_{n,\,n-1}\,y)^3 + O(\bar \eps^{2^{n-1}}y^2) .
\end{aligned} 
\end{equation*}
Thus
\begin{equation*}
\begin{aligned}
(A^j_{n,\,k})'(y)\,y &= (A^j_{n,\,k})'(\si_{n,\,n-1}\,y)\cdot \si_{n,\,n-1}^3 \,y - 3 \:\! A^j_{n,\,k}(y) + 3 \:\! A^j_{n-1,\,k}(\si_{n,\,n-1}\,y)\cdot \si_{n,\,n-1}^2 \\[0.2em]
&\qquad + O(\bar \eps^{2^{n-1}}) .
\end{aligned} \msk
\end{equation*}
Then 
\begin{equation*}
\begin{aligned}
\| \:\! (A^j_{n,\,k})' \| &\leq \ \| \:\! (A^j_{n-1,\,k})' \| \cdot \|\:\! \si_{n,\,n-1} \|^3 + 3 \| \:\! A^j_{n,\,k}\| + 3 \| \:\! A^j_{n-1,\,k}\| \cdot \|\, \si_{n,\,n-1} \|^2 + O(\bar \eps^{2^{n-1}}) \\[0.3em]
&\leq \ \| \,(A^j_{n-1,\,k})' \| \cdot \|\, \si_{n,\,n-1} \|^3 + C \|\, \si_{n,\,n-1} \|^2
\end{aligned} \msk
\end{equation*}
for some $ C>0 $. Then $ (A^j_{n,\,k})' \ra 0 $ as $ n \ra \infty $ exponentially fast. Then so does $ (R^j_{n,\,k})'(y) $ exponentially fast for all $ 1 \leq j \leq m $. Furthermore, since $ |\:\!a_{n,\,k}^j |$,  $ \| \:\! A_{n,\,k}^j \|$ and $ \|\:\! (A^j_{n,\,k})'\| $ are bounded above by $ C\,\si^{n-k} $ for some $ C> 0 $ and for every $ 1 \leq j \leq m $, so does the norm of $ {\bf R}_{n,\,k} $ and $ ({\bf R}_{n,\,k})' $, that is,
\begin{equation*}
\begin{aligned}
\| \:\! {\bf R}_{n,\,k} \| \leq C\,\si^{n-k} 
, \qquad \| \:\! ({\bf R}_{n,\,k})' \| \leq C\,\si^{n-k} .
\end{aligned}
\end{equation*}
\end{proof}
\msk
\nin Let $ w^1 $ and $ w^2 $ be two points in $ B(R^nF) $ and $ w^j = (x^j,\, y^j,\, \Bz^j) $ for $ j =1,2 $ for the next Proposition. Let $ \Psi^n_{i,\, {\bf v}}(w^j) = w_i^j $ for $ i \in \N $ and $ j =1,2 $. In particular, let $ \dot w^j = (\dot x^j,\, \dot y^j,\, \dot \Bz^j) $ be $ \Psi^n_k(w^j) $ for $ j = 1,2 $.
\msk

\begin{prop} \label{formal expression of pi-z Psi difference}
Let $ F \in \II_B(\bar \eps) $. 
Then 
\begin{equation*}
\begin{aligned}
\dot \Bz^1 - \dot \Bz^2 = \pi_{\Bz} \circ \Psi^n_k(w^1) - \pi_{\Bz} \circ \Psi^n_k(w^2) =  \si_{n,\,k} \cdot (\Bz^1 - \Bz^2) + \si_{n,\,k} \sum_{i=k}^{n-1} \Bq_i (\si_{n,\,i} \,\bar y) \cdot (y^1 -y^2)
\end{aligned}
\end{equation*}
where 
$ \bar y $ is in the line segment between $ y_1 $ and $ y_2 $. Moreover,
\begin{equation*}
\begin{aligned}
\sum_{i=k}^{n-1} \Bq_i \circ (\si_{n,\,i} \,\bar y) \cdot (y^1 -y^2) = \Bd_{n,\,k} \cdot (y^1 -y^2) + {\bf R}_{n,\,k}(y^1) - {\bf R}_{n,\,k}(y^2) .
\end{aligned}
\end{equation*}
\end{prop}

\begin{proof}
Firstly, \ssk let us express $ \pi_{\Bz} \circ \Psi^n_k(w) $. Let $ \Bp_{\,i}(y) $ be $ \bde_i (y, f^{-1}_i(y),\B0) $ in order to simplify the expression. Recall the definition of $ \Bq_i(y) $, namely, $ \frac{d}{dy}\,\Bp_{\,i}(y) = \Bq_i(y) $. \ssk 
Let $ \Psi^n_{i,\, {\bf v}}(w) = w_i $ for $ k \leq i \leq n-1 $ and let $ w_i = (x_i,\, y_i,\, \Bz_i) $.\footnote{For notational compatibility, let $\Psi^i_i = \id $ and let $ \si_{i,\, i} =1 $ for every $ i \in \N $.} \ssk Let $ w = w_n $. Recall $ \pi_{\Bz} \circ \psi^{i+1}_i (w_{i+1}) = \si_i\, \Bz_{i+1} + \Bp_{\,i} (\si_i\, y_{i+1}) $. Since $ \Psi^n_k = \psi^{k+1}_k \circ \Psi^n_{k+1} $, we estimate $ \Bz_k $ using recursive formula \msk
\begin{equation} \label{pi-z Psi-n,k expression}
\begin{aligned}
\Bz_k &= \ \pi_{\Bz} \circ \Psi^n_k(w) = \pi_{\Bz} \circ \psi^{k+1}_k(w_{k+1}) \\[0.3em]
&= \ \si_k \cdot \Bz_{k+1} + \Bp_{\,k} (\si_k \cdot y_{k+1}) \\[0.3em]
&= \ \si_k \big[\,\si_{k+1} \cdot \Bz_{k+2} + \Bp_{\,k+1} (\si_{k+1} \cdot y_{k+2})\,\big] + \Bp_{\,k}(\si_k \cdot y_{k+1}) \\[0.3em]
&= \ \si_k \si_{k+1} \cdot \Bz_{k+2} + \si_k \cdot \Bp_{\,k+1} (\si_{k+1} \cdot y_{k+2}) + \Bp_{\,k}(\si_k \cdot y_{k+1}) \\
&\hspace{2in} \vdots \\
&= \ \si_k \si_{k+1} \cdots \si_{n-1} \cdot \Bz + \big[\, \si_k \si_{k+1} \cdots \si_{n-2} \cdot \Bp_{\,n-1}(\si_{n-1} \cdot y) \\
&\qquad + \si_k \si_{k+1} \cdots \si_{n-3} \cdot \Bp_{\,n-2}(\si_{n-2} \cdot y_{n-1}) + \cdots + \Bp_{\,k}(\si_k \cdot y_{k+1}) \,\big] \\[0.6em]
&= \ \si_{n,\, k} \cdot \Bz + \si_{n-1,\, k} \cdot \Bp_{\,n-1}(\si_{n-1} \cdot y) + \si_{n-2,\, k} \cdot \Bp_{\,n-2}(\si_{n-2} \cdot y_{n-1}) + \cdots + \Bp_{\,k}(\si_k \cdot y_{k+1}) \\
&= \ \si_{n,\, k} \cdot \Bz + \sum_{i=k}^{n-1} \si_{i,\, k} \cdot \Bp_{\,i}\:\!(\si_i \cdot y_{i+1})
\end{aligned}
\end{equation}
where $ \si_{k+1,\, k} = \si_k $. 
Moreover, 
$ H_i \circ \La_i(w) = (\phi_i^{-1}(\si_i w),\ \si_i\, y,\ \bullet ) $ \ for each $ k \leq i \leq n-1 $.  
Thus
\begin{equation*}
\begin{aligned}
\si_{n,\, i} \cdot y = \si_i \cdot y_{i+1} = y_i 
\end{aligned} \msk
\end{equation*} 

\nin Secondly, let us estimate \ssk $ \dot \Bz^1 - \dot \Bz^2 = \pi_{\Bz} \circ \Psi^n_k(w^1) - \pi_{\Bz} \circ \Psi^n_k(w^2) $ where $ w^j \in B(R^nF) $ for $ j=1,2 $. By the equation \eqref{pi-z Psi-n,k expression} and Mean Value Theorem, we obtain that \msk
\begin{equation} \label{difference of z-1 and z-2}
\begin{aligned}
\dot \Bz^1 - \dot \Bz^2 &= \ \pi_{\Bz} \circ \Psi^n_k(w^1) - \pi_{\Bz} \circ \Psi^n_k(w^2) \\[0.3em]
&= \ \si_{n,\,k} \cdot (\Bz^1 - \Bz^2) + \sum_{i=k}^{n-1} \si_{i,\, k} \cdot \big[\,\Bp_{\,i}\:\!(\si_i \cdot y_{i+1}^1) - \Bp_{\,i}\:\!(\si_i \cdot y_{i+1}^2)\,\big]  \\
&= \ \si_{n,\, k} \cdot (\Bz^1 - \Bz^2) + \sum_{i=k}^{n-1} \si_{i,\, k} \cdot \big[\,\Bp_{\,i}\:\!(\si_{n,\, i} \cdot y^1) - \Bp_{\,i}\:\!(\si_{n,\, i} \cdot y^2)\,\big]  \\
&= \ \si_{n,\, k} \cdot (\Bz^1 - \Bz^2) + \sum_{i=k}^{n-1} \si_{i,\, k} \cdot \Bq_{\,i}\circ (\si_{n,\, i}\cdot \bar y )  \cdot \si_{n,\, i+1} \cdot ( y^1 - y^2 ) \\[-0.3em]
&= \ \si_{n,\, k} \cdot (\Bz^1 - \Bz^2) + \si_{n,\, k}  \sum_{i=k}^{n-1} \Bq_{\,i}\circ (\si_{n,\, i}\cdot \bar y )  \cdot ( y^1 - y^2 )
\end{aligned}
\end{equation}
where $ \bar y $ is in the line segment between $ y^1 $ and $ y^2 $ which is contained in $ \pi_y \circ B(R^nF) $. Moreover, by the expression of $ \Psi^n_k $, 
$$ \pi_{\Bz} \circ \Psi^n_k(w) = \si_{n,\, k}\:\! \big[\,\Bd_{n,\, k}\, y + z + {\bf R}_{n,\,k}(y) \,\big] . $$
Then
\begin{equation} \label{difference of z-1 and z-2 2}
\begin{aligned}
\dot \Bz^1 - \dot \Bz^2 &= \ \pi_{\Bz} \circ \Psi^n_k(w^1) - \pi_{\Bz} \circ \Psi^n_k(w^2) \\[0.3em]
&= \ \si_{n,\, k}\, \big[\,\Bd_{n,\, k}\, (y^1 -y^2) + (\Bz_1 - \Bz_2) + {\bf R}_{n,\,k}(y^1) -  {\bf R}_{n,\,k}(y^2) \,\big] \\[0.3em]
&= \ \si_{n,\, k} \cdot (\Bz^1 - \Bz^2) + \si_{n,\, k} \cdot \big[\,\Bd_{n,\, k}\, (y^1 -y^2) + {\bf R}_{n,\,k}(y^1) -  {\bf R}_{n,\,k}(y^2) \,\big] .
\end{aligned}
\end{equation}
Compare equations \eqref{difference of z-1 and z-2} and \eqref{difference of z-1 and z-2 2}. Then
$$  \sum_{i=k}^{n-1} \Bq_{\,i}\circ (\si_{n,\, i}\cdot \bar y )  \cdot ( y^1 - y^2 ) = \Bd_{n,\, k}\, (y^1 -y^2) + {\bf R}_{n,\,k}(y^1) -  {\bf R}_{n,\,k}(y^2) $$
for $ k < n $. The proof is complete.
\end{proof}
\msk

\begin{cor} \label{distortion parts of z-coordinate of Psi}
Let $ F \in \II_B(\bar \eps) $. Then
\begin{equation*}
\sum_{i=k}^{n-1} \Bq_i \circ (\pi_y \circ \Psi^n_{i,\,{\bf v}}(w)) = \Bd_{n,\,k} + ({\bf R}_{n,\,k})'(y)
\end{equation*}
for every $ w \in B(R^nF) $ and for each $ k < n $. Moreover,
$$ \lim_{n \ra \infty}\; \sum_{i=k}^{n-1} \Bq_i \circ (\pi_y \circ \Psi^n_{i,\,{\bf v}}(w)) = \Bd_{*,\,k} . $$
The convergence is exponentially fast.
\end{cor}

\begin{proof}
Let us compare the equation \eqref{difference of z-1 and z-2} and \eqref{difference of z-1 and z-2 2}

\begin{align}
\sum_{i=k}^{n-1} \si_{i,\, k} \cdot \big[\,\Bp_{\,i}\,(\si_{n,\, i} \cdot y^1) - \Bp_{\,i}\,(\si_{n,\, i} \cdot y^2)\,\big] &\ = \ \si_{n,\, k} \cdot \big[\,\Bd_{n,\, k}\, (y^1 -y^2) + {\bf R}_{n,\,k}(y^1) -  {\bf R}_{n,\,k}(y^2) \,\big] \\ 
\label{slope of secant lines}
\sum_{i=k}^{n-1} \si_{i,\ k} \cdot \frac{\Bp_{\,i}\,(\si_{n,\,i} \cdot y^1) - \Bp_{\,i}\,(\si_{n,\,i} \cdot y^2)}{y^1 -y^2} 
 &\ = \ \si_{n,\, k} \cdot \left[\,\Bd_{n,\, k} + \frac{{\bf R}_{n,\,k}(y^1) - {\bf R}_{n,\,k}(y^2)}{y^1 -y^2}  \,\right] .
\end{align}
Since both $ y^1 $ and $ y^2 $ are arbitrary, we may choose two points $ y $ and $ y + h $ instead of $ y^1 $ and $ y^2 $. The differentiability of both $ \Bp_i $ and $ {\bf R}_{n,\,k} $ enable us to take the limit of \eqref{slope of secant lines} as $ h \ra 0 $. Then 
\begin{equation*}
\si_{n,\, k} \cdot \sum_{i=k}^{n-1} \Bq_{\,i}\circ (\si_{n,\,i} \cdot y ) =  \si_{n,\, k} \cdot \big[\,\Bd_{n,\, k} + ({\bf R}_{n,\,k})'(y)\,\big]
\end{equation*}
for every $ y \in \pi_y(B(R^nF)) $. Moreover, $ \Bd_{n,\,k} \ra \Bd_{*,\,k} $ as $ n \ra \infty $ super exponentially fast by Proposition \ref{recursive formula of d, u, and t} and $ ({\bf R}_{n,\,k})' $ converges to zero exponentially fast by Proposition \ref{exponential smallness of R-n-k}. Hence,
\begin{equation*}
\begin{aligned}
\lim_{n \ra \infty} \;\sum_{i=k}^{n-1} \Bq_i \circ (\pi_y \circ \Psi^n_{i,\,{\bf v}}(w)) &= \ \lim_{n \ra \infty} \big[\,\Bd_{n,\,k} + ({\bf R}_{n,\,k})'(y)\,\big] \\
 &= \ \Bd_{*,\,k} .
\end{aligned}
\end{equation*}
\end{proof}
\msk
\nin Let us collect the estimations of numbers and functions which are used in the following sections. \msk
\begin{enumerate}
\item $ |\:\! t_{k+1,\,k} | $, $ \|\:\! \Bu_{k+1,\,k} \| $ and $ \|\:\! \Bd_{k+1,\,k} \| $ are $ O \big(\bar \eps^{2^k} \big) $. \msk
\item $ \|\:\! \Bu_{n,\,k} \| $ and $ \|\:\! \Bd_{n,\,k} \| $ are $ O \big(\bar \eps^{2^k} \big) $. \msk
\item $ \si_{n,\,k} = (-\si)^{n-k}(1+O(\rho^k)) $ and $ \alpha_{n,\,k} = \si^{2(n-k)}(1+O(\rho^k)) $ for $ k<n $ and for some $ 0<\rho<1 $. \msk
\item $ \|\:\! {\bf R}_{k+1,\,k} \| $ is $ O \big(\bar \eps^{2^k} \big) $. \msk
\item $ \|\:\! {\bf R}_{n,\,k} \| $ and $ \|\:\! ({\bf R}_{n,\,k})' \| $ are $ O \big(\si^{n-k} \bar \eps^{2^k} \big) $ by Proposition \ref{exponential smallness of R-n-k}.
\end{enumerate}

\bsk

\section{Unbounded geometry on the Cantor set}
The unbounded geometry of a certain class of $ m+2 $ dimensional H\'enon-like maps would be proved. 
The notations in this section is used in \cite{HLM} and adapted in $ m+2 $ dimensional maps. 
%
Recall the {\em pieces} $B^n_{\bf w} \equiv B^n_{\bf w}(F) = \Psi^n_{\bf w}(B)$ on the $n^{th}$ level or $n^{th}$ generation 
. The word, $ {\bf w} = (w_1 \ldots  w_n) \in W^n := \lbrace v, c \rbrace^n  $ has length $ n $. Recall that the map 
$$ {\bf w} = (w_1 \ldots  w_n) \mapsto \sum_{k=0}^{n-1} w_{k+1}2^k $$
is the one to one correspondence between words of length $ n $ and the additive group of numbers with base 2 mod $ 2^n $. Let the subset of the critical Cantor set on each pieces be $ \OO_{\bf w} \equiv B^n_{\bf w} \cap \OO $. Then by the definition of $ \OO_{\bf w} $, 
we have the following fact.\ssk
\begin{enumerate}
\item $$ \OO_F = \bigcup_{ {\bf w} \in W^n}  \OO_{\bf w} $$
\item $ F(B^n_{\bf w}) \subset B^n_{{\bf w}+1} $  for every $ {\bf w} = (w_1 \ldots  w_n) \in W^n $.\msk
\item $ \diam(B^n_{\bf w}) \leq C \si^n $ for some $ C>0 $ depending only on $ B $ and $ \bar \eps $.\ssk
\end{enumerate}
\nin Let the minimal distance between two boxes $ B_1, B_2 $ be the infimum of the distance between all elements of each boxes and express this distance to be $ \dist_{\min}(B_1, B_2) $. \msk
\begin{defn}
The map $ F \in \II_B(\oeps) $ has the {\em bounded geometry} if
\begin{align*}
 \dist_{\min}(B^{n+1}_{{\bf w}v}, B^{n+1}_{{\bf w}c}) &\asymp \diam(B^{n+1}_{{\bf w}\nu}) \quad \text{for} \ \nu \in \{v, c\} \\[0.2em]
 \diam(B^n_{\bf w}) &\asymp \diam(B^{n+1}_{{\bf w}\nu}) \quad \text{for} \ \nu \in \{v, c\}
\end{align*}
for all $ {\bf w} \in W^n $ and for all $ n \geq 0 $. 
\end{defn}
\nin 
If $ F $ does not have bounded geometry, then we call $ \OO_F $ has {\em unbounded geometry}.
\msk

\subsection{Horizontal overlap of two adjacent boxes} \label{subsec-Horizontal overlap}
The proof of unbounded geometry of the Cantor set requires to compare the diameter of box and the minimal distance of two adjacent boxes. In order to compare these quantities, we would use the maps, $ \Psi^n_k(w) $ and $ F_k(w) $ with the two points $ w_1 = (x_1, y_1, \Bz_1) $ and $ w_2 = (x_2, y_2, \Bz_2) $ in $ F_n(B) $. Recall that $ \Bz_j = (z^1_j, z_j^2, \ldots , z_j^m) $ for $ j = 1,2 $. Let us each successive image of $ w_j $ under $ \Psi^n_k(w) $ and $ F_k(w) $ be $ \dot{w}_j $, $ \ddot{w}_j $ and $ \dddot{w}_j $ for $ j=1,2 $.
\begin{displaymath}
\xymatrix  
{   { w_j}  \ar@{|->}[r]^ {\Psi^n_k} & \dot{w}_j  \ar@{|->}[r]^{F_k}  &\ddot{w}_j \ar@{|->}[r]^{ \Psi^k_0}     & \dddot{w}_j
   }
\end{displaymath}

\nin For example, $ \dot{w}_j = \Psi^n_k(w_j) $ and $ \dot{w}_j = (\dot x_j, \dot y_j, \dot \Bz_j ) $ for  $ j=1,2 $. \ssk Let $ S_1 $ and $ S_2 $ be the (path) connected set on $ \R^{m+2} $. If $ \pi_x(\overline{S_1}) \cap \pi_x(\overline{S_2}) $ contains at least two points, then this intersection is called the {\em $ x- $axis overlap} or {\em horizontal overlap} of $ S_1 $ and $ S_2 $. Moreover, we say $ S_1 $ {\em overlaps} $ S_2 $ on the $ x- $axis or horizontally. Recall $ \si $ is the linear scaling of $ F_* $, the fixed point of the renormalization operator and $ \si_k = \si ( 1 + O(\rho^k)) $ for each $ k \in \N $.
\ssk \\
Recall the map $ \Psi^n_k $ from $ B(R^nF) $ to $ B^{n}_{{\bf v}}(R^kF) $ where $ {\bf v} = v^{n-k} \in W^{n-k} $
\msk
\begin{equation*}
\begin{aligned}
\Psi^n_k(w) = 
\begin{pmatrix}
1 & t_{n,\,k} & \Bu_{n,\,k} \\[0.2em]
& 1 & \\
& \Bd_{n,\,k} & \Id_{m \times m}
\end{pmatrix}
\begin{pmatrix}
\alpha_{n,\,k} & & \\
& \si_{n,\,k} & \\
& & \si_{n,\,k} \cdot \Id_{m \times m}
\end{pmatrix}
\begin{pmatrix}
x + S^n_k(w) \\
y \\[0.2em]
\Bz + {\bf R}_{n,\,k}(y)     
\end{pmatrix} 
\end{aligned} \msk
\end{equation*}
where $ \alpha_{n,\,k} = \si^{2(n-k)}(1 + O(\rho^k)) $ and $ \si_{n,\,k} = (-\si)^{n-k}(1 + O(\rho^k)) $. Thus for any $ w \in B(R^nF) $ we have the following equation
\begin{equation*}
\pi_x \circ \Psi^n_k(w) = \alpha_{n,\,k} (x + S^n_k(w)) + \si_{n,\,k} \big( t_{n,\,k}\,y + \Bu_{n,\,k} \cdot (\Bz + {\bf R}_{n,\,k}(y)) \big) .
\end{equation*}

\nin Let us find the sufficient condition of the horizontal overlapping. Horizontal overlapping means that there exist two points $ w_1 \in B^1_v(R^nF) $ and $ w_2 \in B^1_c(R^nF) $ satisfying the equation
\begin{equation*}
\pi_x \circ \Psi^n_k(w_1) - \pi_x \circ \Psi^n_k(w_2) = 0 .
\end{equation*}
Equivalently,
\begin{equation} \label{equation of x-axis overlap 1}
\begin{aligned}
&\ \alpha_{n,\,k} \Big[ \big(x_1 + S^n_k(w_1) \big) - \big(x_2 + S^n_k(w_2)\big) \Big] \\
& \qquad 
+ \si_{n,\,k} \Big[\, t_{n,\,k}(y_1 - y_2) + \Bu_{n,\,k} \cdot \big \{ \;\!\Bz_1 - \Bz_2 + {\bf R}_{n,\,k}(y_1) - {\bf R}_{n,\,k}(y_2) \big \} \Big] =0 .
\end{aligned} \msk
\end{equation} 

\nin Recall that $ x + S^n_k(w) = v_*(x) + O( \bar \eps^{2^k} + \rho^{n-k}) $ for some $ 0 < \rho < 1 $. 
Since the universal map $ v_*(x) $ is a diffeomorphism and $ | \:\! x_1 - x_2 | = O(1) $, we have the estimation by mean value theorem
$$ \big|\, x_1 + S^n_k(w_1) - \big(x_2 + S^n_k(w_2)\big) \big| = O(1) . $$
Recall that $ \tau_i $ is the tip of $ F_i $ for $ i \in \N $.
\msk
\begin{prop} \label{prop-estimation of t-n-k by b-1}
Let $ F \in {\NN} \cap \II_B(\bar \eps) $. Let $ b_1 = b/b_{\Bz} $ where $ b $ is the average Jacobian of $ F $ and $ b_{\Bz} $ is the universal number defined in Proposition \ref{prop-universal number b-Z}. Then 
$$ b_1^{2^k} \asymp t_{n,\,k} $$
for $ n > k + A $ where $ A $ is a uniform constant depending only on $ b_1 $ and $ \oeps $.
\end{prop}
\begin{proof}
Plugging $ \tau_n $ into the equation \eqref{eq-b1-2-k with partial derivatives of de-k} in Lemma \ref{lem-di-y eps and q asymptotic from n to k}. Then \msk
\begin{equation} \label{eq-asymptotic of di-y eps-k with b-1}
\begin{aligned}
\di_y \eps_k(\tau_k) + E_k(\tau_k) \cdot \sum_{i=k}^{n-1} \; \Bq_i \circ \big( \pi_y (\tau_i) \big) - E_k(\tau_k) \cdot (Z_n(\tau_n))^{-1} \cdot Y_n(\tau_n)= b_1^{2^k} a(\pi_x(\tau_k))(1 + O(\rho^k)) 
\end{aligned}
\end{equation}
where $ x \mapsto a(x) $ is the universal map. Recall that $ H_k^{-1} \circ \La_k^{-1} = \Psi^{k+1}_k $. Since $ F \in \NN $, we have \msk
\begin{equation*}
\begin{aligned}
(Z_n(\tau_n))^{-1} \cdot Y_n(\tau_n) + X_n(c_{F_n}) 
& = \ (Z_n(\tau_n))^{-1} \cdot \Big[\,Y_n(\tau_n) + Z_n(\tau_n) \cdot X_n(c_{F_n}) \,\Big] \\[0.3em]
& = \ \B0
\end{aligned} \ssk
\end{equation*}
Thus $ \| \:\! E_k(\tau_k) \| \cdot \| \:\! (Z_n(\tau_n))^{-1} \cdot Y_n(\tau_n) \| \leq \| \:\! E_k(\tau_k) \| \cdot \| \:\! X_n(c_{F_n}) \| = O(\oeps^{2^k}\oeps^{2^n}) $. \ssk Let us find the sufficient condition satisfying $ \bar \eps^{2^k} \bar \eps^{2^n} \lesssim b_1^{2^k} $. If $ b_1 \geq \bar \eps^2 $, then $ \bar \eps^{2^k} \bar \eps^{2^n} \leq b_1^{2^k} $ for $ n > k $. Assume that $ b_1 < \bar \eps^2 \ll 1 $. Thus
\begin{equation*}
\begin{aligned}
\bar \eps^{2^n} \bar \eps^{2^k} \lesssim  b_1^{2^k} &\Longleftrightarrow \ ( 2^n + 2^k) \log \bar \eps \ \lesssim \ 2^k \log b_1 \\[0.3em]
&\Longleftrightarrow \ 2^n \geq 2^k \left( \frac{\log b_1}{ \log \bar \eps} - 1 \right) + C_0
\end{aligned}
\end{equation*}
for some positive $ C_0>0 $. 
Define $ A $ as follows \msk
\begin{equation}
\begin{aligned}
A = \left\{
\begin{array}{cll}
0 & \quad \text{where} & b_1 \geq \bar \eps^2 \\[0.6em]
C_1\, \log_2 \! \left( \dfrac{\log b_1}{ \log \bar \eps} - 1 \right) & \quad \text{where} & b_1 < \bar \eps^2 
\end{array} \right.
\end{aligned} \msk
\end{equation}
Thus if $ n \geq k + A $, then $ \bar \eps^{2^k} \bar \eps^{2^k} \lesssim b_1^{2^k} $. \ssk Let us compare each components of the derivatives of $ D(\Psi^{k+1}_k)^{-1}(\tau_k) $ and $ (D^{k+1}_k)^{-1} = D(\La_k \circ H_k)(\tau_k) $. Then comparison of each elements of the matrices shows that
\msk
\begin{equation}
\begin{aligned}
t_{k+1,\,k} - \Bu_{k+1,\,k}\cdot  \Bd_{k+1,\,k} & \ = \ \frac{\alpha_{k+1,\,k}}{\si_{k+1,\,k}} \cdot \di_y \eps_k(\tau_k) \\[0.2em]
\Bu_{k+1,\,k} & \ = \ \frac{\alpha_{k+1,\,k}}{\si_{k+1,\,k}} \cdot E_k(\tau_k), \qquad 
\Bd_{i+1,\,i} \ = \ \Bq_i \circ \big( \pi_y (\tau_i) \big)  .
\end{aligned} \msk
\end{equation}
The equation, $ \displaystyle{\Bd_{n,\,k+1} = \sum_{i=k+1}^{n-1} \Bd_{i+1,\,i}} $ holds by Proposition \ref{recursive formula of d, u, and t}. Then
\begin{equation}
\begin{aligned}
t_{k+1,\,k} + \Bu_{k+1,\,k}\cdot \Bd_{n,\,k+1} &= \ t_{k+1,\,k} - \Bu_{k+1,\,k}\cdot \Bd_{k+1,\,k} + \Bu_{k+1,\,k}\cdot \Bd_{n,\;k} \\[0.2em]
&= \ t_{k+1,\,k} - \Bu_{k+1,\,k}\cdot \Bd_{k+1,\,k} + \Bu_{k+1,\,k} \sum_{i=k}^{n-1} \Bd_{i+1,\,i} \\[-0.4em]
&= \ \frac{\alpha_{k+1,\,k}}{\si_{k+1,\,k}} \,\Big[\, \di_y \eps_k(\tau_k) + E_k(\tau_k) \cdot \sum_{i=k}^{n-1} \; \Bq_i \circ \big( \pi_y (\tau_i) \big) \,\Big] \\[0.2em]
&= \ b_1^{2^k} \cdot \si \cdot a(\pi_x(\tau_k))(1 + O(\rho^k)) .
\end{aligned} \msk
\end{equation}
Hence, by Proposition \ref{recursive formula of d, u, and t} again, $ t_{n,\,k} $ is as follows
\begin{equation*}
\begin{aligned}
t_{n,\,k} & \asymp \sum_{i=k}^{n-1} \si^{i-k}\,\big[\,t_{i+1,\,i} + \Bu_{i+1,\,i} \cdot \Bd_{n,\,i+1} \,\big](1 + O(\rho^k)) \\[0.2em]
& \asymp t_{k+1,\,k} + \Bu_{k+1,\,k} \cdot \Bd_{n,\,k+1} \\[0.4em]
& \asymp b_1^{2^k} .
\end{aligned}
\end{equation*}

\end{proof}

\msk
\nin Let us choose two points $ w'_1 $ and $ w'_2 $ in $ F_n(B) $. In particular, we may assume that 
$ w_j \in \OO_{R^nF} $ where $ \pi_z(w_j) =z_j $ for $ j=1,2 $. 
Let $ w'_1 $ and $ w'_2 $ be the pre-image of $ w_1 $ and $ w_2 $ respectively. Then 
\begin{equation} \label{z distance of two points}
\begin{aligned}
| \:\! z_1^j -z_2^j | = | \,\de_n^j(w_1') - \de_n^j(w_2') | \leq  C_j\:\!\| D \de_n^j \| \cdot \| \;\! w'_1 - w'_2 \| = O(\,\bar \eps^{2^n})
\end{aligned} \msk
\end{equation}
for some $ C_j>0 $ and for $ 1 \leq j \leq m $. Thus $ \| \;\! \Bz_1 - \Bz_2 \| = O(\bar \eps^{2^n}) $.
\msk
\begin{cor} \label{cor-comparison sigma n-k with b-1}
Let $ F \in \NN \cap \II_B(\bar \eps) $. \ssk Suppose that $ \Psi^n_k(B^{n+1}_v) $ overlaps $ \Psi^n_k(B^{n+1}_c) $ on $ x- $axis. In particular, \ssk $ \pi_x \circ \Psi^n_k(B^{n+1}_v) \cap \pi_x \circ \Psi^n_k(B^{n+1}_c) $ contains two points $ \Psi^n_k(w_1) $ and $ \Psi^n_k(w_2) $ where $ w_1 \in B^{n+1}_v \cap \OO_{R^nF} $ and $ w_2 \in B^{n+1}_c \cap \OO_{R^nF} $. 
Then 
$$ \si^{n-k} \asymp b_1^{2^k}  $$
for every big enough $ k $.
\end{cor}
\begin{proof}
Recall the following equation
\msk
\begin{equation*}
\begin{aligned}
\pi_x \circ \Psi^n_k(w) = \alpha_{n,\,k} \big[\,x + S^n_k(w)\,\big] + \si_{n,\,k} \big[\, t_{n,\,k}\,y + \Bu_{n,\,k} \cdot (\Bz + {\bf R}_{n,\,k}(y)) \,\big] .
\end{aligned} \msk
\end{equation*}
Recall 
$ x + S^n_k(w) = v_*(x) + O(\bar \eps^{2^k} + \rho^{n-k}) $ where $ v_*(x) $ is a diffeomorphism and
$$ |\, v_*(x_1) - v_*(x_2) | = |\,v_*'(\bar x)\cdot (x_1 - x_2) | \geq C_0 > 0  $$
where $ \bar x $ is in the line segment between $ x_1 $ and $ x_2 $. Thus 
\msk
\begin{equation*}
\begin{aligned}
 \dot x_1 - \dot x_2 
&= \ \alpha_{n,\,k} \big[ \big(x_1 + S^n_k(w_1) \big) - \big(x_2 + S^n_k(w_2)\big) \big] \\
& \qquad 
+ \si_{n,\,k} \big[\, t_{n,\,k}(y_1 - y_2) + \Bu_{n,\,k} \cdot \big \{ \Bz_1 - \Bz_2 + {\bf R}_{n,\,k}(y_1) - {\bf R}_{n,\,k}(y_2) \big \} \big] \\[0.3em]
&= \ \alpha_{n,\,k} \big[\,v_*'(\bar x)\cdot (x_1 - x_2) + O(\bar \eps^{2^k} + \rho^{n-k}) \big] \\
&\qquad + \si_{n,\,k} \big[\,t_{n,\,k}(y_1 - y_2) + \Bu_{n,\,k} \cdot \big \{ \Bz_1 - \Bz_2 + ({\bf R}_{n,\,k})'(\bar y)\cdot (y_1 - y_2) \big \} \big] .
\end{aligned} \bsk
\end{equation*}

\nin Then by Proposition \ref{prop-estimation of t-n-k by b-1} and the estimations in the end of Section \,\ref{sec-recursive formula of Psi-n-k}, we obtain that
\msk 
\begin{equation} \label{lower estimate of x dot distance}
\begin{aligned}
\big| \;\! \dot x_1 - \dot x_2 \big| &= \Big|\: C_3\:\! \si^{2(n-k)} + \si^{n-k} \,\big[\, C_4 \;\! b_1^{2^k} + C_5\;\! \bar \eps^{2^k} \big( \,\bar \eps^{2^n} + \si^{n-k} \bar \eps^{2^k} \big) \,\big] \:\!\Big| \\[0.2em]
&= \Big|\:\si^{2(n-k)}\,\big[\,C_3 + C_4\,\bar \eps^{2^{k+1}} \,\big] + \si^{n-k} \,\big[\, C_4 \;\! b_1^{2^k} + C_5\;\! \bar \eps^{2^k} \bar \eps^{2^n} \,\big] \:\!\Big| \\[0.2em]
& \leq C_5\,\si^{2(n-k)} + C_6\,\si^{n-k} \big[\,b_1^{2^k} + \bar \eps^{2^k} \bar \eps^{2^n} \,\big]
\end{aligned} \msk
\end{equation}
for some constants $ C_3 $, $ C_4 $, and $ C_5 $, which do not have to be positive. Let us take big enough $ n $ which satisfies $ n \geq k + A $ where $ A $ is the constant defined in Proposition \ref{prop-estimation of t-n-k by b-1}. Then $ b_1^{2^k} \gtrsim \bar \eps^{2^k} \bar \eps^{2^n} $ is satisfied. However, by the horizontal overlapping assumption $ \dot x_1 - \dot x_2 = 0 $ for some two points $ x_1 $ and $ x_2 $. Hence,
\begin{equation*}
\begin{aligned}
\si^{2(n-k)} \asymp \si^{n-k} \, b_1^{2^k} .
\end{aligned}
\end{equation*}

\end{proof}

\msk

\subsection{Unbounded geometry on the critical Cantor set}

\comm{************
Recall the coordinate change map at the tip, $ D_k = D\Psi_k(\tau_{F_k}) $. Then \ssk
\begin{align*}
D_k = 
\begin{pmatrix}
1 & t_k & u_k \\
& 1 & \\
& d_k & 1
\end{pmatrix}
 \cdot
\begin{pmatrix}
\alpha_k & & \\
& \si_k & \\
& & \si_k
\end{pmatrix} = DH_k^{-1}(\tau_{F_k}) 
\end{align*}

\msk
****************}

Let us assume that the $ x- $axis overlapping of two boxes, $ \Psi^n_{k,\,{\bf v}}(B^{n+1}_v) $ and $ \Psi^n_{k,\,{\bf v}}(B^{n+1}_c) $. Under this assumption, we can measure the upper bound the minimal distances of two adjacent boxes $ \dist_{\min}(B^{n}_{{\bf w}v}, B^{n}_{{\bf w}c}\,) $, which are the image of $ B^{n+1}_v $ and $ B^{n+1}_c $ under $ \Psi^k_{0,\,{\bf v}} \circ F_k \circ \Psi^n_{k,\,{\bf v}} $ respectively. Compare this minimal distance with the lower bound of the diameter of the one of the above boxes. Then Cantor attractor has the generic unbounded geometry. Moreover, this result is only related to the universal constant $ b_1 $ (Theorem \ref{unbounded geometry with b-1}).

\msk
\begin{lem} \label{upper bound the distance of boxes}
Let the H\'enon-like map $ F $ is in $ {\NN} \cap \II_B(\bar \eps) $. Suppose that $ B^{n}_{{\bf v}v}(R^kF) $ overlaps $ B^{n}_{{\bf v}c}(R^kF) $ 
on the $ x- $axis where the word $ {\bf v} = v^{n-k} \in W^{n-k} $. Then 
$$ \dist_{\min}(B^{n}_{{\bf w}v}, B^{n}_{{\bf w}c}\,) \leq C \big[\, \si^{2k} \si^{n-k} b_1^{2^k} + \si^{2k} \si^{2(n-k)} \bar \eps^{2^k} \,\big] $$
where $ {\bf w} = v^kc\,v^{n-k-1} \in W^n $ for some $ C>0 $ and sufficiently big $ k $ and $ n \geq k + A $.
\end{lem}
\begin{proof}
Recall the expression of the map $ \Psi^n_k $ from $ B(R^nF) $ to $ B^{n-k}_{{\bf v}}(R^kF) $.
\msk
\begin{equation*}
\begin{aligned}
\Psi^n_k(w) = 
\begin{pmatrix}
1 & t_{n,\,k} & \Bu_{n,\,k} \\[0.2em]
& 1 & \\
& \Bd_{n,\,k} & \Id_{m \times m}
\end{pmatrix}
\begin{pmatrix}
\alpha_{n,\,k} & & \\
& \si_{n,\,k} & \\
& & \si_{n,\,k} \cdot \Id_{m \times m}
\end{pmatrix}
\begin{pmatrix}
x + S^n_k(w) \\
y \\[0.2em]
\Bz + {\bf R}_{n,\,k}(y)
\end{pmatrix} .
\end{aligned} \msk
\end{equation*}
where $ {\bf v} = v^{n-k} \in W^{n-k} $. Let us choose two different points as follows
$$ w_1 = (x_1,\,y_1,\,\Bz_1) \in B^1_v(R^nF) \cap \OO_{R^nF}, \quad w_2 = (x_2,\,y_2,\,\Bz_2) \in B^1_c(R^nF) \cap \OO_{R^nF} . $$
Then by the above expression of $ \Psi^n_k $ and the assumption of the overlapping on the $ x- $axis, we may assume the following estimation
\ssk
\begin{equation*}
\begin{aligned}
\dot x_1 - \dot x_2 &= \ 0 \\
\dot y_1 - \dot y_2  &= \ \si_{n,\,k} (y_1 -y_2) \\
\dot \Bz_1  - \dot \Bz_2  &= \ \si_{n,\,k} \big[\, \Bd_{n,\,k}\,(y_1 -y_2) + \Bz_1 - \Bz_2 + {\bf R}_{n,\,k}(y_1) - {\bf R}_{n,\,k}(y_2)\, \big] .
\end{aligned} \msk
\end{equation*}
Moreover, the definitions of $ F_k $ and $ \Psi^k_{0,\,{\bf v}} $ implies that 
$$ \dddot y_1 - \dddot y_2 = \si_{k,\,0}\cdot (\ddot y_1 - \ddot y_2) = \si_{k,\,0}\cdot(\dot x_1 - \dot x_2) =0 . $$
By mean value theorem and the fact that $ (\ddot x_j,\, \ddot y_j,\, \ddot \Bz_j ) = R^kF (\dot x_j,\, \dot y_j,\, \dot \Bz_j ) $ for $ j =1,2 $, we obtain the following equations
\begin{align}
\ddot{x}_1 - \ddot{x}_2 
& = \ f_k(\dot x_1) - \eps_k(\dot w_1) - \big[\,f_k(\dot x_2) - \eps_k(\dot w_2)\,\big]  \label{distance of ddot x-1 and ddot x-2}\\[0.4em]
&= \ - \eps_k(\dot w_1) + \eps_k(\dot w_2) \nonumber \\[0.4em]
&= \ - \di_y \eps_k(\eta) \cdot (\dot y_1 - \dot y_2 ) - E_k(\eta) \cdot (\dot \Bz_1 - \dot \Bz_2 ) \nonumber  \\[0.4em]
&= \ - \di_y \eps_k(\eta)  \cdot  \si_{n,\,k} (y_1 -y_2) \nonumber  \\ & \qquad 
- E_k(\eta) \cdot \si_{n,\,k} \big[\, \Bd_{n,\,k}\,(y_1 -y_2) + \Bz_1 - \Bz_2 + {\bf R}_{n,\,k} (y_1) - {\bf R}_{n,\,k}(y_2)\, \big]  \nonumber \\
%
&= \ -  \di_y \eps_k(\eta)  \cdot  \si_{n,\,k} (y_1 -y_2)  - E_k(\eta) \cdot \si_{n,\,k} \ \sum_{i=k}^{n-1} \Bq_i \circ (\si_{n,\,i}\, \bar y) \cdot (y_1 -y_2)  \nonumber \\[-0.6em]
&\qquad - E_k(\eta) \cdot \si_{n,\,k} (\Bz_1 - \Bz_2) \nonumber  \\
&= \ - \Big[\, \di_y \eps_k(\eta) + E_k(\eta) \cdot \sum_{i=k}^{n-1} \Bq_i \circ (\si_{n,\,i}\, \bar y) \,\Big] \cdot \si_{n,\,k} \, (y_1 -y_2) - E_k(\eta) \cdot \si_{n,\,k} \, (\Bz_1 - \Bz_2) \nonumber  
\end{align} \msk
where $ \eta $ is some point in the line segment between $ \dot{w}_1 $ and $ \dot{w}_2 $ in $ \Psi^n_k(B) $ and $ \bar y $ is in the line segment between $ y_1 $ and $ y_2 $. The second last equation 
is involved with Proposition \ref{formal expression of pi-z Psi difference}. Recall that $ | \;\! y_1 - y_2 | \asymp 1 $ and $ \| \;\! \Bz_1 - \Bz_2 \| = O\big(\bar \eps^{2^n}\big) $ because every point in the critical Cantor set $ \OO_{F_n} $ has its inverse image under $ F_n $. Thus by Lemma \ref{lem-di-y eps and q asymptotic from n to k}, we obtain that
\begin{equation} \label{ddot x distance estimation}
\begin{aligned}
| \, \ddot x_1 - \ddot x_2 | \leq C_1\, \si^{n-k}\big[\, b_1^{2^k} + \oeps^{2^k}\si^{n-k} + \oeps^{2^k} \oeps^{2^n} \,\big] .
\end{aligned}\msk
\end{equation}
Similarly, we have 
\msk
\begin{align}
&\quad \ \ \ddot{\Bz}_1 - \ddot{\Bz}_2  \nonumber \\[0.2em]
&= \ \bde_k(\dot{w}_1) - \bde_k(\dot{w}_{\,2})  \nonumber \\[0.4em]
&= \ Y_k(\zeta) \cdot (\dot y_1 - \dot y_2 ) + Z_k(\zeta) \cdot (\dot \Bz_1 - \dot \Bz_2 ) \nonumber \\[0.4em]
&= \ Y_k(\zeta) \cdot \si_{n,\,k} (y_1 -y_2) 
+ Z_k(\zeta) \cdot \si_{n,\,k} \big[\, \Bd_{n,\,k}(y_1 -y_2) + \Bz_1 - \Bz_2 + {\bf R}_{n,\,k}(y_1) - {\bf R}_{n,\,k}(y_2)\, \big]  \nonumber \\
&= \ Y_k(\zeta) \cdot \si_{n,\,k} \, (y_1 -y_2) + Z_k(\zeta) \cdot \si_{n,\,k} \; \sum_{i=k}^{n-1} \Bq_i \circ (\si_{n,\,i}\, \bar y) \cdot (y_1 -y_2)  \nonumber \\[-0.6em]
&\qquad + Z_k(\zeta) \cdot \si_{n,\,k} \, (\Bz_1 - \Bz_2)  \nonumber \\
&= \ \Big[\, Y_k(\zeta) + Z_k(\zeta) \cdot \sum_{i=k}^{n-1} \Bq_i \circ (\si_{n,\,i}\, \bar y) \,\Big] \cdot \si_{n,\,k} \, (y_1 -y_2) + Z_k(\zeta) \cdot \si_{n,\,k} \, (\Bz_1 - \Bz_2) \label{distance of ddot z-1 and ddot z-2}
\end{align}
where $ \zeta $ is some point in the line segment between $ \dot{w}_1 $ and $ \dot{w}_2 $ in $ \Psi^n_k(B) $. By Corollary \ref{cor-Y-k Z-k and q estimation}, the upper bounds of $ \| \:\!\ddot \Bz_1 - \ddot \Bz_2 \| $ is 
\msk
\begin{equation} \label{ddot z distance estimation}
\begin{aligned}
\|\:\! \ddot \Bz_1 - \ddot \Bz_2 \| \leq C_2\, \si^{n-k} \big[\, \si^{n-k} 
\oeps^{2^k} + \oeps^{2^k} \oeps^{2^n} \,\big] . 
\end{aligned} \msk
\end{equation}
Recall 
\begin{equation*}
\begin{aligned}
\pi_x \circ \Psi^n_k(w) = \alpha_{n,\,k} \big[\,x + S^n_k(w) \,\big] + \si_{n,\,k} \big[\, t_{n,\,k}\,y + \Bu_{n,\,k} \cdot ( \Bz + {\bf R}_{n,\,k}(y)) \,\big] .
\end{aligned} \msk
\end{equation*}
Then the fact that $ \ddot y_1 - \ddot y_2 = 0 $ implies that
\msk
\begin{equation} 
\begin{aligned}
\dddot x_1 - \dddot x_2 &= \ \pi_x \circ \Psi^k_0(\ddot w_1) - \pi_x \circ \Psi^k_0(\ddot w_2) \\[0.3em]
&= \ \alpha_{k,\,0} \big[\,(\ddot x_1 + S^k_0( \ddot w_1)) - (\ddot x_2 + S^k_0( \ddot w_2)) \,\big] \\
&\qquad + \si_{k,\,0}\, \big[\, t_{k,\,0}\,(\ddot y_1 - \ddot y_2 ) + \Bu_{k,\,0} \cdot \big( \ddot \Bz_1 -\ddot \Bz_2 + {\bf R}_{k,\,0} ( \ddot y_1) - {\bf R}_{k,\,0}( \ddot y_2) \big)\, \big] \\[0.3em]
&= \ \alpha_{k,\,0} \big[\, v_*'(\bar x) + O(\bar \eps + \rho^k) \,\big](\ddot x_1 - \ddot x_2) + \si_{k,\,0} \cdot \Bu_{k,\,0}\cdot ( \ddot \Bz_1 -\ddot \Bz_2)
\end{aligned} \msk
\end{equation}
where $ \bar x $ is some point in the line segment between $ \ddot x_1 $ and $ \ddot x_2 $. Moreover,
\msk
\begin{equation}
\begin{aligned}
\dddot \Bz_1 - \dddot \Bz_2 &= \ \pi_z \circ \Psi^k_0(\ddot w_1) - \pi_z \circ \Psi^k_0(\ddot w_2) \\[0.2em]
&= \ \si_{k,\,0}\,(\ddot \Bz_1 - \ddot \Bz_2 ) +  \si_{k,\,0} \big[\, \Bd_{k,\,0}(\ddot y_1 - \ddot y_2) + {\bf R}_{n,\,k}(\ddot y_1) - {\bf R}_{n,\,k}(\ddot y_2)\, \big] \phantom{***} 
\\[0.3em]
&= \ \si_{k,\,0}\,(\ddot \Bz_1 - \ddot \Bz_2 ) .
\end{aligned} \msk
\end{equation}
Let us apply the estimations in \eqref{ddot x distance estimation} and \eqref{ddot z distance estimation} to $ | \:\!\dddot x_1 - \dddot x_2 | $ and $ \| \;\! \dddot \Bz_1 - \dddot \Bz_2 \| $. Then the minimal distance is bounded above as follows
\msk
\begin{equation} \label{intermediate minimal distance between regions}
\begin{aligned}
\dist_{\min}(B^{n}_{{\bf w}v}, B^{n}_{{\bf w}c}\,) &\leq \ | \, \dddot x_1 - \dddot x_2 | + \| \,\dddot \Bz_1 - \dddot \Bz_2 \| \\[0.3em]
&\leq \ \big[\,\si^{2k} \cdot | \,\ddot x_1 - \ddot x_2 | \cdot v_*(\bar x) + \si^k \cdot (1 + \| \:\! \Bu_{k,\,0} \|)  \,\| \;\!\ddot \Bz_1 -\ddot \Bz_2 \|\,\big] (1 + O(\rho^k)) \\[0.3em]
&\leq \ C_3  \si^{2k} \si^{n-k}\,\big[\, b_1^{2^k} + \bar \eps^{2^k} \si^{n-k}+ \bar \eps^{2^k} \bar \eps^{2^n} \,\big] + C_4 \si^k \si^{n-k} \,\big[ \si^{n-k} 
\oeps^{2^k} + \oeps^{2^k} \bar \eps^{2^n} \,\big] \\[0.2em]
& \leq C_5\,\Big[\, \si^{2k} \si^{n-k} b_1^{2^k} + \si^{2k} \si^{2(n-k)} \bar \eps^{2^k} + \si^{2k} \si^{n-k} \bar \eps^{2^k} \bar \eps^{2^n} + \si^k \si^{2(n-k)} \bar \eps^{2^k} \,\Big] \\
& = C_5\,\Big[\,  \si^{2k} \si^{2(n-k)} \big( b_1^{2^k} + \bar \eps^{2^k} \bar \eps^{2^n} \big) + \si^k \si^{2(n-k)} \big( \si^k \bar \eps^{2^k} +  \bar \eps^{2^k} \big)\,\Big]    
\end{aligned} \msk
\end{equation}
for some positive numbers $ C_3 $, $ C_4 $ and $ C_5 $. Moreover, if $ n $ is big enough satisfying $ n \geq k + A $ where $ A $ is the constant defined in Proposition \ref{prop-estimation of t-n-k by b-1}, then we obtain the condition, $ b_1^{2^k} \gtrsim \bar \eps^{2^k} \bar \eps^{2^n} $. Hence, the estimation \eqref{intermediate minimal distance between regions} is refined as follows
\msk
\begin{equation} \label{minimal distance between domains}
\begin{aligned}
\dist_{\min}(B^{n}_{{\bf w}v}, B^{n}_{{\bf w}c}\,) &\leq \ | \, \dddot x_1 - \dddot x_2 | + \| \,\dddot \Bz_1 - \dddot \Bz_2 \| \\
&\leq \ C \big[\, \si^{2k} \si^{n-k} b_1^{2^k} + \si^{k} \si^{2(n-k)} \bar \eps^{2^k} \,\big] 
\end{aligned}
\end{equation}
for some $ C>0 $.
\end{proof}
\msk
\nin Let the box $ B^{n}_{{\bf w}v} $ be $ \Psi^n_{0,\;\!{\bf w}}(B^1_v(R^nF)) $ for the word $ {\bf w} $ of length $ n $.

\begin{lem} \label{lower bounds of the diameter of boxes}
Let $ F \in {\NN} \cap \II_B(\bar \eps) $. Take big enough $ n $ such that $ n \geq k + A $ where $ A $ is the number defined in Proposition \ref{prop-estimation of t-n-k by b-1} depending only on $ \oeps $ and $ b_1 $. Then 
$$  \diam (B^{n}_{{\bf w}v}) \geq | \, C_1 \, \si^k \si^{2(n-k)} - C_2\, \si^k \si^{n-k}b_1^{2^k} | $$
where $ {\bf w} = v^kc\,v^{n-k-1} \in W^n $ for some positive constants $ C_1 $ and $ C_2 $.
\end{lem}

\begin{proof}
Let us choose two points 
$$ w_j = (x_j,\,y_j,\,\Bz_j) \in B^1_v(R^nF) \cap \OO_{R^nF} $$
for $ j =1,2 $ 
satisfying $ | \:\! x_1 - x_2| \asymp 1 $ and $ | \:\! y_1 - y_2| = O(1) $. Thus we may assume that $ \| \:\! \Bz_1 - \Bz_2\| = O(\bar \eps^{2^n}) $ by the equation \eqref{z distance of two points}. Recall that the diameter of the box $ B^{n}_{{\bf w}v} $ is greater than the distance between any two points in $ B^{n}_{{\bf w}v} $. Let $ \dddot w_j = \Psi^k_{0,\,{\bf v}} \circ F_k \circ \Psi^n_{k,\,{\bf v}}(w_j) $, $ \Psi^k_{0,\,{\bf v}}(w_j) = \dot w_j $, and $ F_k(\dot w_j) = \ddot w_j $ for $ j =1,2 $. 
Then
\begin{equation*}
\begin{aligned}
\diam (B^{n+1}_{v}) = \sup\; \{\,\| \;\! w_1 -  w_2 \| \;\big| \ w_1,\;  w_2 \in B^{n+1}_{v} \,\} \asymp 1 . 
\end{aligned} 
\end{equation*} 
We may assume that $ |\;\! x_1 - x_2 | \asymp 1 $ and $ |\;\! y_1 - y_2 | \asymp 1 $ by the appropriate choice of $ w_1 $ and $ w_2 $. The definition of the H\'enon-like map, $ F_k $ and the coordinate change map, $ \Psi^k_{0,\,{\bf v}} $ implies that
\msk
\begin{equation*}
\begin{aligned}
\diam (B^{n}_{{\bf w}v}) & \geq \| \:\!\dddot w_1 - \dddot w_2 \|\geq | \;\!\dddot y_1 - \dddot y_2 | \\[0.2em]
&= | \,\si_{k,\,0}\, (\ddot y_1 - \ddot y_2 )| \\[0.2em]
&= | \,\si_{k,\,0}\, (\dot x_1 - \dot x_2 ) | \\[0.2em]
&= \big| \, \si_{k,\,0} \,\big[\,\pi_x \circ \Psi^n_k(w_1) - \pi_x \circ \Psi^n_k(w_2) \,\big] \:\!\big| 
\end{aligned} \msk
\end{equation*}
for any two points $ \dddot w_1, \dddot w_2 \in B^{n}_{{\bf w}v} $. Recall the equation
\msk
\begin{equation*}
\begin{aligned}
\pi_x \circ \Psi^n_k(w) = \alpha_{n,\,k} \big[\,x + S^n_k(w)\,\big] + \si_{n,\,k} \big[\, t_{n,\,k}\,y + \Bu_{n,\,k} \cdot (\Bz + {\bf R}_{n,\,k} (y)) \,\big] .
\end{aligned} \msk
\end{equation*}
and recall $ x + S^n_k(w) = v_*(x) + O(\bar \eps^{2^k} + \rho^{n-k}) $ for a diffeomorphism $ v_*(x) $. Thus 
$$ |\, v_*(x_1) - v_*(x_2) | = |\,v_*'(\bar x)\cdot (x_1 - x_2) | \geq C_0 > 0  $$
where $ \bar x $ is in the line segment between $ x_1 $ and $ x_2 $. Thus 

\begin{equation*}
\begin{aligned}
 \dot x_1 - \dot x_2 
&= \ \alpha_{n,\,k} \big[ \big(x_1 + S^n_k(w_1) \big) - \big(x_2 + S^n_k(w_2)\big) \big] \\
& \qquad 
+ \si_{n,\,k} \big[\, t_{n,\,k}(y_1 - y_2) + \Bu_{n,\,k} \cdot \big \{ \Bz_1 - \Bz_2 + {\bf R}_{n,\,k}(y_1) - {\bf R}_{n,\,k}(y_2) \big \} \big] \\[0.3em]
&= \ \alpha_{n,\,k} \big[\,v_*'(\bar x)\cdot (x_1 - x_2) + O(\bar \eps^{2^k} + \rho^{n-k}) \big] \\
&\qquad + \si_{n,\,k} \big[\,t_{n,\,k}(y_1 - y_2) + \Bu_{n,\,k} \cdot \big \{ \Bz_1 - \Bz_2 + ({\bf R}_{n,\,k})'(\bar y)\cdot (y_1 - y_2) \big \} \big] 
\end{aligned} \msk
\end{equation*}

\nin Then by Proposition \ref{prop-estimation of t-n-k by b-1} and the estimations in the end of Section \,\ref{sec-recursive formula of Psi-n-k}, we obtain that
\msk 
\begin{equation} \label{lower estimate of x dot distance 2}
\begin{aligned}
\big| \;\! \dot x_1 - \dot x_2 \big| &= \Big|\: C_3\:\! \si^{2(n-k)} + \si^{n-k} \,\big[\, C_4 \;\! b_1^{2^k} + C_5\;\! \bar \eps^{2^k} \big( \,\bar \eps^{2^n} + \si^{n-k} \bar \eps^{2^k} \big) \,\big] \:\!\Big| \\[0.2em]
&= \Big|\:\si^{2(n-k)}\,\big[\,C_3 + C_4\,\bar \eps^{2^{k+1}} \,\big] + \si^{n-k} \,\big[\, C_4 \;\! b_1^{2^k} + C_5\;\! \bar \eps^{2^k} \bar \eps^{2^n} \,\big] \:\!\Big|
\end{aligned} \msk
\end{equation}
for some constants $ C_3 $, $ C_4 $, and $ C_5 $, which do not have to be positive. Let us take big enough $ n $ satisfying $ n \geq k + A $ where $ A $ is defined in Proposition \ref{prop-estimation of t-n-k by b-1}. Thus we obtain that $ b_1^{2^k} \gtrsim \bar \eps^{2^k} \bar \eps^{2^n} $. Hence, 
\begin{equation*}
\begin{aligned}
\diam (B^{n}_{{\bf w}}) \geq \big|\;\!\dddot y_1 - \dddot y_2 \big| \geq \big| \,\si_{k,\,0}\, (\dot x_1 - \dot x_2 ) \big| \geq
\big| \, C_1 \, \si^k \si^{2(n-k)} - C_2\, \si^k \si^{n-k}b_1^{2^k} \big|
\end{aligned} \ssk
\end{equation*}
where $ {\bf w} = v^kc\,v^{n-k-1} \in W^n $.

\end{proof}

\begin{rem}
In the above lemma, we may choose two points $ w_1 $ and $ w_2 $ which maximize $ |\;\!\dot x_1 - \dot x_2| $. Thus we may assume that
$$ \diam (B^{n}_{{\bf w}}) \geq \max \,\{\, C_1 \, \si^k \si^{2(n-k)},\ \ C_2\, \si^k \si^{n-k}b_1^{2^k} \,\} $$
with appropriate positive constants $ C_1 $ and $ C_2 $.
\end{rem} \msk
\nin For the generic geometry of the Cantor attractor, both $ k $ and $ n $ travels through any big natural numbers toward the infinity
. 
Then by the comparison of the diameter of the box and minimal distance between adjacent boxes, $ \OO_F $ has the unbounded geometry.
\msk

\begin{prop} 
Let $ F_{b_1} $ is an element of parametrized family for $ b_1 \in [0,1] $ in $ {\NN} \cap \II_B(\bar \eps) $. If $ b_1^{2^k} \asymp \si^{n-k} $ for infinitely many $ k $ and $ n $, then there exists $ b_1 $ for $ F_{b_1} $ such that $ B^n_{{\bf v}v}(R^kF) $ overlaps $ B^n_{{\bf v}c}(R^kF) $ on the $ x- $axis where the word $ {\bf v} = v^{n-k} \in W^{n-k} $. Furthermore $ F_{b_1} $ has no bounded geometry on $ \OO_{F_{b_1}} $.
\end{prop}

\begin{proof}
Let us choose the two points
\begin{equation*}
\begin{aligned}
w_1 = (x_1,\,y_1,\,\Bz_1) \in B^1_v(R^nF) \cap \OO_{R^nF}, \quad w_2 = (x_2,\,y_2,\,\Bz_2) \in B^1_c(R^nF) \cap \OO_{R^nF}
\end{aligned}
\end{equation*}
such that $ | \:\! x_1 - x_2 | \asymp 1 $, $ | \:\! y_1 - y_2 | \asymp 1 $. Recall $ \| \:\! \Bz_1 - \Bz_2 \| = O(\bar \eps^{2^n}) $. Let $ \dot w_j = (\dot x_j,\, \dot y_j,\, \dot \Bz_j) $ be $ \Psi^n_k(w_j) $ for $ j =1,2 $. Thus
\begin{equation} \label{x coordinate of the image under Psi n-k again}
\begin{aligned}
 \dot x_1 - \dot x_2 
&= \ \alpha_{n,\,k} \Big[ \big(x_1 + S^n_k(w_1) \big) - \big(x_2 + S^n_k(w_2)\big) \Big] \\[0.2em]
& 
+ \si_{n,\,k} \Big[\, t_{n,\,k}(y_1 - y_2) + \Bu_{n,\,k} \cdot \big \{ \Bz_1 - \Bz_2 + {\bf R}_{n,\,k}(y_1) - {\bf R}_{n,\,k}(y_2) \big \} \Big] .
\end{aligned} \ssk
\end{equation}
Recall that $ \alpha_{n,\,k} = \si^{2(n-k)}(1 + O(\rho^k)) $,\, $ \si_{n,\,k} = (-\si)^{n-k}(1 + O(\rho^k)) $ and $ x + S^n_k(w) = v_*(x) + O(\bar \eps^{2^k} + \rho^{n-k}) $. Since $ v_* $ is a diffeomorphism and $ | \;\! x_1 - x_2 | \asymp 1 $, $ | \, v_*(x_1) - v_*(x_2) | \asymp 1 $ by mean value theorem. 
Moreover, Proposition \ref{prop-estimation of t-n-k by b-1} implies that
$$ b_1^{2^k} \asymp t_{n,\,k} . $$
In addition to the above estimation, the fact that $ \| ({\bf R}_{n,\,k})' \| = O\big(\si^{n-k} \bar \eps^{2^k} \big) $ and the estimation in \eqref{lower estimate of x dot distance} implies that 
\msk
\begin{equation}
\begin{aligned}
\big| \, \Bu_{n,\,k}\cdot \big\{ \Bz_1 - \Bz_2 + {\bf R}_{n,\,k}(y_1) - {\bf R}_{n,\,k}(y_2) \,\big\}\,\big| &\leq \big\| \, \Bu_{n,\,k}\cdot (\Bz_1 - \Bz_2) \big\| + \big\| \:\!({\bf R}_{n,\,k})'(\bar y)\cdot (y_1 - y_2) \:\!\big\| \\[0.2em]
&= O\big(\,\bar \eps^{2^k} \bar \eps^{2^n} \big) + O\big( \si^{n-k}\bar \eps^{2^k}\big) .
\end{aligned} \msk
\end{equation}
\nin If $ n \geq k + A $ where $ A $ is the constant defined in Proposition \ref{prop-estimation of t-n-k by b-1}, then we express the equation \eqref{x coordinate of the image under Psi n-k again} as follows \ssk
\begin{equation*}
\begin{aligned}
\dot x_1 - \dot x_2 = \si^{2(n-k)} \big[\, v_*(x_1) - v_*(x_2) \,\big] \cdot \big[\, 1 + r_{n,\,k}\,b_1^{2^k} (-\si)^{-(n-k)} \big](1 + O(\rho^k))
\end{aligned} \ssk
\end{equation*}
where $ r_{n,\,k} $ depends uniformly on $ b_1 $. 
\comm{*********************
Let $ r \leq r_{n,\,k} \leq \frac{1}{r} $. \ssk Let us take any number $ b_1^{-} $ in the parameter space $ (0, \bar b_1) $ and any natural number $ k \geq N $ for some big enough $ N $. 
\nin Then we can find the biggest number $ n $ such that $ n-k $ is odd and $ \si^{n-k} > \dfrac{1}{r}\;  (b_1^{-})^{2^k} $, that is,
$$ 1 + r_{n,\,k}\cdot(b_1^{-})^{2^k}(-\si)^{-(n-k)} \geq 1 + \dfrac{1}{r}\;(b_1^{-})^{2^k}(-\si)^{-(n-k)}>0 $$
Let us increase the parameter from $ b_1^- $ to $ b_1^+ $ such that $ (b_1^{+})^{2^k} = \dfrac{2}{r}\;\si^{(n-k)} $. Then
$$ 1 + r_{n,\,k}\cdot(b_1^{+})^{2^k}(-\si)^{-(n-k)} \leq 1 + r \cdot \dfrac{2}{r}\;(-1) = -1 < 0 $$

\nin Then there exists $ b_1 \in (b_1^-, b_1^+) $ such that \,$ \dot x_1 - \dot x_2 = 0 $, \ssk that is, $ \Psi^n_k(B^1_v(R^nF)) $ and $ \Psi^n_k(B^1_c(R^nF)) $ overlaps over the $ x- $axis with respect to $ \dot w_1 $ and $ \dot w_2 $. 
For all big enough $ k $, $ b_1 \asymp b_1^- $. Thus $ \log(b_1 / b_1^-) = O(2^{-k}) $. \ssk Then $ b_1 $ converges to $ b_1^- $ as $ k \ra \infty $. Then we obtain the dense subset of the parameter, \ssk $ (0, \bar b_1) $ on which $ \Psi^n_k(B^1_v(R^nF)) $ and $ \Psi^n_k(B^1_c(R^nF)) $ overlaps over the $ x- $axis. Moreover, there exists open subset, $ J_m $ of parameter $ (0, \bar b_1) $ for each fixed level $ k \geq m $. Then $ \cap_{m} J_m $ is a $ G_{\de} $ subset of $ (0, \bar b_1) $. 
********************************}
\ssk \\
\nin Let us compare the distance of two adjacent boxes and the diameter of the box for every big $ k + A < n $. Let us take $ n $ such that $ \si^{n-k} \asymp b_1^{2^k} $. We may assume \ssk that $ B^{n-k}_{{\bf v}v}(R^kF) $ overlaps $ B^{n-k}_{{\bf v}c}(R^kF) $ on the $ x- $axis where $ {\bf v} = v^{n-k-1} \in W^{n-k-1} $. By Lemma \ref{lower bounds of the diameter of boxes} and Lemma \ref{upper bound the distance of boxes}, \msk
\begin{equation*}
\begin{aligned}
 \diam (B^{n}_{{\bf w}v}) &\geq \big| \, C_1 \, \si^k \si^{2(n-k)} - C_2\, \si^k \si^{n-k}b_1^{2^k} \big|  \\
 \dist_{\min}(B^{n}_{{\bf w}v}, B^{n}_{{\bf w}c}) &\leq C_0 \big[\, \si^{2k} \si^{n-k} b_1^{2^k} + \si^{k} \si^{2(n-k)} \oeps^{2^k} \,\big]
\end{aligned} \msk
\end{equation*}
where $ {\bf w} = v^kc\,v^{n-k-1} \in W^n $ for some numbers $ C_0 > 0 $ and $ C_1 $ and $ C_2 $. 
Hence, 
\begin{equation*}
\dist_{\min}(B^{n}_{{\bf w}v}, B^{n}_{{\bf w}c}) \leq C \,\si^k  \diam (B^{n}_{{\bf w}v}) \msk
\end{equation*}
for every sufficiently large $ k \in \N $ and for some $ C>0 $. Then the critical Cantor set has unbounded geometry.
\end{proof}

Overlapping with $ \si^{n-k} \asymp b_1^{2^k} $ is valid almost everywhere with respect to Lebesgue measure in \cite{HLM}.

\begin{thm}[\cite{HLM}] Given any $ 0 < A_0 < A_1,\ 0 < \si < 1 $ and any $ p \geq 2 $, the set of parameters $ b \in [0,1] $ for which there are infinitely many $ 0<k<n $ satisfying 
\begin{equation*}
A_0 < \frac{b^{p^k}}{\si^{n-k}} < A_1
\end{equation*}
is a dense $ G_{\de} $ set with full Lebesgue measure. 
\end{thm}
\nin Unbounded geometry is almost everywhere property in the parameter set of $ b_1 $ for every fixed  $ b_2 $. By Corollary \ref{cor-comparison sigma n-k with b-1}, the condition of the overlapping of two adjacent boxes, $ B^{n-k}_{{\bf v}v}(R^kF) $ and $ B^{n-k}_{{\bf v}c}(R^kF) $ on the $ x- $axis implies that 
$$ \si^{n-k} \asymp b_1^{2^k} $$
for infinitely many $ k $ and $ n $. Then
\begin{thm} \label{unbounded geometry with b-1}
Let $ F_{b_1} $ be an element of parametrized space in $ \NN \cap \II_B(\oeps) $ with $ b_1 = b_F/b_{\Bz} $. Then there exists a small interval $ [0, b_{\bullet}] $ for which there exists a $ G_{\de} $ subset $ S \subset [0, b_{\bullet}] $ with full Lebesgue measure such that the critical Cantor set, $ \OO_{F_{b_1}} $  has unbounded geometry for all $ b_1 \in S $. 
\end{thm}

\bsk

\section{Non rigidity on the critical Cantor set}

Let $ F $ and $ \widetilde F $ be H\'enon-like maps in $ \NN \cap \II_B(\bar \eps) $. Let the universal number $ b_1 $ and $ \widetilde b_1 $ are for the map $ F $ and $ \widetilde F $. Non rigidity on the Cantor set with respect to the universal constant $ b_1 $ means that the homeomorphism between critical Cantor sets, $ \OO_F $ and $ \OO_{\widetilde F} $ is at most $ \alpha- $H\"older continuous with a constant $ \alpha < 1 $ (Theorem \ref{Non rigidity with b-1} below). This kind of non rigidity phenomenon is a generalization of two dimensional non rigidity theorem in \cite{CLM}. However, non rigidity theorem in three or higher dimension only depends essentially on the contracting rate $ b_1 $ from two dimensional H\'enon-like map in higher dimension.

\subsection{Bounds of the distance between two points}
Let us consider the box 
$$ B^n_{\bf w} = \Psi^k_0 \circ F_k \circ \Psi^n_k (B) $$ 
where $ B = B(R^nF) $. Since $ \diam B(R^nF) \asymp \diam B^1_v(R^nF) $, 
by Lemma \ref{lower bounds of the diameter of boxes} we have the estimation of the lower bound of $ \diam B(R^nF) $ as follows
\msk
\begin{equation} \label{lower bound of the distance of two points}
\begin{aligned}
\diam (B^{n}_{{\bf w}}) \geq \big| \, C_1 \, \si^k \si^{2(n-k)} - C_2\, \si^k \si^{n-k}b_1^{2^k} \big|
\end{aligned} \msk
\end{equation}
where $ {\bf w} = v^kc\,v^{n-k-1} \in W^n $ for some constants $ C_1 $ and $ C_2 $. Let us estimate the upper bound of the distance. 

\begin{lem} \label{upper bound of the distance of two points}
Let $ F \in {\NN} \cap \II_B(\bar \eps) $. 
Then 
$$  \diam (B^{n}_{{\bf w}}) \leq C \, \big[\,\si^k \si^{2(n-k)} + \si^k \si^{n-k} b_1^{2^k} \,\big] $$
where $ {\bf w} = v^kc\,v^{n-k-1} \in W^n $ for some $ C>0 $.
\end{lem}

\begin{proof}
Recall the map $ \Psi^n_k $ from $ B(R^nF) $ to $ B^{n}_{{\bf v}}(R^kF) $.\msk
\begin{align*}
\Psi^n_k(w) = 
\begin{pmatrix}
1 & t_{n,\,k} & \Bu_{n,\,k} \\[0.2em]
& 1 & \\
& \Bd_{n,\,k} & \Id_{m \times m}
\end{pmatrix}
\begin{pmatrix}
\alpha_{n,\,k} & & \\
& \si_{n,\,k} & \\
& & \si_{n,\,k} \cdot \Id_{m \times m}
\end{pmatrix}
\begin{pmatrix}
x + S^n_k(w) \\
y \\[0.2em]
z + {\bf R}_{n,\,k}(y)     \msk
\end{pmatrix} .
\end{align*}
\\
where $ {\bf v} = v^{n-k} \in W^{n-k} $. Let us \ssk choose the two points 
$$ w_1 = (x_1,\,y_1,\,\Bz_1) \in B^1_v(R^nF) \cap \OO_{R^nF}, \quad w_2 = (x_2,\,y_2,\,\Bz_2) \in B^1_c(R^nF) \cap \OO_{R^nF}. $$
Recall $ \dot w_j = \Psi^n_k(w_j) $, $ \ddot w_j = F_k(\dot w_j) $ and $ \dddot w_j = \Psi^k_0(\ddot w_j) $ for $ j =1,2 $. 
\ssk Observe that $ | \:\! x_1 - x_2 | $ and $ | \:\! y_1 - y_2| $ is $ O(1) $. We may assume that $ \| \:\! \Bz_1 - \Bz_2 \| = O(\bar \eps^{2^n}) $ because $ \OO_{R^nF} $ is a completely invariant set under $ R^nF $. 
By Corollary \ref{cor-comparison sigma n-k with b-1} and the equation \eqref{lower estimate of x dot distance}, we have \msk
\begin{equation} \label{distance of x coordinate under Psi-n-k}
\begin{aligned}
\dot x_1 - \dot x_2 
&= \ \alpha_{n,\,k} \big[ \big(x_1 + S^n_k(w_1) \big) - \big(x_2 + S^n_k(w_2)\big) \big] \\
&\qquad + \si_{n,\,k} \big[\, t_{n,\,k}(y_1 - y_2) + \Bu_{n,\,k} \cdot \big \{ \Bz_1 - \Bz_2 + {\bf R}_{n,\,k}(y_1) - {\bf R}_{n,\,k}(y_2) \big \} \big] \\[0.4em]
&= \ \alpha_{n,\,k} \big[\,v_*'(\bar x) + O(\bar \eps^{2^k} + \rho^{n-k}) \big] (x_1 - x_2) \\
&\qquad + \si_{n,\,k} \big[\; t_{n,\,k}(y_1 - y_2) + \Bu_{n,\,k} \cdot \big \{ \Bz_1 - \Bz_2 + {\bf R}_{n,\,k}(y_1) - {\bf R}_{n,\,k}(y_2) \big \} \big] \\[0.4em]
&\leq \ C\, \big[\, \si^{2(n-k)} + \si^{n-k} b_1^{2^k} + \si^{2(n-k)}\bar \eps^{2^k}\,\big]
\end{aligned} \msk
\end{equation}
for some $ C>0 $. Moreover, \ssk
\begin{equation*} 
\begin{aligned}
\dot y_1 - \dot y_2 &= \ \si_{n,\,k} (y_1 - y_2) \\[0.3em]
\dot \Bz_1 - \dot \Bz_2 &= \ \si_{n,\,k} \big[\, \Bd_{n,\,k}(y_1 -y_2) + \Bz_1 - \Bz_2 + {\bf R}_{n,\,k}(y_1) - {\bf R}_{n,\,k}(y_2)\, \big] .
\phantom{*****}
\end{aligned} \msk
\end{equation*}

\nin By the equation \eqref{distance of ddot x-1 and ddot x-2}, we estimate the distance between each coordinates of $ F_k(\dot w_1) $ and $ F_k(\dot w_2) $ as follows \ssk
\begin{align*}
\ddot{x}_1 - \ddot{x}_2 &= \ f_k(\dot x_1) - \eps_k(\dot w_1) - [f_k(\dot x_2) - \eps_k(\dot w_2)] \\[0.2em]
&= \ f'_k(\bar x)\cdot(\dot x_1 - \dot x_2) - \eps_k(\dot w_1 ) + \eps_k(\dot w_2)\\[0.3em]
&= \ [\,f'_k(\bar x) - \di_x \eps_k(\eta) \,]\cdot(\dot x_1 - \dot x_2)- \di_y \eps_k(\eta) \cdot (\dot{y_1} - \dot{y_2}) - E_k(\eta) \cdot (\dot \Bz_1 - \dot \Bz_2 ) \\[0.4em]
&= \ [\,f_k'(\bar x) - \di_x \eps_k(\eta) \,]\cdot(\dot x_1 - \dot x_2) - \di_y \eps_k(\eta)  \cdot  \si_{n,\,k} (y_1 -y_2) \\[0.2em]
& \qquad - E_k(\eta) \cdot \si_{n,\,k} \big[\, \Bd_{n,\,k}(y_1 -y_2) + \Bz_1 - \Bz_2 + {\bf R}_{n,\,k}(y_1) - {\bf R}_{n,\,k}(y_2)\, \big] \\[0.4em]
\ddot{y}_1 - \ddot{y}_2 &= \ \dot x_1 - \dot x_2  \\[0.5em]
\ddot{\Bz}_1 - \ddot{\Bz}_2 &= \ \bde_k(\dot{w}_1) - \bde_k(\dot{w}_2) \\[0.3em]
&= \ X_k(\zeta) \cdot(\dot x_1 - \dot x_2) + Y_k(\zeta) \cdot (\dot y_1 - \dot y_2 ) + Z_k(\zeta) \cdot (\dot \Bz_1 - \dot \Bz_2 ) \\[0.3em]
&= \ X_k(\zeta) \cdot(\dot x_1 - \dot x_2) + Y_k(\zeta) \cdot \si_{n,\,k} (y_1 -y_2) \\
& \qquad + Z_k(\zeta) \cdot \si_{n,\,k} \big[\, \Bd_{n,\,k}(y_1 -y_2) + \Bz_1 - \Bz_2 + {\bf R}_{n,\,k}(y_1) - {\bf R}_{n,\,k}(y_2)\, \big]
\end{align*} 
where $ \eta $ and $ \zeta $ are some points in the line segment between $ \dot{w}_1 $ and $ \dot{w}_2 $ in $ \Psi^n_k(B) $. The equations \eqref{distance of ddot x-1 and ddot x-2} and \eqref{ddot x distance estimation} in Lemma \ref{upper bound the distance of boxes} implies that
\begin{equation} \label{ddot x distance upper bound}
\begin{aligned}
| \, \ddot x_1 - \ddot x_2 | \leq |\,f_k'(\bar x) - \di_x \eps_k(\eta) \,| \cdot | \, \dot x_1 - \dot x_2 | + C_2 \, \si^{n-k}\big[\, b_1^{2^k} + \bar \eps^{2^k}\si^{n-k} + \bar \eps^{2^k}\bar \eps^{2^n} \,\big] 
\end{aligned} \msk
\end{equation}
and the equations \eqref{distance of ddot z-1 and ddot z-2} and \eqref{ddot z distance estimation} in the same Lemma implies that
\msk
\begin{equation} \label{ddot z distance upper bound}
\begin{aligned}
\| \, \ddot \Bz_1 - \ddot \Bz_2 \| \leq \|\,X_k(\zeta) \| \cdot | \, \dot x_1 - \dot x_2 | + C_3\,\si^{n-k} \big[\, \si^{n-k} \bar \eps^{2^k} + \bar \eps^{2^k} \bar \eps^{2^n} \,\big] .
\end{aligned} \msk
\end{equation}
Then the difference of each coordinates of $ \Psi^k_0(\ddot w_1) $ and $ \Psi^k_0(\ddot w_2) $ as follows
\msk
\begin{equation} 
\begin{aligned}
\dddot x_1 - \dddot x_2 &= \ \pi_x \circ \Psi^k_0(\ddot w_1) - \pi_x \circ \Psi^k_0(\ddot w_2) \\[0.3em]
&= \ \alpha_{k,\,0} \big[\,(\ddot x_1 + S^k_0( \ddot w_1)) - (\ddot x_2 + S^k_0( \ddot w_2)) \,\big] \\
&\qquad + \si_{k,\,0}\, \big[\, t_{k,\,0}\,(\ddot y_1 - \ddot y_2 ) + \Bu_{k,\,0} \cdot \big( \ddot \Bz_1 -\ddot \Bz_2 + {\bf R}_{k,\,0}( \ddot y_1) - {\bf R}_{k,\,0}( \ddot y_2) \big)\, \big] \\[0.3em]
&= \ \alpha_{k,\,0} \big[\, v_*'(\bar x) + O(\bar \eps + \rho^k) \,\big](\ddot x_1 - \ddot x_2) + \si_{k,\,0} \cdot \Bu_{k,\,0}\cdot ( \ddot \Bz_1 -\ddot \Bz_2) \\
&\qquad + \si_{k,\,0}\, \big[\, t_{k,\,0}\,(\dot x_1 - \dot x_2 ) + \Bu_{k,\,0} \cdot \big( {\bf R}_{k,\,0}(\dot x_1) - {\bf R}_{k,\,0}( \dot x_2) \big)\, \big]
\end{aligned}
\end{equation}
\begin{equation}
\begin{aligned}
\dddot y_1 - \dddot y_2 &= \ \si_{k,\,0}\,(\ddot y_1 - \ddot y_2 ) = \si_{k,\,0}\,(\dot x_1 - \dot x_2 ) \\[0.4em] 
\dddot \Bz_1 - \dddot \Bz_2 &= \ \pi_z \circ \Psi^k_0(\ddot w_1) - \pi_z \circ \Psi^k_0(\ddot w_2) \\[0.2em]
&= \ \si_{k,\,0}\,(\ddot \Bz_1 - \ddot \Bz_2 ) +  \si_{k,\,0} \big[\, \Bd_{k,\,0}(\ddot y_1 - \ddot y_2) + {\bf R}_{n,\,k}(\ddot y_1) - {\bf R}_{n,\,k}(\ddot y_2)\, \big] 
\\[0.2em]
&= \ \si_{k,\,0}\,(\ddot \Bz_1 - \ddot \Bz_2 ) + \si_{k,\,0} \big[\, \Bd_{k,\,0}(\dot x_1 - \dot x_2) + {\bf R}_{n,\,k}(\dot x_1) - {\bf R}_{n,\,k}(\dot x_2)\, \big] . \phantom{**\;\,}
\end{aligned} \msk
\end{equation}
Let us calculate a upper bound of the distance, $ \|\, \dddot w_1 - \dddot w_2 \| $. Applying the estimation \eqref{ddot x distance upper bound} and \eqref{ddot z distance upper bound}, we obtain that \msk
\begin{equation*}
\begin{aligned}
& \quad \ \ \|\, \dddot w_1 - \dddot w_2 \| \leq \ |\,\dddot x_1 - \dddot x_2 | + |\,\dddot y_1 - \dddot y_2 | + \|\,\dddot \Bz_1 - \dddot \Bz_2 \| \\[0.4em]
& \leq \ \big|\, \alpha_{k,\,0} \big[\, v_*'(\bar x) + O(\bar \eps + \rho^k) \,\big](\ddot x_1 - \ddot x_2) + \si_{k,\,0} \cdot \Bu_{k,\,0}\cdot ( \ddot \Bz_1 -\ddot \Bz_2) \\[0.2em]
&\qquad + \si_{k,\,0}\, \big[\, t_{k,\,0}\,(\dot x_1 - \dot x_2 ) + \Bu_{k,\,0} \cdot \big( {\bf R}_{k,\,0}(\dot x_1) - {\bf R}_{k,\,0}( \dot x_2) \big)\, \big] \:\!\big| 
\\[0.2em]
&\qquad + \big|\,\si_{k,\,0}\,(\dot x_1 - \dot x_2 ) \,\big| + \big\|\,\si_{k,\,0}\,(\ddot \Bz_1 - \ddot \Bz_2 ) + \si_{k,\,0} \big[\, \Bd_{k,\,0}(\dot x_1 - \dot x_2) + {\bf R}_{n,\,k}(\dot x_1) - {\bf R}_{n,\,k}(\dot x_2)\, \big] \:\!\big\| \phantom{***\;}
\end{aligned} \ssk
\end{equation*}

\begin{equation*}
\begin{aligned}
& \leq \ \big|\,\alpha_{k,\,0} \big[\, v_*'(\bar x) + O(\bar \eps + \rho^k) \,\big]\cdot |\,f'(\bar x) - \di_x \eps_k(\eta) \,| \cdot | \, \dot x_1 - \dot x_2 | \\[0.2em]
&\qquad + \big|\,\alpha_{k,\,0} \big[\, v_*'(\bar x) + O(\bar \eps + \rho^k) \,\big]\cdot C_2 \, \si^{n-k}\big[\, b_1^{2^k} + \bar \eps^{2^k}\si^{n-k} + \bar \eps^{2^k}\bar \eps^{2^n} \,\big]\:\!\big| \\[0.2em] 
&\qquad + \big|\,\si_{k,\,0}\, [\,1 + |\;\!t_{k,\,0}| + \| \;\!\Bd_{k,\,0} \| \,]\,(\dot x_1 - \dot x_2 ) \,\big| 
+ \big|\,\si_{k,\,0}\, [\,1 + \| \:\!\Bu_{k,\,0} \| \,] \:\! \big| \cdot \big\| \:\!({\bf R}_{n,\,k})'(\widetilde x) \cdot (\dot x_1 - \dot x_2) \big\| \\[0.2em]
&\qquad + \big|\,\si_{k,\,0}\, [\,1 + \| \:\!\Bu_{k,\,0} \| \,]\:\!\big| \cdot \big[\,\|\, X_k(\zeta) \| \cdot 
\big| \dot x_1 - \dot x_2 \big| + C_3\,\si^{n-k} \big( \si^n \bar \eps^{2^k} + \bar \eps^{2^k} \bar \eps^{2^n} \big)\:\!\big]  
\end{aligned} \ssk
\end{equation*}
After factoring out $ |\;\! \dot x_1 - \dot x_2 | $\,, the inequality continues as follows \ssk
\begin{equation*}
\begin{aligned}
& \leq \ C_5 \,\si^k |\,\dot x_1 - \dot x_2 | + C_5 \, \si^{2k}\si^{n-k}\big[\, b_1^{2^k} + \bar \eps^{2^k}\bar \eps^{2^n} \,\big] + C_6\,\si^k\si^{n-k} \big[\, \si^{n-k} \bar \eps^{2^k} + \bar \eps^{2^k} \bar \eps^{2^n} \:\!\big] \hspace{0.93in} \\[0.2em]
& \leq \ C_7\,\si^k \big[\, \si^{2(n-k)} + \si^{n-k} ( b_1^{2^k} + \bar \eps^{2^k} \bar \eps^{2^n} ) \,\big]  \\[0.2em]
& \qquad + C_5 \, \si^{2k}\si^{n-k}\big[\, b_1^{2^k} + \bar \eps^{2^k}\si^{n-k} + \bar \eps^{2^k}\bar \eps^{2^n} \,\big] + C_6\,\si^k\si^{n-k} \big[\, \si^{n-k}  \bar \eps^{2^k} +  \bar \eps^{2^k} \bar \eps^{2^n} \:\!\big] 
\end{aligned} \ssk
\end{equation*}
for some positive constants, $ C_j $, $ 2 \leq j \leq 7 $ which are independent of $ k $ and $ n $. The second last line holds by the estimation in \eqref{distance of x coordinate under Psi-n-k} and $ \| \:\!({\bf R}_{n,\,k})'\| $ in Proposition \ref{exponential smallness of R-n-k}. 
 Then the above estimation continues
\msk
\begin{equation*}
\begin{aligned}
 & \leq \ C_7\,\si^k \big[\, \si^{2(n-k)} + \si^{n-k} ( b_1^{2^k} + \bar \eps^{2^k} \bar \eps^{2^n} ) \,\big]  
+ C_5 \, \si^{2k}\si^{n-k}\big[\, b_1^{2^k} + \bar \eps^{2^k}\si^{n-k} + \bar \eps^{2^k}\bar \eps^{2^n} \,\big]  \hspace{0.93in}\\
 & \qquad + C_8\,\si^k\si^{2(n-k)} \bar \eps^{2^k} \\[0.3em]
 & \leq \ \big( C_7 + C_5\,\si^k\bar \eps^{2^k} + C_8\, \bar \eps^{2^k} \big)\, \si^k \si^{2(n-k)} + \big(C_7 + C_5\,\si^k)\, \si^k\si^{n-k}b_1^{2^k} 
 \\ & \qquad 
+ \big(C_7 + C_5\,\si^k \big)\, \si^k \si^{n-k} \bar \eps^{2^k}\bar \eps^{2^n}
\end{aligned} \bsk
\end{equation*}
for some positive constant $ C_8 $. Moreover, if $ n \geq k +A $ where $ A $ is the number defined in Proposition \ref{prop-estimation of t-n-k by b-1}, then 
$$ b_1^{2^k} \gtrsim \bar \eps^{2^k}\bar \eps^{2^n} . $$
Hence,
\begin{equation*}
\diam (B^{n}_{{\bf w}}) \leq C \, \big[\,\si^k \si^{2(n-k)} + \si^k \si^{n-k} b_1^{2^k} \,\big]
\end{equation*}
where $ {\bf w} = v^kc\,v^{n-k-1} \in W^n $ for some $ C>0 $.
\end{proof}
\msk
\begin{rem}
Lemma \ref{lower bounds of the diameter of boxes} and Lemma \ref{upper bound of the distance of two points} implies the lower and upper bounds of $ \diam B_{\bf w} $ where $ B_{\bf w} = \Psi^k_0 \circ F_k \circ \Psi^n_k (B(R^nF)) $ as follows
\begin{equation*}
C_0\,\si^k | \;\! \dot x_1 - \dot x_2 | \leq \diam B_{\bf w} \leq C_1 \,\si^k |\;\! \dot x_1 - \dot x_2 |
\end{equation*}
for every big enough $ k \in \N $, that is, \,$ \diam B_{\bf w} \asymp \si^k |\;\! \dot x_1 - \dot x_2 | $.
\end{rem}

\msk

\subsection{Non rigidity on the Cantor set with respect to $ b_1 $}

\begin{thm} \label{Non rigidity with b-1}
Let H\'enon-like maps $ F $ and $ \widetilde{F} $ be in $ \NN \cap \II_B(\bar \eps) $. 
Let $ b_1 $ be the ratio $ b_F/b_{\Bz} $ where $ b_F $ is the average Jacobian and $ b_{\Bz} $ for $ F $ is the number defined in Proposition \ref{prop-universal number b-Z}. The number $ \widetilde{b}_1 $ is defined by the similar way for the map $ \widetilde{F} $. Let $ \phi \colon \OO_{\widetilde{F}} \ra \OO_F $ be a homeomorphism \ssk which conjugate $ F_{\,\OO_F} $ and $ \widetilde{F}_{\,\OO_{\widetilde{F}}} $ and $ \phi(\tau_{\widetilde{F}}) = \tau_{F} $. If \,$ b_1 > \widetilde{b}_1 $, \ssk then the H\"older exponent of $ \phi $ is not greater than $ \displaystyle{\frac{1}{2} \left(1 + \frac{\log b_1}{ \log \widetilde{b}_1} \right)} $.
\end{thm}

\begin{proof}
Let two points $ w_1 $ and $ w_2 $ be in $ B^1_v(R^nF) $ and $ B^1_c(R^nF) $ respectively. Similarly, assume that $ {\widetilde w}_1 $ and $ \widetilde{w}_2 $ are in $ B(R^n\widetilde{F}) $. Let us define
$ \dddot w_j = \Psi^k_0 \circ F_k \circ \Psi^n_k (w_j) $ for $ j=1,2 $. The points $ \dddot {\widetilde {w}}_1 $ and $ \dddot {\widetilde {w}}_2 $ are defined by the similar way. For sufficiently large $ k \in \N $, let us choose $ n $ depending on $ k $ which satisfies the following inequality
\begin{equation*}
\si^{n-k+1} \leq \widetilde{b}_1^{2^k} < \si^{n-k} .
\end{equation*}
Observe that $ b_1^{2^k} \gg \widetilde{b}_1^{2^k} $.
By Lemma \ref{lower bounds of the diameter of boxes} and Lemma \ref{upper bound of the distance of two points}, we have the following inequalities
\msk
\begin{equation*}
\begin{aligned}
\dist(\dddot {\widetilde {w}}_1, \dddot{\widetilde{w}}_2) &\leq C_0 \, \Big[\,\si^k \si^{2(n-k)} + \si^k \si^{n-k} \,\widetilde{b}_1^{2^k} \,\Big] \leq C_1 \,\si^k \,\widetilde{b}_1^{2^k}\, \widetilde{b}_1^{2^k}  \\[0.2em]
\dist(\dddot w_1, \dddot w_2) &\geq \big| \, C_2 \, \si^k \si^{2(n-k)} - C_3\, \si^k \si^{n-k}{b}_1^{2^k} \big| \geq C_4 \, \si^k \,\widetilde{b}_1^{2^k}\, b_1^{2^k}
\end{aligned} \msk
\end{equation*}
for some positive $ C_j $ where $ j =0,1,2,3 $ and $ 4 $. 
\nin H\"older continuous function, $ \phi $ with the H\"older exponent $ \alpha $ has to satisfy 
\begin{equation*}
\dist(\dddot w_1, \dddot w_2) \leq C \, \big( \dist(\dddot{\widetilde{w}}_1, \dddot{\widetilde{w}}_2) \big)^{\alpha}
\end{equation*}
for some $ C>0 $. Then we see that
\begin{equation*}
\si^k \,\widetilde{b}_1^{2^k}\, {b}_1^{2^k} \leq C \, \Big( \si^k \,\widetilde{b}_1^{2^k}\, \widetilde b_1^{2^k} \Big)^{\alpha}
\end{equation*}
\nin Take the logarithm both sides and divide them by $ 2^k $. After passing the limit, divide both sides by the negative number, $ 2\log \widetilde{b}_1 $. Then the desired upper bound of the H\"older exponent is obtained
\begin{equation*}
\begin{aligned}
k \log \si + 2^k \log \widetilde{b}_1 + 2^k \log {b}_1 &\leq \log C + \alpha \:\! \big( k \log \si + 2^k \log \widetilde{b}_1 + 2^k \log \widetilde b_1 \big) \\
\frac{k}{2^k} \log \si + \log \widetilde{b}_1 + \log {b}_1 &\leq \frac{1}{2^k} \log C + \alpha \left( \frac{k}{2^k} \log \si + \log \widetilde{b}_1 + \log \widetilde b_1 \right) \\
\log \widetilde{b}_1 + \log {b}_1 &\leq \alpha \cdot 2\log \widetilde b_1 \\
\alpha &\leq \frac{1}{2} \left(1 + \frac{\log b_1}{ \log \widetilde{b}_1} \right) .
\end{aligned}
\end{equation*}
\end{proof} \bsk

\nin 
The best possible regularity of homeomorphic conjugation between two critical Cantor set is unknown. In particular, the possible biggest H\"older exponent of $ \phi $ where $ F $ and $ \widetilde{F} $ are two dimensional H\'enon-like map and $ b_F = b_{\widetilde{F}} $. However, the average Jacobian of higher dimensional H\'enon-like map in $ \NN \cap \II_B(\bar \eps) $ less affects the non rigidity than the number $ b_1 $. In other words, in higher dimension we may less expect the rigidity between two different Cantor attractors.

\msk

\appendix

\section{Recursive formula of $ D\bde_1 $}

The pre-renormalization of H\'enon-like map,
$$ PRF = H \circ F^2 \circ H^{-1} $$
where 
$ H^{-1}(w) = (\phi^{-1}(w),\,y,\,\Bz + \bde(y,f^{-1}(y),\B0)) $. Let $ \text{Pre}\,\bde_1 $ be $ \pi_{\Bz} \circ PRF $. The renormalization of $ F $, $ RF $ is $ \La \circ PRF \circ \La^{-1} $. The third coordinate of $ RF $, $ \bde_1(w) $ is $ \displaystyle{ \frac{1}{\si_0}\; \text{Pre}\,\bde_1 \:\!(\si_0\;\! w) } $. Recall that $ \bde(w) = (\de^1(w),\; \de^2(w),\ldots,\;\de^m(w)) $. Thus each partial derivatives of $ j^{th} $ coordinate function $ \de^j(w) $ is each partial derivatives of $ \text{Pre}\,\de_1^j $ at $ \si_0\:\! w $. 
\ssk
\begin{lem} \label{lem-recursive formula of delta}
Let $ F $ be the $ m+2 $ dimensional renormalizable H\'enon-like map. Let $ \pi_{\Bz} \circ F  = \bde $ and $ \pi_{\Bz} \circ RF  = \bde_1 $. Then
\begin{align*}
\di_x \de^j_1(w) &= \Big[\, \di_y \de^j \circ \psi^1_c(w) + \sum_{l=1}^m \di_{z_l} \de^j \circ \psi^1_c(w) \cdot \di_x \de^l \circ \psi^1_v(w) \,\Big] \cdot \di_x \phi^{-1}( \si_0 w) \\[-0.3em]
& \quad \ + \di_x \de^j \circ \psi^1_c(w) - \frac{d}{dx}\,\de^j (\si_0 x,\,f^{-1}(\si_0 x), {\bf 0}) \\[0.4em]
\di_y \de^j_1(w) &= \Big[\, \di_y \de^j \circ \psi^1_c(w) + \sum_{l=1}^m \di_{z_l} \de^j \circ \psi^1_c(w) \cdot \di_x \de^l \circ \psi^1_v(w) \,\Big] \cdot \di_y \phi^{-1}( \si_0 w) \\[-0.4em]
& \quad \ + \sum_{l=1}^m \di_{z_l} \de^j \circ \psi^1_c(w) \cdot \Big[\, \di_y \de^j \circ \psi^1_v(w) + \sum_{i=1}^m \di_{z_i} \de^l \circ \psi^1_v(w) \cdot \frac{d}{dy}\,\de^i(\si_0y,f^{-1}(\si_0 y), {\bf 0}) \,\Big] \\[0.3em]
\di_{z_i} \de^j_1(w) &= \Big[\, \di_y \de^j \circ \psi^1_c(w) + \sum_{l=1}^m \di_{z_l} \de^j \circ \psi^1_c(w) \cdot \di_x \de^l \circ \psi^1_v(w) \,\Big] \cdot \di_{z_i} \phi^{-1}( \si_0 w) \\[-0.4em]
& \quad \ + \sum_{l=1}^m \di_{z_l} \de^j \circ \psi^1_c(w) \cdot \di_{z_i} \de^l \circ \psi^1_v(w)
\end{align*} \msk
for $ 1 \leq j \leq m $ and $ 1 \leq i \leq m $.
\end{lem}
\begin{proof}
The expression of $ \text{Pre}\,\bde_1 $ is as follows
\ssk
\begin{align*}
\text{Pre}\,\bde_1\:\!(w) & = \bde \circ F \circ H^{-1}(w) - \bde(x,f^{-1}(x),\B0) \\[0.2em]
& = \bde(x,\,\phi^{-1}(w),\,\bde \circ H^{-1}(w)) - \bde(x,f^{-1}(x),\B0) .
\end{align*} 
Recall $ \bde(w) = (\de^1(w),\; \de^2(w),\ldots,\;\de^m(w)) $. The $ j^{th} $ coordinate function of $ \de_1^j(w) $ or $ \text{Pre}\,\de_1^j(w) $ is defined similarly for $ 1 \leq j \leq m $. Let us calculate $ \di_x \text{Pre}\,\de_1^j(w) $ 
\ssk
\begin{align}
&\quad \ \ \di_x \text{Pre}\,\de_1^j(w) \nonumber \\[0.2em]
& = \ \frac{\di}{\di x}\; \big(\,\de^j \circ F \circ H^{-1}(w)\big) - \frac{d}{dx}\; \de^j(x,f^{-1}(x),\B0)  \nonumber \\[0.4em]
& = \ \frac{\di}{\di x}\; \de^j (x,\,\phi^{-1}(w),\, \bde \circ H(w)) - \frac{d}{dx}\; \de^j(x,f^{-1}(x),\B0)  \nonumber \\[0.6em]
& = \ \di_x \de^j \circ (F \circ H^{-1}(w)) + \di_y \de^j \circ (F \circ H^{-1}(w)) \cdot \di_x \phi^{-1}(w)  \nonumber \\
&\qquad + \sum_{l=1}^m \di_{z_l} \de^j \circ (F \circ H^{-1}(w)) \cdot \di_x (\, \de^l \circ H^{-1}(w)) - \frac{d}{dx}\; \de^j(x,f^{-1}(x),\B0)  \nonumber \\[0.6em]
& = \ \di_x \de^j \circ (F \circ H^{-1}(w)) + \di_y \de^j \circ (F \circ H^{-1}(w)) \cdot \di_x \phi^{-1}(w)  \nonumber \\
&\qquad + \sum_{l=1}^m \di_{z_l} \de^j \circ (F \circ H^{-1}(w)) \cdot \di_x \de^l \circ H^{-1}(w) \cdot \di_x \phi^{-1}(w) - \frac{d}{dx}\; \de^j(x,f^{-1}(x),\B0)  \nonumber \\[0.6em]
& = \left[\, \di_y \de^j \circ (F \circ H^{-1}(w)) + \sum_{l=1}^m \di_{z_l} \de^j \circ (F \circ H^{-1}(w)) \cdot \di_x \de^l \circ H^{-1}(w) \,\right] \cdot \di_x \phi^{-1}(w) \label{eq-app-recursive for di-x de-1} \\
& \qquad + \di_x \de^j \circ (F \circ H^{-1}(w)) - \frac{d}{dx}\; \de^j(x,f^{-1}(x),\B0) . \nonumber 
\end{align} \msk
Recall that $ \psi^1_v(w) = H^{-1}(\si_0\:\!w) $ and $ \psi^1_c(w) = F \circ H^{-1}(\si_0\:\!w) $. Since $ \di_x \de^j_1(w) = \di_x \text{Pre}\,\de_1^j( \si_0\:\!w) $, equation \eqref{eq-app-recursive for di-x de-1} implies that \ssk
\begin{equation*}
\begin{aligned}
\di_x \de^j_1(w) &= \Big[\, \di_y \de^j \circ \psi^1_c(w) + \sum_{l=1}^m \di_{z_l} \de^j \circ \psi^1_c(w) \cdot \di_x \de^l \circ \psi^1_v(w) \,\Big] \cdot \di_x \phi^{-1}( \si_0 w) \\[-0.3em]
& \quad \ + \di_x \de^j \circ \psi^1_c(w) - \frac{d}{dx}\,\de^j (\si_0 x,\,f^{-1}(\si_0 x), {\bf 0}) \
\end{aligned} \ssk
\end{equation*}
for each $ 1 \leq j \leq m $. Let us calculate $ \di_y \text{Pre}\,\de_1^j(w) $ \msk
\begin{align*}
&\quad \ \ \di_y \text{Pre}\,\de_1^j(w) \\[0.3em]
& = \ \frac{\di}{\di y}\; \big(\,\de^j \circ F \circ H^{-1}(w)\big) - \frac{d}{dy}\; \de^j(x,f^{-1}(x),\B0) \\[0.3em]
& = \ \frac{\di}{\di y}\; \de^j (x,\,\phi^{-1}(w),\, \bde \circ H(w)) \\
& = \ \di_y \de^j \circ (F \circ H^{-1}(w)) \cdot \di_y \phi^{-1}(w) 
+ \sum_{l=1}^m \di_{z_l} \de^j \circ (F \circ H^{-1}(w)) \cdot \di_y (\, \de^l \circ H^{-1}(w)) \\[0.4em]
& = \ \di_y \de^j \circ (F \circ H^{-1}(w)) \cdot \di_y \phi^{-1}(w) 
+ \sum_{l=1}^m \di_{z_l} \de^j \circ (F \circ H^{-1}(w)) \\[-0.4em]
& \quad \ \cdot \Big[\, \di_x \de^j \circ H^{-1}(w) \cdot \di_y \phi^{-1}(w) + \di_y \de^j \circ H^{-1}(w) + \sum_{i=1}^m \di_{z_i} \de^l \circ H^{-1}(w) \cdot \frac{d}{dy}\;\de^i(y,f^{-1}(y),\B0) \,\Big] \\[0.2em]
& = \ \Big[\, \di_y \de^j \circ (F \circ H^{-1}(w)) + \sum_{l=1}^m \di_{z_l} \de^j \circ (F \circ H^{-1}(w)) \cdot \di_x \de^l \circ H^{-1}(w) \,\Big] \cdot \di_y \phi^{-1}(w) \\[-0.4em]
& \quad \ + \sum_{l=1}^m \di_{z_l} \de^j \circ (F \circ H^{-1}(w)) \cdot \Big[\, \di_y \de^j \circ H^{-1}(w) + \sum_{i=1}^m \di_{z_i} \de^l \circ H^{-1}(w) \cdot \frac{d}{dy}\;\de^i(y,f^{-1}(y),\B0) \,\Big] .
\end{align*} 
Then
\begin{align*}
\di_y \de^j_1(w) &= \Big[\, \di_y \de^j \circ \psi^1_c(w) + \sum_{l=1}^m \di_{z_l} \de^j \circ \psi^1_c(w) \cdot \di_x \de^l \circ \psi^1_v(w) \,\Big] \cdot \di_y \phi^{-1}( \si_0 w) \\[-0.4em]
& \quad \ + \sum_{l=1}^m \di_{z_l} \de^j \circ \psi^1_c(w) \cdot \Big[\, \di_y \de^j \circ \psi^1_v(w) + \sum_{i=1}^m \di_{z_i} \de^l \circ \psi^1_v(w) \cdot \frac{d}{dy}\,\de^i(\si_0y,f^{-1}(\si_0 y), {\bf 0}) \,\Big]
\end{align*} \msk
for each $ 1 \leq j \leq m $. Let us calculate $ \di_{z_i} \text{Pre}\,\de_1^j(w) $ for $ 1 \leq i \leq m $
\begin{align*}
&\quad \ \ \di_{z_i} \text{Pre}\,\de_1^j(w) \\[0.3em]
& = \ \frac{\di}{\di z_i}\; \big(\,\de^j \circ F \circ H^{-1}(w)\big) - \frac{d}{d z_i}\; \de^j(x,f^{-1}(x),\B0) \\[0.3em]
& = \ \frac{\di}{\di z_i}\; \de^j (x,\,\phi^{-1}(w),\, \bde \circ H(w)) \\
& = \ \di_y \de^j \circ (F \circ H^{-1}(w)) \cdot \di_{z_i} \phi^{-1}(w) 
+ \sum_{l=1}^m \di_{z_l} \de^j \circ (F \circ H^{-1}(w)) \cdot \di_{z_i} (\, \de^l \circ H^{-1}(w)) \\[0.4em]
& = \ \di_y \de^j \circ (F \circ H^{-1}(w)) \cdot \di_{z_i} \phi^{-1}(w)  \\ & \qquad 
+ \sum_{l=1}^m \di_{z_l} \de^j \circ (F \circ H^{-1}(w)) 
\cdot \Big[\, \di_x \de^l \circ H^{-1}(w) \cdot \di_{z_i} \phi^{-1}(w) + \di_{z_i} \de^l \circ H^{-1}(w) \,\Big] \\[0.2em]
& = \ \Big[\, \di_y \de^j \circ (F \circ H^{-1}(w)) + \sum_{l=1}^m \di_{z_l} \de^j \circ (F \circ H^{-1}(w)) \cdot \di_x \de^l \circ H^{-1}(w) \,\Big] \cdot \di_{z_i} \phi^{-1}(w) \\[-0.4em]
& \qquad  + \sum_{l=1}^m \di_{z_l} \de^j \circ (F \circ H^{-1}(w)) \cdot \sum_{i=1}^m \di_{z_i} \de^l \circ H^{-1}(w) .
\end{align*} %
\begin{align*}
\di_{z_i} \de^j_1(w) &= \Big[\, \di_y \de^j \circ \psi^1_c(w) + \sum_{l=1}^m \di_{z_l} \de^j \circ \psi^1_c(w) \cdot \di_x \de^l \circ \psi^1_v(w) \,\Big] \cdot \di_{z_i} \phi^{-1}( \si_0 w) \\[-0.4em]
& \quad \ + \sum_{l=1}^m \di_{z_l} \de^j \circ \psi^1_c(w) \cdot \di_{z_i} \de^l \circ \psi^1_v(w)
\end{align*}
for each $ 1 \leq i,\,j \leq m $. Hence, the proof is complete.
\end{proof}

\ssk

\section{Critical point of infinitely renormalizable H\'enon-like map}
The H\'enon-like map $ F $ has the tip if it is infinitely renormalizable. Let $ \tau_k $ be the tip of $ R^kF $ for $ k \in \N $. Define the {\em critical point} of $ R^kF $, $ c_{F_k} $ as $ (R^kF)^{-1}(\tau_k) $. For $ k<n $, define $ \Psi^n_{k,\,\vv} $ and $ \Psi^n_{k,\,\cc} $ as follows
\begin{equation}
\begin{aligned}
\Psi^n_{k,\,\vv} &= \psi^{k+1}_v \circ \psi^{k+2}_v \circ \cdots \circ \psi^n_v \\[0.2em]
\Psi^n_{k,\,\cc} &= \psi^{k+1}_c \circ \psi^{k+2}_c \circ \cdots \circ \psi^n_c .
\end{aligned} \msk
\end{equation}
\nin Recall the equation between renormalized map
$$ F_{j+1} = \big( \psi^{j+1}_v \big)^{-1} \circ F^2_j \circ \psi^{j+1}_v $$
for $ k \leq j < n $. 
\begin{lem} \label{lem-relation between Psi-c and Psi-v}
Let $ F $ be the infinitely renormalizable H\'enon-like map. Then
$$ \Psi^n_{k,\,\vv} \circ F_n = F_k \circ \Psi^n_{k,\,\cc} $$
for $ k < n $.
\end{lem}
\begin{proof}
Recall the equation $ \psi^{j+1}_c = F_j \circ \psi^{j+1}_v $ for $ k \leq j < n $. Thus 
\msk
\begin{align*}
F_k \circ \Psi^n_{k,\,\cc} 
& = F_k \circ \psi^{k+1}_c \circ \psi^{k+2}_c \circ \cdots \circ \psi^n_c \\[0.2em]
& = \big[\,F_k \circ (F_k \circ \psi^{k+1}_v)\,\big] \circ (F_{k+1} \circ \psi^{k+2}_v) \circ \cdots \circ (F_{n-1} \cdot \psi^n_v) \\[0.2em]
& = \psi^{k+1}_v \circ \psi^{k+2}_v \circ F_{k+2} \circ (F_{k+2} \circ \psi^{k+3}_v) \circ \cdots \circ (F_{n-1} \circ \psi^n_v) \\
&  \hspace{1in} \vdots \\[0.2em]
& = \psi^{k+1}_v \circ \psi^{k+2}_v \circ \cdots \circ \psi^n_v \circ F_n \\[0.2em]
& = \Psi^n_{k,\,\vv} \circ F_n
\end{align*}
\end{proof}

\begin{cor} \label{cor-image of critical point under Psi}
Let $ F $ be the infinitely renormalizable H\'enon-like map. Let $ c_{F_j} $ is the critical point of $ F_j $ for $ j \in \N $. Then
$$ c_{F_k} = \Psi^n_{k,\,\cc}(c_{F_n}) $$
for $ k < n $. Moreover, 
$$ \{c_{F_k} \} = \bigcap_{i=k}^{\infty} \Psi^i_{k,\,\cc}(B) . $$
\end{cor}
\begin{proof}
Lemma \ref{lem-relation between Psi-c and Psi-v} implies
$$ \Psi^n_{k,\,\vv} \circ F_n(w) = F_k \circ \Psi^n_{k,\,\cc}(w) $$
for $ k<n $. Then
\begin{equation*}
\begin{aligned}
\Psi^n_{k,\,\cc}(c_{F_n}) &= (F_k)^{-1} \circ \Psi^n_{k,\,\vv} \circ F_n(c_{F_n}) \\[0.2em]
&= (F_k)^{-1} \circ \Psi^n_{k,\,\vv} (\tau_n) \\[0.2em]
&= (F_k)^{-1}(\tau_k) \\
&= c_{F_k}
\end{aligned} \ssk
\end{equation*}
By the above equation, $ c_{F_k} \in \Psi^i_{k,\,\cc}(B) $ for every $ i > k $. Furthermore, 
$$ \Psi^{k+1}_{k,\,\cc}(B) \supset \Psi^{k+2}_{k,\,\cc}(B) \supset \cdots $$ 
is the nested sequence by the definition of $ \Psi^n_{k,\,\cc} $. Since $ \diam\, (\Psi^n_{k,\,\cc}) \leq C\si^{n-k} $ where $ \si $ is the universal scaling factor for $ k < n $ and for some $ C>0 $. Then 
$$ \lim_{ i \ra \infty} \Psi^i_{k,\,\cc}(B) = \{ c_{F_k} \}. $$
Hence, the intersection of $ \Psi^i_{k,\,\cc}(B) $ for all $ i > k $ is the set of critical point of $ F_k $.
\end{proof}



\end{document}